\documentclass[10pt,reqno]{amsart}
\usepackage{graphicx,enumerate,amssymb,bbold}
\usepackage[notrig]{physics}
\usepackage[font=footnotesize,labelfont=small]{subcaption}
\usepackage[ruled,vlined]{algorithm2e}
\usepackage[left=3.0cm,right=3.0cm,top=2.5cm,bottom=2.5cm,includeheadfoot]{geometry}
\usepackage{tikz-cd}

\usepackage[colorlinks=true, pdfstartview=FitV, linkcolor=blue,
            citecolor=blue, urlcolor=blue]{hyperref}

\usepackage[margin=10pt,font=small]{caption}

\numberwithin{equation}{section}
\numberwithin{figure}{section}
\theoremstyle{plain}
\newtheorem{theorem}{Theorem}[section]
\newtheorem{lemma}[theorem]{Lemma}
\newtheorem{proposition}[theorem]{Proposition}

\theoremstyle{definition}
\newtheorem{definition}[theorem]{Definition}

\newtheorem{remark}[theorem]{Remark}

\newcommand{\bitem}{\begin{itemize}}
\newcommand{\eitem}{\end{itemize}}
\newcommand{\mc}[1]{\mathcal{#1}}

\newcommand{\N}{\mathbb{N}}
\newcommand{\R}{\mathbb{R}}

\newcommand{\bpm}{\begin{pmatrix}}
\newcommand{\epm}{\end{pmatrix}}
\newcommand{\bvm}{\begin{vmatrix}}
\newcommand{\evm}{\end{vmatrix}}
\newcommand{\bsm}{\left(\begin{smallmatrix}}
\newcommand{\esm}{\end{smallmatrix}\right)}
\newcommand{\T}{\top}

\newcommand{\la}{\langle}
\newcommand{\ra}{\rangle}

\newcommand{\mrm}[1]{\mathrm{#1}}

\newcommand{\veps}{\varepsilon}

\newcommand{\vphi}{\varphi}

\newcommand{\eins}{\mathbb{1}}

\newcommand{\LG}[1]{\mathrm{#1}}
\newcommand{\Lg}[1]{\mathfrak{#1}}

\DeclareMathSymbol{\mydiv}{\mathbin}{symbols}{"04}

\DeclareMathOperator{\Diag}{Diag}

\DeclareMathOperator{\Span}{span}

\DeclareMathOperator{\dexp}{dexp}

\DeclareMathOperator{\Exp}{Exp}
\DeclareMathOperator{\expm}{expm}

\title[Geometric Numerical Integration of the Assignment Flow]{Geometric Numerical Integration of the Assignment Flow}
\author[A.~Zeilmann, F.~Savarino, S.~Petra, C.~Schn\"{o}rr]{Alexander Zeilmann, Fabrizio Savarino, Stefania Petra, Christoph Schn\"{o}rr}
\address[A.~Zeilmann]{Image and Pattern Analysis Group, Heidelberg University, Germany}
\email{alexander.zeilmann@iwr.uni-heidelberg.de}
\address[F.~Savarino]{Image and Pattern Analysis Group, Heidelberg University, Germany}
\email{fabrizio.savarino@iwr.uni-heidelberg.de}
\address[S.~Petra]{Mathematical Imaging Group, Heidelberg University, Germany}
\email{petra@math.uni-heidelberg.de}
\urladdr{\url{https://www.stpetra.com}}
\address[C.~Schn\"{o}rr]{Image and Pattern Analysis Group, Heidelberg University, Germany}
\email{schnoerr@math.uni-heidelberg.de}
\urladdr{\url{https://ipa.math.uni-heidelberg.de}}
\date{}

\keywords{image labeling, assignment flow, assignment manifold, geometric integration, adaptive step size selection, Krylov subspace approximation}

\begin{document}

\begin{abstract}
The \textit{assignment flow} is a smooth dynamical system that evolves on an elementary statistical manifold and performs contextual data labeling on a graph. We derive and introduce the \textit{linear assignment flow} that evolves nonlinearly on the manifold, but is governed by a linear ODE on the tangent space. Various numerical schemes adapted to the mathematical structure of these two models are designed and studied, for the geometric numerical integration of  both flows: embedded Runge-Kutta-Munthe-Kaas schemes for the nonlinear flow, adaptive Runge-Kutta schemes and exponential integrators for the linear flow. All algorithms are parameter free, except for setting a tolerance value that specifies adaptive step size selection by monitoring the local integration error, or fixing the dimension of the Krylov subspace approximation. These algorithms provide a basis for applying the assignment flow to machine learning scenarios beyond supervised labeling, including unsupervised labeling and learning from controlled assignment flows.
\end{abstract}

\maketitle
\vspace{-0.5cm}
\tableofcontents

\newpage
\section{Introduction}
\label{sec:Introduction}

\textbf{Overview, motivation.} The \textit{assignment flow}, recently introduced by \cite{Astrom:2017ac} and detailed in Section \ref{sec:AssignmentFlow}, denotes a \textit{smooth dynamical system} evolving on an elementary statistical manifold, for the contextual classification of a finite set of data given on an arbitrary graph. 
Vertices of the graph are associated with the elements of the data set and correspond to locations in space and/or time in typical applications. \textit{Classification} means to assign each datum exactly to one class representative, called \textit{label}, out of a finite set of predetermined labels. \textit{Contextual} classification means that these decision directly depend on each other, as encoded by the edges (adjacency relation) of the underlying graph. In the context of image analysis, classifying given image data on a regular grid graph in this way is called the \textit{image labeling problem}. We point out, however, that the assignment flow applies to arbitrary data represented on a graph.

A key property of the assignment flow is that decision variables do \textit{not} live in the space used to model the data. Rather, a probability simplex is associated with each datum, on which a flow evolves until it converges to one of the vertices of the simplex that encode the labels. Each simplex is equipped with the Fisher-Rao metric which turns the relative interior of the simplex into a smooth Riemannian manifold. It is this particular geometry that effectively promotes discrete decisions that interact in a \textit{smooth} way. Replacing in addition the Riemannian (Levi Civita) connection by the $\alpha$-connection (with $\alpha=1$) introduced by Amari and Chentsov \cite{Amari:2000aa}, enables to carry out basic geometric operations in a computationally \textit{efficient} way. Keeping  the assignment flow as `inference engine' separate from the data space and model allows to flexibly apply it to a broad range of contextual data classification problems. We refer to \cite{Amari:2000aa,Ay:2017aa} as basic texts on information geometry and to \cite{Kappes:2015aa} for more information on the image labeling problem. 

From a more distant viewpoint, our work ties in with the recent trend to explore the mathematics of deep networks from a dynamical systems perspective \cite{E:2017aa}. A frequently cited paper in this respect is \cite{He:2016aa} where a connection was made between the so-called residual architecture of networks and explicit Euler integration steps of a corresponding system of nonlinear \textit{ordinary differential equations (ODEs)}. We refer to \cite{Haber:2017aa} for a good exposition. While this offers a novel and fresh perspective on the \textit{learning problem} of network parameters, it does not alter the basic ingredients of such networks that apparently have been adopted in an ad-hoc way, like parametrized static layers connected by nonlinear transition functions, ReLU activations etc. 

By contrast, the assignment flow provides a smooth dynamical system on a graph (network), where all ingredients coherently fit into the overall mathematical framework. Based on this, we recently showed how discrete graphical models for image labeling can be evaluated using the assignment flow \cite{Huhnerbein:2018aa}, and how \textit{unsupervised} labeling can be modeled by coupling the assignment flow and Riemannian gradient flows for label evolution on feature manifolds \cite{Zern:2018ac}. Our current work, to be reported elsewhere, studies machine learning problems based on \textit{controlling} the assignment flow. Here, in particular, \textit{algorithms} play a decisive role that accurately \textit{integrate the assignment flow numerically on the manifold} where it evolves. A thorough study of such algorithms is the subject of the present paper.

\vspace{0.2cm}
\noindent
\textbf{Contribution, organization.} This paper presents two interrelated contributions, as illustrated by Figure \ref{fig:paperStructure}. 
\begin{enumerate}[(1)]
\item 
We derive from the assignment flow -- henceforth called \textit{nonlinear assignment flow} -- the \textit{linear assignment flow}, that still is nonlinear but governed by a \textit{linear} ODE on the tangent space. This property is attractive for modeling tasks (e.g.~parameter estimation and control) as well as for the design of numerical algorithms. In particular, our experiments show that the linear flow closely approximates the nonlinear flow, as far as concerns the final \textit{labeling} results.
\item
We study a range of algorithms for numerically integrating both the nonlinear and the linear assignment flow, respectively, while properly taking into account the underlying \textit{geometry}.
\begin{enumerate}
\item Regarding the \textit{nonlinear} assignment flow, we adopt the machinery of Lie group methods for the numerical integration of ODEs on manifolds \cite{Iserles:2005aa} and devise corresponding extensions of classical Runge-Kutta schemes, called \textit{RKMK schemes} (Runge-Kutta-Munthe-Kaas) in the literature \cite{Munthe-Kaas:1999aa}. We combine pairs of these extensions to form \textit{embedded} RKMK schemes for adaptive step size control, analogous to classical embedded Runge-Kutta (RK) schemes \cite{Hairer:2008aa}.
\item
Regarding the \textit{linear} assignment flow, we take advantage in two alternative ways of the \textit{linearity} of the flow on the tangent space.
\begin{enumerate}
\item On the one hand, we derive a local error estimate in order to apply classical Runge-Kutta schemes \cite{Hairer:2008aa} on the tangent space, with step sizes that adapt automatically.
\item
On the other hand, we evaluate the integral representation of the linear flow, due to Duhamels formula, and approximately evaluate this integral using Krylov subspace methods, as has been developed in the literature on exponential integrators \cite{Saad:1992aa,Hochbruck:1997aa,Hochbruck:2010aa}.
\end{enumerate}
\end{enumerate}
\end{enumerate}
All these \textit{explicit} numerical schemes are evaluated and discussed in Section \ref{sec:Experiments}, using `ground truth' flows as a baseline that were computed using the \textit{implicit} geometric Euler scheme with a sufficiently small step size. All algorithms are parameter free, except for specifying a single tolerance value with respect to the local error, that governs adaptive step size selection. Our experiments indicate a value for this parameter that `works' regarding integration accuracy and labeling quality, but is not too conservative (i.e.~small). In the case of the exponential integrator, we merely have to supply the final point of time $T$ at which the linear assignment flow should be evaluated, in addition to the dimension of the Krylov subspace which controls the quality of the approximation. We conclude with a synopsis of our results in Section \ref{sec:Conclusion}.

\begin{figure}
\centerline{
\includegraphics[width=0.5\textwidth]{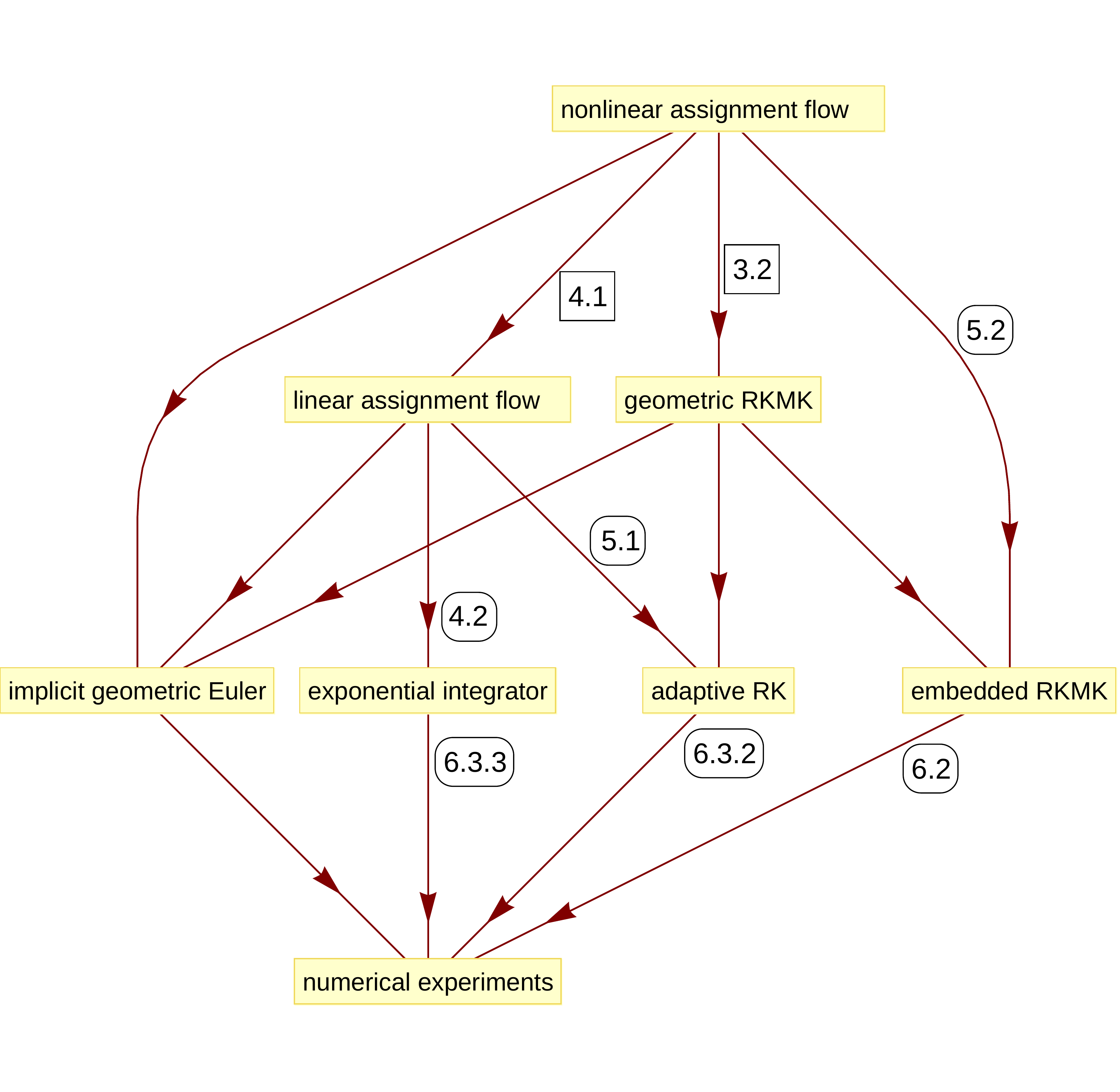}
}
\caption{
Topics addressed in this paper and their interrelations. Edge labels refer to the corresponding sections. Section numbers framed by squares address modeling aspects, whereas those framed by rounded squares address the design of algorithms and their numerical evaluation. Unlabelled edges mean `is derived from' or `provides the basis for'.
}
\label{fig:paperStructure}
\end{figure}

\vspace{0.25cm}
\noindent
\textbf{Basic notation.} Index sets $I$ and $J$ index vertices $i \in I$ of the underlying graph and labels $j \in J$, respectively. $\mc{S}$ and $\mc{W}$ denote the basic statistical manifolds that we work with, defined in Section \ref{sec:AssignmentFlow}. Points $p,q \in \mc{S}$ are strictly positive probability vectors, and we denote efficiently by
\[
q p = (q_{1} \cdot p_{1},\dotsc, q_{|J|} \cdot p_{|J|})^{\T},
\qquad\qquad
\frac{q}{p} = \Big(\frac{q_{1}}{p_{1}},\dotsc,\frac{q_{|J|}}{p_{|J|}}\Big)
\]
componentwise multiplication for general vectors, and componentwise subdivision only if $p \in \mc{S}$. Likewise, functions like the exponential function and the logarithm with vectors as arguments apply componentwise,
\[
e^{v} = (e^{v_{1}},e^{v_{2}},\dotsc)^{\T},
\qquad\qquad
\log v = (\log v_{1},\log v_{2},\dotsc)^{\T}.
\]
$\Exp$ and $\exp$ denote exponential mappings defined in Section \ref{sec:AssignmentFlow}, whereas $\expm$ denotes the \textit{matrix} exponential in Section \ref{sec:Exponential-Integrator}. The ordinary exponential function defined on the real line $\R$ is always denoted by $e^{x},\,x \in \R$.

$\eins=(1,\dotsc,1)^{\T}$ denotes the constant $1$-vector with appropriate number of components depending on the context. We use the common shorthand $[n]=\{1,2,\dotsc,n\}$ with $n \in \N$. $\|\cdot\|$ denotes the $\ell_{2}$-norm and $\|\cdot\|_{p}$ the $\ell_{p}$-norm if $p \neq 2$.

\section{The Assignment Flow}
\label{sec:AssignmentFlow}

We summarize the assignment flow introduced by \cite{Astrom:2017ac} and related concepts required in this paper. 
Let $G=(I,E)$ be a given undirected graph and let
\begin{subequations}\label{eq:def-FI}
\begin{align}
\mc{F}_{I} 
&= \big\{f_{i} \colon i \in I\big\} \subset \mc{F}
\intertext{
be data given in a metric space 
}
&(\mc{F},d). 
\end{align}
\end{subequations}
We call $\mc{F}_{I}$ \textit{image data} even though the $f_{i}$ typically represent \textit{features} extracted from raw image at pixel $i \in I$ as a preprocessing step. $G$ may be a regular grid graph as in low-level image processing or a less structured graph, with arbitrary connectivity in terms of the neighborhoods
\begin{equation}
\mc{N}_{i} = \{k \in I \colon ik=ki \in E\}.
\end{equation}
We associate with each neighborhood $\mc{N}_{i}$ weights satisfying
\begin{equation}\label{eq:def-weights}
w_{ik} > 0,\qquad
\sum_{k \in \mc{N}_{i}} w_{ik} = 1,\qquad \forall i \in I.
\end{equation}
These weights parametrize the regularization property of the assignment flow and are assumed to be given. How to learn them from data in order to control the assignment flow will be reported elsewhere.

Along with $\mc{F}_{I}$ we assume prototypical data 
\begin{equation}\label{eq:def-GJ}
\mc{G}_{J} = \big\{ g_{j} \in \mc{F} \colon j \in J\big\}
\end{equation}
to be given, henceforth called \textit{labels}. Each label $g_{j}$ represents the data of class $j$. \textit{Image labeling} denotes the problem to assign class labels to image data depending on the local context encoded by the graph $G$. We refer to \cite{Huhnerbein:2018aa} for more details and background on the image labeling problem.

Assignment of labels to data are represented by discrete probability distributions
\begin{equation}\label{eq:def-Wi}
W_{i} = (W_{i1},\dotsc,W_{i|J|})^{\T} \in \mc{S},\quad i \in I,
\end{equation}
where 
\begin{equation}\label{eq:def-S}
\mc{S} = \big\{p \in \R^{|J|} \colon p_{j}>0,\, j \in J,\; \la\eins,p \ra=1\big\}
\end{equation}
denotes the relatively open probability simplex equipped with the Fisher-Rao metric
\begin{equation}\label{eq:def-gp-T0}
g_{p}(u,v) = \sum_{j \in J} \frac{u_{j} v_{j}}{p_{j}},\qquad
u, v \in T_{0} = \{p \in \R^{|J|} \colon \la \eins, p \ra = 0\},\qquad p \in \mc{S},
\end{equation}
which turns $\mc{S}$ into a Riemannian manifold, with tangent spaces $T_{p}\mc{S} = T_{0}$ that do not depend on the base point $p$. In connection with $\mc{S}$, we define the 
\begin{equation}
\eins_{\mc{S}} = \frac{1}{|J|}(1,\dots,1)^{\T} \in \R^{|J|}
\qquad\qquad(\textbf{barycenter})
\end{equation}
of $\mc{S}$, i.e.~the uniform distribution, the orthogonal projection 
\begin{subequations}
\begin{align}\label{eq:def-Pi-T0}
\Pi_{T_{0}} &\colon \R^{|J|} \to T_{0}, &
\Pi_{T_{0}}(z) &= \big(\Diag(\eins)-\eins\eins_{\mc{S}}^{\T}\big) z,
\intertext{
and the linear mapping
}\label{eq:def-Pi-p}
\Pi_{p} &\colon \R^{|J|} \to T_{0}, &
\Pi_{p}(z) &= \big(\Diag(p)-p p^{\T}\big) z,\qquad p \in \mc{S}
\end{align}
\end{subequations}
satisfying
\begin{equation}\label{eq:Pi-p-T0-commute}
\Pi_{p} = \Pi_{p} \Pi_{T_{0}} = \Pi_{T_{0}} \Pi_{p}.
\end{equation}
Adopting the $\alpha$-connection with $\alpha=1$ from information geometry as introduced by Amari and Chentsov \cite[Section 2.3]{Amari:2000aa}, \cite{Ay:2017aa}, the exponential map based on the corresponding affine geodesics reads 
\begin{subequations}
\begin{align}\label{eq:exp-e-connection}
\Exp &\colon \mc{S} \times T_{0} \to \mc{S}, &
(p,v) &\mapsto \Exp_{p}(v) = \frac{e^{\frac{v}{p}}}{\la p, e^{\frac{v}{p}}\ra} p
\intertext{with inverse \cite[Appendix]{Astrom:2017ac}}
\label{eq:def-invexp}
\Exp^{-1} &\colon \mc{S} \times \mc{S} \to T_{0}, &
(p,q) &\mapsto \Exp_{p}^{-1}(q) = \Pi_{p}\log\frac{q}{p}.
\end{align}
\end{subequations}
Specifically, we define
\begin{subequations}
\begin{align}\label{eq:def-exp-1}
\exp_{p} 
&= \Exp_{p} \circ \Pi_{p} \colon \R^{|J|} = T_{0}\oplus\R\eins\to \mc{S}, &
z &\mapsto \frac{p e^{z}}{\la p, e^{z}\ra},\qquad \forall p \in \mc{S}
\intertext{with inverse \cite[Appendix]{Astrom:2017ac}}
\label{eq:def-invexp-eins}
\exp_{p}^{-1} &\colon \mc{S} \to T_{0}, &
q &\mapsto \Pi_{T_{0}}\log\frac{q}{p}.
\end{align}
\end{subequations}
\begin{remark}\label{rem:exp-T0}
Calling \eqref{eq:def-invexp-eins} the ``inverse'' map is justified by the fact that $\exp_{p}$ does not depend on any constant component $\R\eins \in T_{0}^{\perp}$ of the argument vector $z$. Yet, we choose $\R^{|J|}$ as domain because $\exp_{p}$ will be applied to arbitrary distance vectors $D_{i} \in \R^{|J|}$ (cf.~\eqref{eq:def-Di}) arising from given data, and the notation indicates that the implementation does not need to remove this component explicitly \cite[Remark 4]{Astrom:2017ac}.
\end{remark}
These mappings naturally extend to the collections of assignment vectors \eqref{eq:def-Wi}, regarded as points on the
\begin{equation}
\mc{W} = \mc{S} \times \dotsb \times \mc{S} \qquad(|I|\; \text{times})\qquad\qquad(\textbf{assignment manifold})
\end{equation}
with tangent space
\begin{equation}
\mc{T}_{0} = T_{0} \times \dotsb \times T_{0} \qquad(|I|\; \text{times})
\end{equation}
and the corresponding mappings
\begin{subequations}\label{eq:assignment-flow-maps}
\begin{align}
\eins_{\mc{W}} 
&= (\eins_{\mc{S}},\dotsc,\eins_{\mc{S}}) \in \mc{W}
&& (\textbf{barycenter})
\\ 
\Pi_{\mc{T}_{0}}(Z) 
&= \big(\Pi_{T_{0}}(Z_{1}),\dotsc,\Pi_{T_{0}},(Z_{|I|})\big) \in \mc{T}_{0}, &
W &\in \mc{W},\quad Z \in \R^{|I||J|}
\\ \label{eq:assignment-flow-maps-Pi-W}
\Pi_{W}(Z) 
&= \big(\Pi_{W_{1}}(Z_{1}),\dotsc,\Pi_{W_{|I|}},(Z_{|I|})\big) \in \mc{T}_{0}, &
W &\in \mc{W},\quad Z \in \R^{|I||J|}
\\ \label{eq:exp-e-W}
\Exp_{W}(V) 
&= \big(\Exp_{W_{1}}(V_{1}),\dotsc,\Exp_{W_{|I|}}(V_{|I|})\big) \in \mc{W}, &
W &\in \mc{W},\quad V \in \mc{T}_{0}
\end{align}
\end{subequations}
and $\exp_{W}, \Exp^{-1}_{W}, \exp^{-1}_{W}$ similarly defined based on \eqref{eq:def-invexp}, \eqref{eq:def-exp-1} and \eqref{eq:def-invexp-eins}.
Finally, we define the 
\textit{geometric mean} of assignment vectors \cite[Lemma 5]{Astrom:2017ac}
\begin{equation}\label{eq:def-geom-mean}
\mc{G}^{w}_{i}(W) = \Exp_{W_{i}} \Big(\sum_{k \in \mc{N}_{i}} w_{ik} \Exp_{W_{i}}^{-1}(W_{k})\Big)
= \exp_{W_{i}}\Big(\log\frac{\prod_{k \in \mc{N}_{i}} W_{k}^{w_{ik}}}{W_{i}}\Big),\qquad i \in I.
\end{equation}

\vspace{0.2cm}
Using this setting, the assignment flow accomplishes image labeling as follows. Based on \eqref{eq:def-FI}, \eqref{eq:def-GJ}, distance vectors
\begin{subequations}\label{eq:def-Di}
\begin{align}
D 
&= (D_{1},\dotsc,D_{|I|}) \in \R^{|I||J|}
\qquad\qquad\qquad\qquad
(\textbf{distance vectors})
\\
D_{i}
&=\big(d(f_{i},g_{1}),\dotsc,d(f_{i},g_{|J|})\big)^{\T},\quad i \in I
\end{align}
\end{subequations}
are defined and mapped to
\begin{subequations}\label{eq:def-LW}
\begin{align}
L(W) &= \exp_{W}(D) \in \mc{W},
\qquad\qquad\qquad\qquad\qquad
(\textbf{likelihood vectors})
\\ \label{eq:def-LiWi}
L_{i}(W_{i}) 
&:= \frac{W_{i} e^{-\frac{1}{\rho} D_{i}}}{\la W_{i},e^{-\frac{1}{\rho} D_{i}} \ra},\quad \rho > 0,\qquad i \in I,
\end{align}
\end{subequations}
where $\rho$ is a user parameter to normalize the distances induced by the specific features $f_{i}$ at hand. This representation of the data is regularized by local geometric smoothing to obtain
\begin{equation}\label{eq:def-SW}
S(W) \in \mc{W},\qquad
S_{i}(W) = \mc{G}^{w}_{i}\big(L(W)\big),\quad i \in I,
\qquad(\textbf{similarity vectors})
\end{equation}
which in turn evolves the assignment vectors $W_{i},\, i \in I$ through the
\begin{equation}\label{eq:assignment-flow}
\dot W = \Pi_{W}\big(S(W)\big),\qquad
W(0) = \eins_{\mc{W}}.
\qquad\qquad\qquad (\textbf{assignment flow})
\end{equation}
Methods for numerically integrating this flow are examined in the following sections.

\section{Geometric Runge-Kutta Integration}
\label{sec:Runge-Kutta}

We apply the general approach of \cite{Munthe-Kaas:1999aa} to our problem. For background and more details, we refer to 
\cite{Iserles:2005aa} and \cite[Chapter 4]{Hairer:2006aa}.
\subsection{General Approach}
\label{sec:Lie-General-Approach}
In order to apply Lie group methods to the integration of an ODE on a manifold $\mc{M}$, one has to check first if the ODE can be represented properly. Let  
\begin{subequations}\label{eq:Lambda-properties}
\begin{align}
\Lambda \colon \LG{G} \times \mc{M} &\to \mc{M}
\intertext{
denote the action of a Lie group $\LG{G}$ on $\mc{M}$ satisfying
}
\Lambda(e,p) &= p 
&&\text{with identity $e \in \LG{G}$},
\\
\Lambda(g_{1}\cdot g_{2},p) &= \Lambda(g_{1},\Lambda(g_{2},p)),
&&\text{for all}\;
g_{1},g_{2} \in \LG{G},\; p \in \mc{M}.
\end{align}
\end{subequations}
Furthermore, let $\Lg{g}$ denote the Lie algebra of $\LG{G}$, $\Lg{X}(\mc{M})$ the set of all smooth vector fields on $\mc{M}$, 
\begin{equation}
\lambda \colon \Lg{g}\times\mc{M} \to \mc{M}
\end{equation}
a smooth function and $\lambda_{\ast}$ the induced map defined by
\begin{equation}\label{eq:def-lambda-ast-a}
\lambda_{\ast} \colon \Lg{g} \to \Lg{X}(\mc{M}),
\qquad\qquad
(\lambda_{\ast} v)_{p} = \dv{t}\lambda(t v,p)\big|_{t=0}
\qquad\text{for all}\; v \in \Lg{g},\; p \in \mc{M}.
\end{equation}
Then $\lambda$ is a Lie algebra action if the induced map $\lambda_{\ast}$ is a Lie algebra homomorphism, i.e.~$\lambda_{\ast}$ is linear and satisfies $\lambda_{\ast}[u,v]=[\lambda_{\ast}u,\lambda_{\ast}v],\, u,v \in \Lg{g}$, with the Lie brackets on $\Lg{g}$ and $\Lg{X}(\mc{M})$ on the left-hand side and the right-hand side, respectively. In particular, based on a Lie group action $\Lambda$, a Lie algebra action is given by \cite[Lemma 4]{Munthe-Kaas:1999aa}
\begin{equation}\label{eq:lambda-by-Lambda}
\lambda(v,p) = \Lambda(\exp_{\LG{G}}(v),p),
\end{equation}
where $\exp_{\LG{G}} \colon \Lg{g} \to \LG{G}$ denotes the exponential map of $\LG{G}$. Thus, for this choice of $\lambda$, the induced map \eqref{eq:def-lambda-ast-a} is given by \cite[Thm.~5]{Munthe-Kaas:1999aa}
\begin{equation}\label{eq:lambda-ast-Lambda}
(\lambda_{\ast} v)_{p} = \dv{t}\Lambda(\exp_{\LG{G}}(t v),p)\big|_{t=0} \qquad\text{for all}\; v \in \Lg{g},\; p \in \mc{M}.
\end{equation}
Now, given an ODE on $\mc{M}$, the \textit{basic assumption} underlying the application of Lie group methods is the existence of a function $f \colon \R \times \mc{M} \to \Lg{g}$ such that the ODE admits the representation
\begin{equation}\label{eq:dot-y-on-M}
\dot y = \big(\lambda_{\ast} f(t,y)\big)_{y},\qquad y(0)=p.
\end{equation}
For sufficiently small $t$, the solution of \eqref{eq:dot-y-on-M} then can be parametrized as
\begin{subequations}\label{eq:dot-y-on-g}
\begin{align}
y(t) &= \lambda\big(v(t),p\big),
\intertext{
where $v(t) \in \Lg{g}$ satisfies the ODE
}\label{eq:dot-y-on-g-b}
\dot v &= (\dexp_{\LG{G}}^{-1})_{v}\big(f(t,\lambda(v,p))\big),\qquad v(0)=0,
\end{align}
\end{subequations}
with the inverse differential $(\dexp_{\LG{G}}^{-1})_{v}$ of $\exp_{\LG{G}}$ evaluated at $v \in \Lg{g}$. A major advantage of the representation \eqref{eq:dot-y-on-g} is that the task of numerical integration concerns the ODE \eqref{eq:dot-y-on-g-b} evolving on the vector space $\Lg{g}$, rather than the original ODE evolving on the manifold $\mc{M}$. As a consequence, established methods can be applied to \eqref{eq:dot-y-on-g-b}, possibly after approximating $\dexp_{\LG{G}}^{-1}$ by a truncated series in a computationally feasible form.

\subsection{Application to the Assignment Flow}
\label{sec:Lie-Assignment-Flow}
Assume an ODE on $\mc{S}$ defined by \eqref{eq:def-S} is given. The application of the approach of Section \ref{sec:Lie-General-Approach} is considerably simplified by identifying $\LG{G} = T_{0}$ with the \textit{flat} tangent space \eqref{eq:def-gp-T0} and consequently also $T_{0} \cong \Lg{g} = T_{e}\LG{G}$. One easily verifies that the action $\Lambda \colon T_{0} \times \mc{S} \to \mc{S}$ defined as
\begin{equation}\label{eq:Lambda-simplex}
\Lambda(v,p) = \exp_{p}(v),
\end{equation}
with the right-hand side given by \eqref{eq:def-exp-1}, satisfies \eqref{eq:Lambda-properties}, i.e.
\begin{subequations}\label{eq:Lambda-vp-rules}
\begin{align}
\Lambda(0,p) &= p,
\\
\Lambda(v_{1}+v_{2},p) &= \Lambda(v_{1},\Lambda(v_{2},p))
= \frac{p e^{v_{1}+v_{2}}}{\la p, e^{v_{1}+v_{2}} \ra}.
\end{align}
\end{subequations}
\begin{proposition}
The solution $W(t)$ to assignment flow \eqref{eq:assignment-flow} emanating from any $W_{0}=W(0)$ admits the representation
\begin{subequations}\label{eq:ass-flow-parametrization}
\begin{align}\label{eq:ass-flow-W-by-V}
W(t) &= \exp_{W_{0}}\big(V(t)\big)
\intertext{
where $V(t) \in \mc{T}_{0}$ solves
}\label{eq:ass-flow-dot-V}
\dot V &= \Pi_{\mc{T}_{0}} S\big(\exp_{W_{0}}(V)\big),\qquad V(0)=0.
\end{align}
\end{subequations}
\end{proposition}
\begin{proof}
Since geodesics through $0 \in T_{0}$ in directions $v \in T_{0}$ have the form $\gamma(t) = t v$, the differential of the exponential map of $T_{0}=\LG{G}$, $\exp_{T_{0}}(v) = \gamma(1) = v$, is the identity and thus \eqref{eq:lambda-ast-Lambda} gives
\begin{equation}
(\lambda_{\ast} v)_{p} 
= \dv{t}\Lambda\big(\gamma(t),p\big)\big|_{t=0}
= \Pi_{p}(v),
\end{equation}
with $\Pi_{p}$ defined by \eqref{eq:def-Pi-p}. As a result, the basic assumption \eqref{eq:dot-y-on-M} concerns ODEs on $\mc{S}$ that admit the representation
\begin{equation}
\dot p = \Pi_{p}\big(f(t,p)\big),\qquad p(0)=p_{0},
\end{equation}
for some function $f \colon \R \times \mc{S} \to T_{0}$ and some $p_{0} \in \mc{S}$. Since $\lambda=\Lambda$ by \eqref{eq:lambda-by-Lambda}, the parametrization \eqref{eq:dot-y-on-g} reads
\begin{subequations}
\begin{align}
p(t) &= \Lambda\big(v(t),p_{0}\big)
\intertext{
where $v(t) \in T_{0}$ solves
}
\dot v &= f\big(t,\Lambda(v,p_{0})\big),\qquad v(0)=0.
\end{align}
\end{subequations}
This setting extends to the assignment flow by defining (cf.~\eqref{eq:assignment-flow-maps}) $
\Lambda \colon \mc{T}_{0} \times \mc{W} \to \mc{W}$ and $\lambda_{\ast} \colon \mc{T}_{0} \to \mc{T}_{0}$ as
\begin{equation}
\Lambda(V,W) = \exp_{W}(V),\qquad
\big(\lambda_{\ast}(V)\big)_{W} = \Pi_{W}(V).
\end{equation}
The basic assumption \eqref{eq:dot-y-on-M} then reads
\begin{equation}
\dot W = \big(\lambda_{\ast}f(t,W)\big)_{W}
= \Pi_{W}\big(\Pi_{\mc{T}_{0}} S(W)\big)
\overset{\eqref{eq:Pi-p-T0-commute}}{=}
\Pi_{W}\big(S(W)\big),\qquad W(0) = \eins_{\mc{W}},
\end{equation}
which is the assignment flow \eqref{eq:assignment-flow}. Due to \eqref{eq:dot-y-on-M}, for any $W_{0}=W(0)$, it admits the representation
\begin{subequations}
\begin{align}
W(t) &= \Lambda\big(V(t),W_{0}\big),
\intertext{
where $V(t) \in \mc{T}_{0}$ solves
}
\dot V 
&= f\big(t,\Lambda(V,W_{0})\big)
= \Pi_{\mc{T}_{0}}S(\exp_{W_{0}}(V)),\qquad V(0)=0,
\end{align}
\end{subequations}
which is \eqref{eq:ass-flow-parametrization}.
\end{proof}
\begin{remark}\label{rem:non-basic-cases}
While the basic formulation \eqref{eq:assignment-flow} of the assignment flow is \textit{autonomous}, we keep in what follows the explicit time dependency of the function $f(t,\cdot)$ of the parametrization \eqref{eq:ass-flow-parametrization}, because in more advanced scenarios the flow may become \textit{non-autonomous}. A basic example concerns \textit{unsupervised} problems \cite{Zern:2018ac} where labels \textit{vary}, and hence the distance vectors \eqref{eq:def-Di} and in turn the vector field defining the assignment flow depend on $t$.
\end{remark}
Using the above representation and taking into account the simplifications of the general approach of Section \ref{sec:Lie-General-Approach}, the \textbf{RKMK algorithm} \cite{Munthe-Kaas:1999aa} for integrating the assignment flow from $t=0$ to $t=h$ is specified as follows. Let $a_{i,j}, b_{j}$ be the coefficients of an $s$-stage, $q$-th order classical Runge-Kutta method satisfying the consistency condition $c_{i}=\sum_{j \in [s]} a_{i,j}$ \cite[Section II]{Hairer:2008aa}. Starting from any point 
\begin{subequations}\label{eq:RKMK-algorithm}
\begin{align}
W_{0} &= W(0),
\intertext{
the algorithm amounts to compute the vector fields
}
U^{i} &= h \sum_{j \in [s]} a_{i,j} \widetilde U^{j},\qquad i \in [s]
\\ \label{eq:RKMK-f-step}
\widetilde U^{i} &= f\big(h c_{i},\Lambda(U^{i},W_{0})\big),\qquad i \in [s]
\\ \label{eq:RKMK-V-update}
V &= h \sum_{j \in [s]} b_{j} \widetilde U^{j}
\intertext{
and the update
} \label{eq:RKMK-W-update}
W(h) &= \Lambda(V,W_{0}).
\end{align}
\end{subequations}
Replacing $W_{0} \leftarrow W(h)$, computing the update and iterating this procedure generates the sequence $(W^{(k)})_{k \geq 0}$ which approximates $(W(t_{k}))_{k \geq 0},\, t_{k} = k h$.

A \textit{$s$-stage RKMK scheme} is specified using the corresponding Butcher tableau of the form
\begin{center}
\begin{tabular}{c|ccccc}
$0$ & \\
$c_{2}$ & $a_{21}$ \\
$c_{3}$ & $a_{31}$ & $a_{32}$ \\
\vdots & \vdots & \vdots & $\ddots$ \\
$c_{s}$ & $a_{s1}$ & $a_{s2}$ & $\dotsb$ & $a_{s(s-1)}$ \\ \hline
& $b_{1}$ & $b_{2}$ & $\dotsb$ & $b_{s-1}$ & $b_{s}$
\end{tabular}
\end{center}
Specifically, we consider the following \textit{explicit RKMK schemes} or order $1,2,3,4$ (`FE' stands for Forward Euler):
\begin{center}
\parbox{0.175\textwidth}{\centering
\begin{tabular}{c|c}
$0$ & \\ \hline
& $1$
\end{tabular}
} \hfill
\parbox{0.175\textwidth}{\centering
\begin{tabular}{c|cc}
$0$ & \\ 
$1$ & $1$ \\ \hline
& $1/2$ & $1/2$
\end{tabular}
} \hfill
\parbox{0.25\textwidth}{\centering
\begin{tabular}{c|ccc}
$0$ & \\ 
$1/3$ & $1/3$ \\ 
$2/3$ & $0$ & $2/3$ \\ \hline
& $1/4$ & $0$ & $3/4$
\end{tabular}
} \hfill
\parbox{0.3\textwidth}{\centering
\begin{tabular}{c|cccc}
$0$ & \\ 
$1/2$ & $1/2$ \\ 
$1/2$ & $0$ & $1/2$ \\
$1$ & $0$ & $0$ & $1$ \\ \hline
& $1/6$ & $1/3$ & $1/3$ & $1/6$
\end{tabular}
}
\parbox{0.175\textwidth}{\centering
Euler (\textbf{FE})
} \hfill
\parbox{0.175\textwidth}{\centering
Heun-2 (\textbf{H2})
} \hfill
\parbox{0.25\textwidth}{\centering
Heun-3 (\textbf{H3})
} \hfill
\parbox{0.3\textwidth}{\centering
Classical RK (\textbf{RK4})
} 
\end{center}
Note the increasing number of stages that raise the approximation order. This comes at a price, however, because each stage evaluates at step \eqref{eq:RKMK-f-step} the right-hand side of \eqref{eq:ass-flow-dot-V} which is the most expensive operation. As a consequence, it is not clear a priori if using a a multi-stage scheme and a larger step size $h$ is superior to a simpler scheme that is evaluated more frequently using a smaller step size.

In addition to the above explicit schemes, we consider the simplest \textit{implicit RKMK scheme} (`BE' stands for Backward Euler)
\begin{center}
\parbox{0.175\textwidth}{\centering
\begin{tabular}{c|c}
$1$ & $1$ \\ \hline
& $1$
\end{tabular}
} \\
\parbox{0.175\textwidth}{\centering
Euler (\textbf{BE})
}
\end{center}
Implicit schemes are known to be stable for much larger step sizes. Yet, they require to solve at every step a fixed point equation which is done by an iterative inner loop.

The performances of these numerical schemes are examined in Section \ref{sec:Experiments}.

\section{Linear Assignment Flow, Exponential Integrator}
\label{sec:ExponentialIntegrators}

The ODE \eqref{eq:ass-flow-dot-V} which parametrizes the assignment flow together with \eqref{eq:ass-flow-W-by-V}, evolves on a linear space but is a \textit{nonlinear} system of ordinary differential equations. In this section, we provide an \textit{approximate representation} of the assignment flow within time intervals through a \textit{linear} ODE evolving on the tangent space (Section \ref{sec:linearized-flow}), and a corresponding numerical scheme (Section \ref{sec:Exponential-Integrator}). 

The resulting flow on the assignment \textit{manifold} is still nonlinear, though. The basic idea is to capture locally a major part of the nonlinearity of the (full) assignment flow, by a linear ODE on the tangent space that enables to apply alternative integration schemes.

\subsection{Linear Assignment Flow}
\label{sec:linearized-flow}

Our Ansatz has two ingredients. Firstly, we adopt the parametrization
\begin{equation}\label{eq:W-by-V-expW0}
W(t) = \Exp_{W_{0}}\big(V(t)\big),\qquad V(t) \in \mc{T}_{0}
\end{equation}
of the solution $W(t)$ to the assignment flow by a trajectory in the tangent space $\mc{T}_{0}$, similar to \eqref{eq:ass-flow-W-by-V}, except for using the `true' exponential map \eqref{eq:exp-e-connection} and \eqref{eq:exp-e-W}, respectively, corresponding to the underlying affine connection. Secondly, we use an \textit{affine approximation} of the vector field on the right-hand side of \eqref{eq:assignment-flow}, that defines the assignment flow. 
The following corresponding definition generalizes the flow studied by \cite{Ay:2005aa} from the barycenter to arbitrary base points $W_{0}$, and from a flow on $\mc{S}$ to a flow on $\mc{W}$.
\begin{definition}[\textbf{linear assignment flow}]
We call \textit{linear assignment flow} every flow induced by an ODE of the form
\begin{equation}\label{eq:ass-flow-linear}
\dot W = \Pi_{W}\Big(s_{0} + S_{0} \Pi_{W_{0}}\log\frac{W}{W_{0}}\Big),\qquad W(0) = W_{0} \in \mc{W},
\end{equation}
with a fixed vector $s_{0}$ and a fixed matrix $S_{0}$, for arbitrary $W_{0}$.
\end{definition}
An important property of the flow \eqref{eq:ass-flow-linear} -- which explains its name -- is the possibility to parametrize it by a \textit{linear} ODE evolving on the tangent space $\mc{T}_{0}$.
\begin{proposition}\label{prop:linear-ass-parametrization}
The linear assignment flow \eqref{eq:ass-flow-linear} admits the representation
\begin{subequations}\label{eq:linear-ass-parametrization}
\begin{align}\label{eq:W-ass-flow-linear}
W(t) &= \Exp_{W_{0}}\big(V(t)\big),
\intertext{
where $V(t) \in \mc{T}_{0}$ solves
}\label{eq:dot-V-ass-flow-linear}
\dot V &= \Pi_{W_{0}}(s_{0} + S_{0} V),\qquad V(0)=0.
\end{align}
\end{subequations}
\end{proposition}
\begin{proof}
Parametrization \eqref{eq:W-by-V-expW0} yields 
\begin{equation}\label{eq:proof-lin-ass-Vt}
V(t) = \Exp_{W_{0}}^{-1}(W(t))
\overset{\eqref{eq:def-invexp}}{=}
\Pi_{W_{0}}\log\frac{W_{(t)}}{W_{0}}
\end{equation}
and by differentiation
\begin{subequations}
\begin{align}
\dot V(t) &= \Pi_{W_{0}}\Big(\frac{\dot W(t)}{W(t)}\Big).
\intertext{
Solving \eqref{eq:ass-flow-linear} for $\frac{\dot W}{W}$ after inserting \eqref{eq:assignment-flow-maps-Pi-W}, and substitution in the preceding equation gives
}
&= \Pi_{W_{0}}\Big(s_{0} + S_{0} \Exp_{W_{0}}^{-1}\big(W(t)\big)
- \big\la W(t),s_{0} + S_{0} \Exp_{W_{0}}^{-1}\big(W(t)\big)\big\ra\eins\Big),
\intertext{
and since $\Pi_{W_{0}}\eins=0$ by \eqref{eq:def-Pi-p}
}
&= \Pi_{W_{0}}\Big(s_{0} + S_{0} \Exp_{W_{0}}^{-1}\big(W(t)\big)\Big)
\overset{\eqref{eq:proof-lin-ass-Vt}}{=} 
\Pi_{W_{0}}\big(s_{0} + S_{0} V(t)\big).
\end{align}
\end{subequations}
The initial condition follows from $V(0)=\Exp_{W_{0}}^{-1}(W_{0})=0$.
\end{proof}
\begin{remark}\label{rem:ass-flow-linear}
Note that, despite the linearity of \eqref{eq:dot-V-ass-flow-linear}, the resulting flow \eqref{eq:W-ass-flow-linear} solving \eqref{eq:ass-flow-linear} is \textit{nonlinear}. Thus, one may hope to capture the major nonlinearity of the full assignment flow \eqref{eq:assignment-flow} by a linear ODE on the tangent space, at least locally in some time interval. Within this interval, the evaluation of \eqref{eq:def-SW} is \textit{not} required, and the linearity of the tangent space ODE \eqref{eq:dot-V-ass-flow-linear} can be exploited for integration.
\end{remark}
We conclude this section by computing the natural choice
\begin{equation}\label{eq:s0-S0-natural}
s_{0} = S(W_{0}),\qquad
S_{0} = dS_{W_{0}}
\end{equation}
of the parameters of the linear assignment flow \eqref{eq:ass-flow-linear} in explicit form, where $s_{0}$ is immediate due to \eqref{eq:def-SW}, but the Jacobian $S_{0} = dS_{W_{0}}$ of $S(W)$,  evaluated at $W_{0}$, is not.
\begin{proposition}\label{prop:dS}
Let $S(W) \in \R^{|I||J|}$ denote the global similarity vector obtained by stacking the local similarity vectors $S_{1}(W),\dotsc,S_{|I|}(W)$ of \eqref{eq:def-SW}. Then, with
\begin{subequations}\label{eq:lem-dS}
\begin{align}\label{eq:def-s0}
s_{0} &= S(W_{0}),\qquad
s_{0i} = S_{i}(W_{0}),\qquad
s_{0i,j} = S_{ij}(W_{0}) = \big(S_{i}(W_{0})\big)_{j},
\quad i \in I,\; j \in J
\intertext{
and the projection $\Pi_{s_{0,i}}$ defined by \eqref{eq:def-Pi-p}, the Jacobian of $S(W)$ at $W_{0} \in \mc{W}$ is given by
}
S_{0} = dS_{W_{0}} &= \bpm 
A_{11}(W_{0}) & \dotsb & A_{1|I|}(W_{0}) \\
\vdots & \ddots & \vdots \\
A_{|I|1}(W_{0}) & \dotsb & A_{|I||I|}(W_{0})
\epm \in \R^{|I||J| \times |I||J|},
\intertext{
where each $|J| \times |J|$ block matrix has the form
}\label{eq:def-dS-ik}
A_{ik}(W_{0})(V_{k}) 
&= \begin{cases}
w_{ik}\Pi_{s_{0,i}}\Big(\frac{V_{k}}{W_{0k}}\Big),
&\text{if}\; k \in \mc{N}_{i} \\
0, &\text{if}\; k \not\in\mc{N}_{i}
\end{cases},
\qquad
W_{0} \in \mc{W},\quad V_{k} \in T_{0},\qquad i,k \in I
\intertext{
and the non-zero entries if $k \in \mc{N}_{i}$ (using \eqref{eq:def-s0}) 
}\label{eq:dSik-entries}
A_{ik,jl}(W_{0}) 
&= \big(A_{ik}(W_{0})\big)_{jl}
= w_{ik} \begin{cases}
(1 - s_{0i,j})\frac{s_{0i,j}}{W_{0k,j}},
&\text{if}\; j=l
\\
-s_{0i,j}\frac{s_{0i,l}}{W_{0k,l}},
&\text{if}\; j\neq l
\end{cases},\qquad
j,l \in J.
\end{align}
\end{subequations}
\end{proposition}
The proof follows below after two preparatory Lemmata.
\begin{lemma}\label{lem:dexp-p}
Let $p \in \mc{S}$. Then the differential of $\exp_{p} \colon T_{0} \to \mc{S}$ at $u \in T_{0}$ applied to $v \in T_{0}$ is given by
\begin{equation}\label{eq:dexp-puv}
d\exp_{p}(u)(v) = \Pi_{\exp_{p}(u)}(v),\quad
u,v \in T_{0},\; p \in \mc{S}.
\end{equation}
Moreover, we have
\begin{equation}\label{eq:exp-eins-v}
\exp_{p}(v-\log p) = \exp_{\eins_{\mc{S}}}(v),\quad
v \in T_{0},\; p \in \mc{S}.
\end{equation}
\end{lemma}
\begin{proof}
Let $\gamma(t)$ be a smooth curve in $T_{0}$ with $\gamma(0)=u$ and $\dot\gamma(0)=v$. Using \eqref{eq:def-exp-1}, we compute
\begin{equation}
\frac{d}{dt}\frac{p e^{\gamma(t)}}{\la p, e^{\gamma(t)}\ra}\bigg|_{t=0}
= \frac{\la p,e^{u} \ra p v e^{u} - \la p, v e^{u}\ra p e^{u}}{\la p, e^{u}\ra^{2}}
= \frac{p e^{u}}{\la p, e^{u} \ra}\Big(v - \frac{\la p e^{u},v\ra}{\la p, e^{u}\ra}\eins\Big)
= \Pi_{\exp_{p}(u)}(v),
\end{equation}
which is \eqref{eq:dexp-puv}. As for \eqref{eq:exp-eins-v}, using the representation
\begin{equation}\label{eq:p-exp-eins}
p 
\overset{\eqref{eq:def-exp-1}}{=}
\exp_{\eins_{\mc{S}}}(\log p) 
= \exp_{\eins_{\mc{S}}}(\Pi_{T_{0}}\log p),
\end{equation}
where the last equation takes into account Remark \eqref{rem:exp-T0}, we obtain 
\begin{subequations}
\begin{align}
\exp_{p}(v-\log p) 
&= \exp_{p}(v-\Pi_{T_{0}}\log p)
\overset{\eqref{eq:p-exp-eins}}{=}
\exp_{\exp_{\eins_{\mc{S}}}(\Pi_{T_{0}}\log p)}(v-\Pi_{T_{0}}\log p)
\\
&\overset{\eqref{eq:Lambda-simplex}}{=} 
\Lambda\big(v-\Pi_{T_{0}}\log p,\Lambda(\Pi_{T_{0}}\log p,\eins_{\mc{S}})\big)
\overset{\eqref{eq:Lambda-vp-rules}}{=}
\Lambda(v-\Pi_{T_{0}}\log p + \Pi_{T_{0}}\log p,\eins_{\mc{S}})
\\
&= \Lambda(v,\eins_{\mc{S}})
= \exp_{\eins_{\mc{S}}}(v).
\end{align}
\end{subequations}
\end{proof}
We use this Lemma to represent the similarity vectors in a convenient form for subsequently proving Prop.~\ref{prop:dS}.
\begin{lemma}
The similarity vectors \eqref{eq:def-SW} admit the representation
\begin{equation}\label{eq:Si-proof-repr}
S_{i}(W) = \exp_{\eins_{\mc{S}}}\Big(\sum_{k \in \mc{N}_{i}}
w_{ik}\big(\log W_{k}-\frac{1}{\rho}D_{k}\big)\Big),\quad
i \in I.
\end{equation}
\end{lemma}
\begin{proof}
By \eqref{eq:def-SW} and \eqref{eq:def-geom-mean}, we obtain
\begin{subequations}
\begin{align}
S_{i}(W)
&= \exp_{W_{i}}\Big(\log\frac{\prod_{k \in \mc{N}_{i}} L_{k}(W_{k})^{w_{ik}}}{W_{i}}\Big)
= \exp_{W_{i}}\Big(\sum_{k \in \mc{N}_{i}} w_{ik} \log L_{k}(W_{k}) - \log W_{i}\Big)
\\
&\overset{\eqref{eq:def-LiWi}}{=}
\exp_{W_{i}}\bigg(\sum_{k \in \mc{N}_{i}} w_{ik}\Big(
\log W_{k}-\frac{1}{\rho} D_{k} - \log\big(\la W_{k},e^{-\frac{1}{\rho}D_{k}}\ra\big)\eins\Big) - \log W_{i}\bigg)
\intertext{
using again $\exp_{p}(v+\lambda\eins)=\exp_{p}(v)$ for all $\lambda \in \R,\, v \in T_{0},\, p \in \mc{S}$ (cf.~Remark \ref{rem:exp-T0})
}
&= \exp_{W_{i}}\Big(\sum_{k \in \mc{N}_{i}} w_{ik}\big(
\log W_{k}-\frac{1}{\rho} D_{k}\big) - \log W_{i}\Big)
\overset{\eqref{eq:exp-eins-v}}{=}
\exp_{\eins_{\mc{S}}}\Big(\sum_{k \in \mc{N}_{i}} w_{ik}\big(
\log W_{k}-\frac{1}{\rho} D_{k}\big)\Big).
\end{align}
\end{subequations}
\end{proof}
\begin{proof}[Proof of Prop.~\ref{prop:dS}]
Setting
\begin{equation}\label{eq:Si-exp-Zi}
S_{i}(W)
\overset{\eqref{eq:Si-proof-repr}}{=}
\exp_{\eins_{\mc{S}}} \circ\, Z_{i}(W),\qquad
Z_{i}(W) = \sum_{k \in \mc{N}_{i}}
w_{ik}\big(\log W_{k}-\frac{1}{\rho}D_{k}\big),
\end{equation}
we compute using a smooth curve $\gamma(t)$ in $\mc{W}$ with $\gamma(0)=W$ and $\dot\gamma(0)=V$,
\begin{equation}\label{eq:dZi-WV}
dZ_{i}(W)(V) = \frac{d}{dt} Z_{i}\big(\gamma(t)\big)\big|_{t=0} = \sum_{k \in \mc{N}_{i}} w_{ik} \frac{d}{dt}\log\big(\gamma(t)\big)\big|_{t=0}
= \sum_{k \in \mc{N}_{i}} w_{ik} \frac{V_{k}}{W_{k}}.
\end{equation}
Thus, using \eqref{eq:Si-exp-Zi} and \eqref{eq:dexp-puv} gives
\begin{subequations}
\begin{align}
dS_{i}(W)(V)
&\overset{\eqref{eq:Si-exp-Zi}}{=}
d\exp_{\eins_{\mc{S}}}\big(Z_{i}(W)\big)\big(dZ_{i}(W)(V)\big)
\overset{\eqref{eq:dexp-puv},\eqref{eq:Si-exp-Zi}}{=} 
\Pi_{S_{i}(W)}\big(dZ_{i}(W)(V)\big),
\intertext{
and using the linearity of the map $\Pi_{S_{i}(W)}$ and \eqref{eq:dZi-WV},
}
&=
\sum_{k \in \mc{N}_{i}} w_{ik} \Pi_{S_{i}(W)}\Big(\frac{V_{k}}{W_{k}}\Big),
\end{align}
\end{subequations}
which proves \eqref{eq:def-dS-ik}. 
Inserting $\Pi_{S_{i}(W)}$ due to 
\eqref{eq:def-Pi-p} yields \eqref{eq:dSik-entries}.
\end{proof}
The following section specifies an alternative integration scheme for the linear assignment flow \eqref{eq:ass-flow-linear}. Its approximation properties are numerically examined in Section \ref{sec:Experiments}.

\subsection{Exponential Integrator}
\label{sec:Exponential-Integrator}

We focus on the linear ODE \eqref{eq:dot-V-ass-flow-linear} that together with \eqref{eq:W-ass-flow-linear} determines the linear assignment flow due to \eqref{eq:ass-flow-linear}. The solution to \eqref{eq:dot-V-ass-flow-linear} is given by Duhamel's formula  \cite{Teschl:2012aa},
\begin{equation}\label{eq:V-t-Duhamel}
V(t) = \expm \left(t A\right) V(0) + \int_{0}^{t} \expm\big((t-\tau)A\big) a \dd{\tau}\quad\text{where}\quad
A = \Pi_{W_{0}} S_{0},\quad
a = \Pi_{W_{0}} s_{0},
\end{equation}
which involves the matrix exponential of the matrix $A$ of dimension $|I||J| \times |I||J|$ (square of number of pixels $\times$ number of labels), which can be quite large in image labeling problems ($10^4$--$10^{7}$ variables). Explicitly computing the matrix exponential is neither feasible, because it is dense even if $A$ is sparse, nor required in view of the multiplication with the vector $a$. Rather, taking into account $V(0)=0$ and that uniformly converging series can be integrated term by term, we set $t=T$ large enough and evaluate
\begin{subequations}\label{eq:V-by-phi1}
\begin{align}
V(T) &= \int_{0}^{T} \expm\big((T-\tau)A\big) a \dd{\tau}
= \expm(T A)\int_{0}^{T} \sum_{k=0}^{\infty} \frac{(-\tau A)^{k}}{k!} a \dd{\tau}
\\
&= \expm(T A) \sum_{k=0}^{\infty} \Big[\frac{\tau (-\tau A)^{k}}{(k+1)!}\Big]_{\tau=0}^{T} a
= T \expm(T A) \sum_{k=0}^{\infty} \frac{(-T A)^{k}}{(k+1)!} a
\\ \label{eq:V-by-phi1-c}
&= T\vphi_{1}(T A) a
\end{align}
\end{subequations}
where $\vphi_{1}$ is the entire function
\begin{equation}\label{eq:def-phi-1}
\vphi_{1}(z) = \sum_{k=0}^{\infty} \frac{z^{k}}{(k+1)!} = \frac{e^{z}-1}{z},
\end{equation}
whose series representation yields valid expressions \eqref{eq:V-by-phi1-c} also if the matrix $A$ is singular.

We refer to \cite{Higham:2008aa} for a detailed exposition of matrix functions and to \cite{Moler2003} and \cite[Sections 10 and 13]{Higham:2008aa} for a survey of methods for computing the matrix exponential and the product of matrix functions times a vector. For large problem sizes, the established methods of the two latter references are known to deteriorate, however, and methods based on Krylov subspaces have been developed \cite{Saad:1992aa,Hochbruck:1997aa} and become the method of choice in connection with exponential integrators \cite{Hochbruck:2010aa}.

We confine ourselves with sketching below a state-of-the-art method \cite{Niesen:2012aa} for the approximate numerical evaluation of \eqref{eq:V-by-phi1}. The evaluation of its performance for integrating the linear assignment flow and a comparison to the methods of Section \ref{sec:step-size-linear}, are reported in Section \ref{sec:Experiments}. A more comprehensive evaluation of further recent methods for evaluating \eqref{eq:V-by-phi1} that cope with large problem sizes as well (e.g.~\cite{Al-Mohy2011}), is beyond the scope of this paper.

In order to compute approximately $\vphi_{1}(T A) a$, one considers the Krylov subspace
\begin{equation}
K_{m} = \Span\{a, A a, \dotsc,A^{m-1} a\},
\end{equation}
with orthogonal basis $V_{m} = (v_{1},\dotsc,v_{m})$ arranged as column vectors of an orthogonal matrix $V_{m}$ and computed using the basic Arnoldi iteration \cite{Saad:1992aa}. The action of $A$ is approximated by
\begin{equation}
    H_{m} = V_{m}^{\T} A V_{m}
\end{equation}
which in turn yields the approximation
\begin{equation}\label{eq:vphi-1-approximation}
    \vphi_{1}(A) a
    \approx \vphi_{1}(V_{m} H_{m} V_{m}^{\T}) a = V_{m} \vphi_{1}(H_{m}) V_{m}^{\T} a
    = \|a\| V_{m}\vphi_{1}(H_{m}) e_{1},
\end{equation}
where $e_{1}=(1,0,\dotsc,0)^{\T}$ denotes the first unit vector and the last equality is implied by the Arnoldi iteration producing $V_{m}, H_{m}$, which sets $v_1 = a/\|a\|$. Note that $\vphi_{1}$ merely has to be applied to the much smaller $m \times m$ matrix $H_{m}$, which can be savely and efficiently computed using standard methods \cite{Moler2003,Higham:2008aa}. The vector of $\vphi_{1}(H_{m}) e_{1}$ can recovered \cite[Thm.~1]{Sidje1998} in form of the upper $m$ entries of the last column of $\expm(\widehat{H}_{m,k})$ with the extended matrix
\begin{equation}
\widehat H_{m} = \bpm H_{m} & e_{1} \\ 0 & 0 \epm.
\end{equation}
If the degree of the minimal polynomial of $a$ (i.e.~the the nonzero monic polynomial $p$ of lowest degree such that $p(A) a=0$) is equal to $m$, then the approximation \eqref{eq:vphi-1-approximation} is even exact \cite[Thm.~3.6]{Saad:1992aa}.

\section{Step Sizes, Adaptivity}
\label{sec:Adaptivity}

We specify in this section step size selection for the numerical RKMK schemes of Section \ref{sec:Runge-Kutta}. In addition, for  
the \textit{linear assignment flow} (Section \ref{sec:linearized-flow}), we conduct a local error analysis in Section \ref{sec:step-size-linear} for RK schemes based on the linearity of the tangent space ODE that governs this flow. A corresponding explicit error estimate enables to determine a sequence $(h_{k})_{k \geq 0}$ of step sizes that ensure a prespecified local accuracy at each step $k$.

In order to determine step sizes for the \textit{nonlinear assignment flow} \eqref{eq:assignment-flow}, we proceed differently, because the corresponding vector field depends nonlinearly on the current iterate and estimating local Lipschitz constants will be expensive and less sharp. We therefore adapt in Section \ref{sec:step-size-nonlinear} classical methods for local error estimation and step size selection for nonlinear ODEs based on \textit{embedded} Runge-Kutta methods \cite[Section II.4]{Hairer:2008aa}, to the \textit{geometric} RKMK methods of Section \ref{sec:Runge-Kutta}.
 
The experimental evaluation of both approaches is reported in Section \ref{sec:Experiments}.

\subsection{Linear Assignment Flow}
\label{sec:step-size-linear}

We focus on the linear ODE \eqref{eq:dot-V-ass-flow-linear} that together with \eqref{eq:W-ass-flow-linear} determines the linear assignment flow \eqref{eq:ass-flow-linear}. Due to its approximation property demonstrated in Section \ref{sec:Approximation-Property}, we only consider the linearization point $W_{0}=\eins_{\mc{W}}$. Since the ODE \eqref{eq:ass-flow-linear} evolves on the linear space $\mc{T}_{0}$, we apply the \textit{classical} $s$-stage explicit Runge-Kutta (RK) scheme, rather than the \textit{geometric} $s$-stage RKMK scheme \eqref{eq:RKMK-algorithm}, to obtain
\begin{subequations}\label{eq:RK-Euclidean}
\begin{align}
U^{i} &= \Pi_{W_{0}} s_{0} + \Pi_{W_{0}} S_{0}\Big(
V^{(k)} + h \sum_{j \in [s-1]} a_{i,j} U^{j}\Big),\qquad i \in [s],
\label{eq:RK-Ui} \\ \label{eq:RK-V-update}
V^{(k+1)} &= V^{(k)} + h \sum_{i \in [s]} b_{i} U^{i},\qquad
V^{(0)} = V(0) = 0.
\end{align}
\end{subequations}
Specifically, regarding the \textit{explicit schemes} listed at the end of Section \ref{sec:Lie-Assignment-Flow} in terms of their Butcher tableaus, consecutively inserting \eqref{eq:RK-Ui} into \eqref{eq:RK-V-update} yields with the shorthands $a, A$ defined by \eqref{eq:V-t-Duhamel},
\begin{subequations}\label{eq:RK-Euclidean-q}
\begin{align}
V^{(k+1)} 
&= h a + (I + h A) V^{k},
&&(\textbf{FE})
\\
V^{(k+1)}
&= \Big(h + \frac{h^{2}}{2} A\Big) a + \Big(I + h A + \frac{h^{2}}{2} A^{2}\Big) V^{(k)},
&&(\textbf{H2})
\\
V^{(k+1)}
&= \Big(h+\frac{h^{2}}{2}A+\frac{h^{3}}{6}A^{2}\Big)a
+ \Big(I+h A+\frac{h^{2}}{2}A^{2}+\frac{h^{3}}{6}A^{3}\Big) V^{(k)},
&&(\textbf{H3})
\\
V^{(k+1)}
&= \Big(h+\frac{h^{2}}{2}A+\frac{h^{3}}{6}A^{2}+\frac{h^{4}}{24}A^{3}\Big) a + \Big(I+h A+\frac{h^{2}}{2}A^{2}+\frac{h^{3}}{6}A^{3}+\frac{h^{4}}{24} A^{4}\Big) V^{(k)},
&&(\textbf{RK4})
\intertext{with}
V^{(0)} &= 0.
\end{align}
\end{subequations}
Comparison with \eqref{eq:V-t-Duhamel} shows that due to the linearity of the ODE, each scheme results in a corresponding Taylor series approximation, depending on its order $q$, of the equation
\begin{equation}\label{eq:linear-update-exact}
V(t_{k+1}) = h\vphi_{1}(h A) a + \expm(h A) V(t_{k})
\end{equation}
that is,
\begin{subequations}\label{eq:linear-RK-update}
\begin{align}
V^{(k+1)} &= 
h \Big(\sum_{i=0}^{q-1} \frac{(h A)^{i}}{(i+1)!} \Big) a
+ \Big(\sum_{i=0}^{q} \frac{(h A)^{i}}{i!} \Big) V^{(k)}
\label{eq:Vk+1-series} \\
&= p_{1,q}(h A) a + p_{2,q}(h A) V^{(k)},
\end{align}
\end{subequations}
with matrix-valued polynomials $p_{1,q}, p_{2,q}$. 
Our strategy for choosing the step size $h$ is based 
on the local error estimate specified below as 
Theorem \ref{thm:local-error} and prepared by the 
following Lemma.
\begin{lemma}\label{lem:local-error}
Let $q \in \N$ and $t \in \R$. Then
\begin{equation}\label{lem:taylor-sum}
\sum_{i=0}^{q} \frac{t^{i}}{i!} = e^{t} \frac{\Gamma(1+q,t)}{q!}
\end{equation}
with the incomplete Gamma function
\begin{equation}\label{eq:def-Gamma-incomplete}
\Gamma(1+q,t) = \int_{t}^{\infty} \tau^{q} e^{-\tau}\dd{\tau}.
\end{equation}
\end{lemma}
\begin{proof}
Partial integration shows the recursion
\begin{subequations}
\begin{align}
\Gamma(1+q,t) 
&= -\tau^{q}e^{-\tau}\big|_{t}^{\infty}
+ \int_{t}^{\infty} q \tau^{q-1} e^{-\tau}\dd{\tau}
\\
&= t^{q} e^{-t} + q \Gamma(q,t)
\\
&= t^{q} e^{-t} + q\big(t^{q-1} e^{-t} + (q-1)\Gamma(q-1,t)\big)
\\
&= (t^{q}+q t^{q-1}) e^{-t} + q (q-1) \Gamma(q-1,t)
\\
&= (t^{q}+q t^{q-1}) e^{-t} + q (q-1)\big(
t^{q-2} e^{-t} + (q-2)\Gamma(q-2,t)\big)
\\
&= (t^{q}+q t^{q-1}+q (q-1)t^{q-2})e^{-t} + q (q-1)(q-2)\Gamma(q-2,t)
\\
&= \dotsb
\\
&= (t^{q}+q t^{q-1} + \dotsb + q (q-1) \dotsb 2 \cdot t) e^{-t} + q! \,\underbrace{\Gamma(1,t)}_{= e^{-t}}
\\
&= (t^{q}+q t^{q-1} + \dotsb + q (q-1) \dotsb 2 \cdot t + q!) e^{-t}.
\end{align}
\end{subequations}
Subdividing both sides by $q!\, e^{-t}$ yields \eqref{lem:taylor-sum}.
\end{proof}
\begin{theorem}\label{thm:local-error}
Let $V(t)$ solve \eqref{eq:ass-flow-linear} with $W_{0}=\eins_{\mc{W}}$, and let $(V^{(k)})_{k > 0}$ be a sequence generated by a RK scheme \eqref{eq:RK-Euclidean} of order $q$. Set $V(t_{k})=V^{(k)}$ in \eqref{eq:linear-update-exact}. Then $V(t_{k+1})$ in \eqref{eq:linear-update-exact} is the exact value of the linear assignment flow emanating from $V^{(k)}$, and regarding \eqref{eq:linear-RK-update} the local error estimate
\begin{subequations}\label{eq:local-error-bound}
\begin{align}\label{eq:local-error-bound-a}
\|V(t_{k+1})-V^{(k+1)}\| 
&\leq e^{h\|A\|}\Big(1-\frac{\Gamma(1+q,h\|A\|)}{q!}\Big)\Big(\frac{\|a\|}{\|A\|} + \|V^{(k)}\|\Big)
\\ \label{eq:local-error-bound-b}
&< e^{h\|A\|} (1-e^{-h\|A\|})^{(1+q)} 
\Big(\frac{\|a\|}{\|A\|} + \|V^{(k)}\|\Big)
\end{align}
\end{subequations}
holds, where $\Gamma(1+q,h\|A\|)$ is given by \eqref{eq:def-Gamma-incomplete} and $\|A\|$ denotes the spectral norm of the matrix $A = \Pi_{W_{0}}(S_{0})$.
\end{theorem}
\begin{proof}
Using \eqref{eq:linear-update-exact}, \eqref{eq:linear-RK-update} and $V(t_{k})=V^{(k)}$, we bound the local error by
\begin{subequations}\label{eq:proof-bound-series}
\begin{align}
\|V(t_{k+1})-V^{(k+1)}\|
&\leq \|h\vphi_{1}(h A)-p_{1,q}(h A)\| \|a\|
+ \|\expm(h A)-p_{2,q}(h A)\| \|V^{(k)}\|,
\intertext{
and inserting the series \eqref{eq:Vk+1-series} gives
}\label{eq:proof-bound-series-b}
&\leq h \Big(\sum_{i=q}^{\infty} \frac{(h\|A\|)^{i}}{(i+1)!}\Big) \|a\|
+ \Big(\sum_{i=q+1}^{\infty} \frac{(h\|A\|)^{i}}{i!} \Big) \|V^{(k)}\|.
\end{align}
\end{subequations}
Both series absolutely converge for any $h$. 
By Lemma \ref{lem:local-error}, we have
\begin{subequations}
\begin{align}
\sum_{i=q} \frac{t^{i}}{(i+1)!}
&\overset{j=i+1}{=} \sum_{j=q+1}^{\infty} \frac{1}{t} \cdot \frac{t^{j}}{j!}
= \frac{e^{t}}{t}\Big(1-\frac{\Gamma(1+q,t)}{q!}\Big),
\\
\sum_{i=q+1}^{\infty} \frac{t^{i}}{i!}
&= e^{t}\Big(1-\frac{\Gamma(1+q,t)}{q!}\Big).
\end{align}
\end{subequations}
Applying these equations to \eqref{eq:proof-bound-series-b} yields
\begin{subequations}
\begin{align}
h \sum_{i=q}^{\infty} \frac{(h\|A\|)^{i}}{(i+1)!}
= \frac{e^{h\|A\|}}{\|A\|}\Big(1-\frac{\Gamma(1+q,h\|A\|)}{q!}\Big),
\\
\sum_{i=q+1}^{\infty} \frac{(h\|A\|)^{i}}{i!}
= e^{h\|A\|}\Big(1-\frac{\Gamma(1+q,h\|A\|)}{q!}\Big),
\end{align}
\end{subequations}
and substitution into \eqref{eq:proof-bound-series}
\begin{equation}\label{eq:proof-local-bound}
\|V(t_{k+1})-V^{(k+1)}\|
\leq e^{h\|A\|}\Big(1-\frac{\Gamma(1+q,h\|A\|)}{q!}\Big)\Big(\frac{\|a\|}{\|A\|} + \|V^{(k)}\|\Big),
\end{equation}
which is \eqref{eq:local-error-bound-a}. 
To show \eqref{eq:local-error-bound-b}, we use the representation 
\begin{equation}\label{eq:Gamma-inc-repr}
\frac{1}{p}\Gamma\Big(\frac{1}{p},x^{p}\Big)
\overset{\eqref{eq:def-Gamma}}{=} 
\frac{1}{p} \int_{x^{p}}^{\infty} t^{\frac{1}{p}-1} e^{-t} \dd{t}
\overset{t=\tau^{p}}{=} 
\int_{x}^{\infty} e^{-\tau^{p}}\dd{\tau}
\end{equation}
and the lower bound \cite[Corollary of Thm.~1]{Alzer:1997aa}
\begin{equation}\label{eq:Gamma-Alzer-bound}
\frac{1}{\Gamma(1+1/p)} \int_{x}^{\infty} e^{-t^{p}}\dd{t}
> 1 - (1-e^{-\alpha x^{p}})^{1/p},\qquad
\alpha \geq \max\big\{1,\big(\Gamma(1+1/p)\big)^{-p}\big\},
\end{equation}
that holds for all $x > 0$ and $0 < p \neq 1$, with the Gamma function   
\begin{equation}\label{eq:def-Gamma}
\Gamma(q) = \int_{0}^{\infty} \tau^{q-1} e^{-\tau}\dd{\tau},\qquad
\Gamma(n+1) = n! \quad\text{if}\;n \in \N.
\end{equation}
Put 
\begin{equation}\label{eq:proof-Gamma-set-x-p}
x=h\|A\| \qquad\text{and}\qquad
p=\frac{1}{1+q}. 
\end{equation}
Since $q \geq 1$ and $\big(\Gamma(1+1/p)\big)^{-p} = \Gamma(2+q)^{-\frac{1}{1+q}} = \big(\frac{1}{(q+1)!}\big)^{\frac{1}{1+q}} < 1$, we set $\alpha=1$ in view of \eqref{eq:Gamma-Alzer-bound}. Furthermore, we have
\begin{subequations}
\begin{align}
\Gamma(1+q,h\|A\|)
&\overset{\eqref{eq:proof-Gamma-set-x-p}}{=} 
\Gamma\big(1/p,x\big)
= \Gamma\big(1/p,(x^{1/p})^{p}\big)
\overset{\eqref{eq:Gamma-inc-repr}}{=} 
p\int_{x^{1/p}}^{\infty} e^{-t^{p}}\dd{t}
\\
&\overset{\eqref{eq:Gamma-Alzer-bound}, \alpha=1}{>} p\Gamma\big(1+1/p\big)\big(1-\big(1-e^{-x}\big)^{1/p}\big),
\intertext{and using
$p\,\Gamma\big(1+1/p\big)
\overset{\eqref{eq:proof-Gamma-set-x-p}}{=}
\frac{\Gamma(2+q)}{1+q}
= \frac{(1+q)!}{1+q}=q!$, since $q$ is integer,
}
&= q!\big(1-(1-e^{-x})^{(1+q)}\big).
\end{align}
\end{subequations}
Thus,
\begin{equation}
e^{x}\Big(1-\frac{\Gamma\big(1/p,x\big)}{q!}\Big)
< e^{x}\big(1-e^{-x}\big)^{(1+q)}
\end{equation}
which after substituting \eqref{eq:proof-Gamma-set-x-p} and in turn into \eqref{eq:proof-local-bound}, proves \eqref{eq:local-error-bound-b}.
\end{proof}
Theorem \ref{thm:local-error} enables to control the local integration error by choosing the step size $h$, using the simple form of the bound \eqref{eq:local-error-bound}, depending on the constants $\|a\|, \|A\|$ and the norm $\|V^{(k)}\|$ of the current iterate. Specifically, we choose
\begin{equation}\label{eq:hk-RK-linear}
h = h_{k} \qquad\text{such that}\qquad
\|V(t_{k+1})-V^{(k+1)}\| \leq \tau
\end{equation}
by \eqref{eq:local-error-bound}, for some prespecified value $\tau$. 
Inspecting the parametrization \eqref{eq:linear-ass-parametrization} shows that $\|V(t)\|$ grows -- and hence the step sizes \eqref{eq:hk-RK-linear} \textit{decrease} -- until $W(t)$ is close enough to a vertex of $\mc{W}$ (which represents a labeling) and satisfies a termination criterion that stops the chosen iterative RK scheme \eqref{eq:RK-Euclidean}. 

In order to check how large $\|V(t)\|$ then will be, assume
\begin{equation}
W_{i}=(\veps,\dotsc,\veps,1-(|J|-1)\veps,\veps,\dotsc,\veps) \in \R^{|J|}\qquad\text{and}\qquad \veps \ll \frac{1}{|J|-1} \leq 1,
\end{equation}
that is $W_{ij} \approx 1$ and $W_{il} \approx 0$ if $l \neq j$. Then with $W_{0i}=\eins_{\mc{S}}$ and by \eqref{eq:linear-ass-parametrization}, \eqref{eq:def-exp-1}
\begin{subequations}
\begin{align}
V_{i} &= \Exp_{\eins_{\mc{S}}}^{-1}(W_{i})
= \Pi_{\eins_{\mc{S}}}\big(\log W_{i}-\log\eins_{\mc{S}}\big)
= \frac{1}{|J|} \big(\log W_{i} - \frac{1}{|J|}\la\eins,\log W_{i}\ra\eins\big)
\\
\log W_{i} &\approx (\log\veps,\dotsc,\log\veps,0,\log\veps,\dotsc,\log\veps),\qquad
\frac{1}{|J|}\la\eins,\log W_{i}\ra
\approx \frac{|J|-1}{|J|}\log\veps
\approx \log\veps
\intertext{
and hence
}\label{eq:termination-rough}
\|V_{i}\|
&\approx \frac{1}{|J|}\log\frac{1}{\veps},
\qquad\qquad
\|V\| \approx \frac{|I|}{|J|}\log\frac{1}{\veps}.
\end{align}
\end{subequations}
Thus, as soon as the norm $\|V(t)\|$ has grown to the order $\log\frac{1}{\veps}$, a termination criterion that checks if $W(t)$ is $\veps$-close to some vertex of $\mc{W}$, will be satisfied.

Figure \ref{fig:RK-linear-boundFactor} quantitatively illustrates how much the factor $e^{h\|A\|}(1-e^{h\|A\|})^{(1+q)}$ of the upper bound \eqref{eq:local-error-bound} overestimates the exact factor \eqref{eq:proof-local-bound} computed in the proof of Thm.~\ref{thm:local-error}, and hence how conservative (i.e.~too small) the step size $h_{k}$ will be chosen to achieve \eqref{eq:hk-RK-linear}. The curves of Figure \ref{fig:RK-linear-boundFactor} show that the estimate \eqref{eq:local-error-bound} is fairly tight and suited to adapt the step size. Furthermore, comparing the ordinate values of both panels for $q=1$ and $q=4$ shows that, in order to achieve a fixed accuracy $\tau$ in \eqref{eq:hk-RK-linear}, using a \textit{higher}-order RK scheme \eqref{eq:RK-Euclidean} enables to choose a \textit{larger} step size.

\begin{figure}
\centerline{
\includegraphics[width=0.4\textwidth]{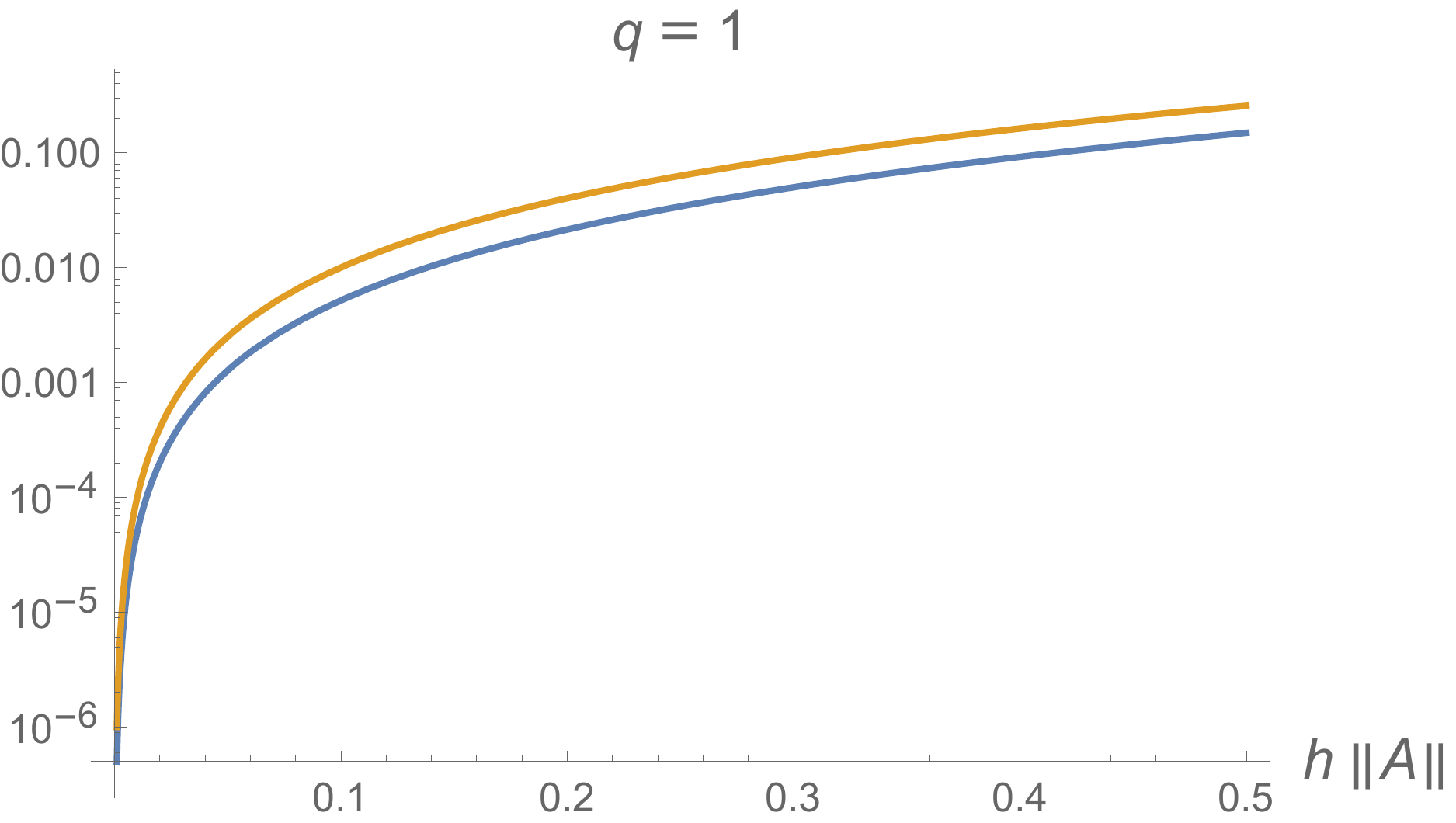}\hspace{0.05\textwidth}
\includegraphics[width=0.4\textwidth]{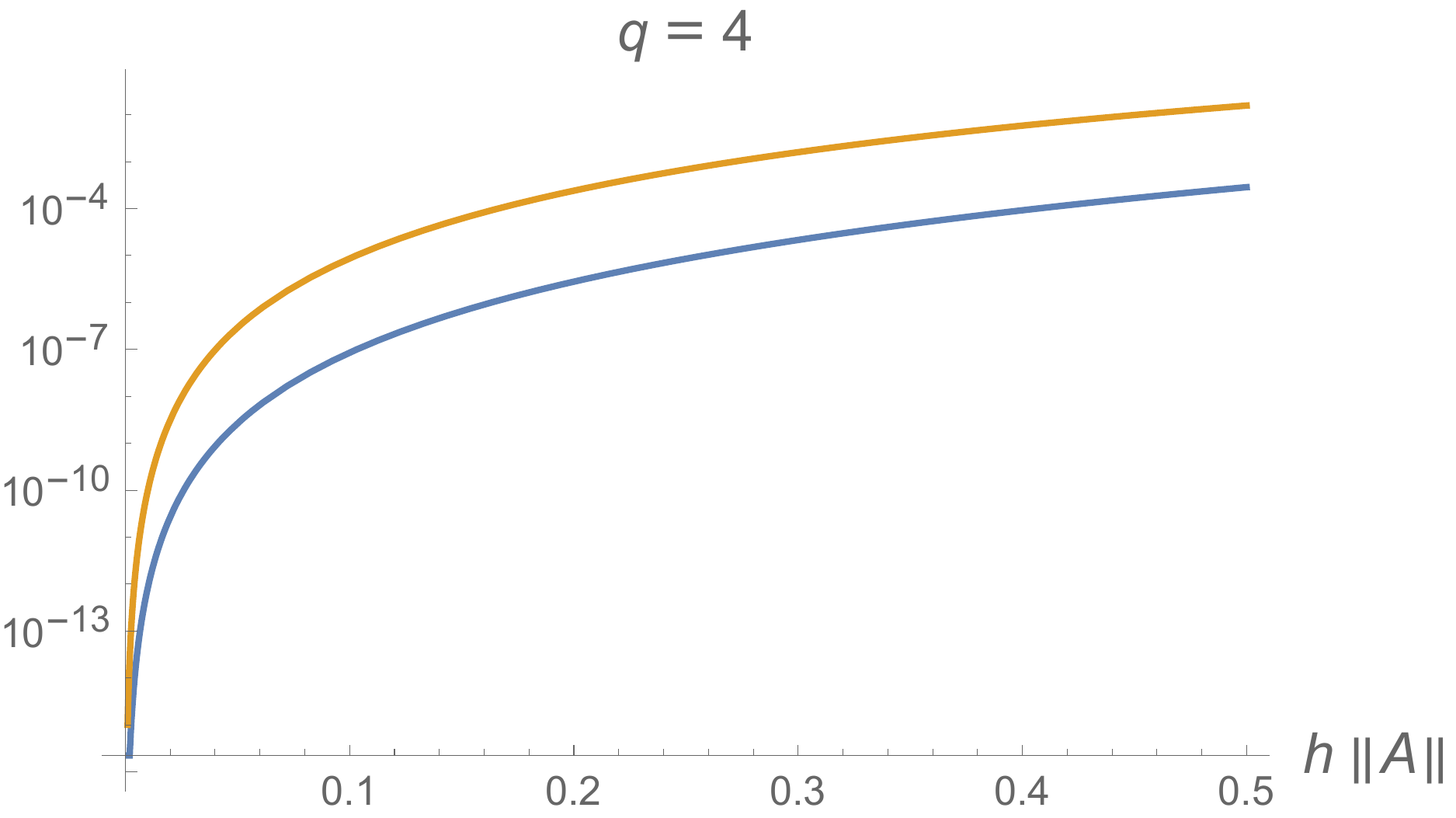}
}
\caption{
The factor $e^{h\|A\|}(1-e^{-h\|A\|})^{(1+q)}$ of the upper bound \eqref{eq:local-error-bound} (beige curve) and the numerically computed exact factor due to \eqref{eq:proof-local-bound} (blue line), as a function of $h\|A\|$, for $q=1$ (left panel) and $q=4$ (right panel). The \textit{relative} overestimation factor increases with the order $q$ and leads to more or less conservative step size choices \eqref{eq:hk-RK-linear}. Comparing the \textit{absolute} ordinate values of both panels shows that in order to achieve \eqref{eq:hk-RK-linear}, using a higher-order RK-scheme \eqref{eq:RK-Euclidean} enables to choose a larger step size.
}
\label{fig:RK-linear-boundFactor}
\end{figure}
\subsection{Nonlinear Assignment Flow}
\label{sec:step-size-nonlinear}

Similar to the preceding section, we wish to select step sizes $(h_{k})_{k \geq 0}$ in order to control the local error on the left-hand side of \eqref{eq:local-error-bound}. Because an estimate like the right-hand side of \eqref{eq:local-error-bound}  that is valid at each step $k$, is not available for the \textit{nonlinear} assignment flow, we adapt established \textit{embedded RK methods} \cite[Section II.V]{Hairer:2008aa} to the geometric RKMK schemes \eqref{eq:RKMK-algorithm}.

The basic strategy is to evaluate \textit{twice} step \eqref{eq:RKMK-V-update} 
\begin{equation}\label{eq:RK-embedded-V-hat-V}
V = h \sum_{j \in [s]} b_{j} \widetilde U^{j},\qquad\qquad
\widehat V = h \sum_{j \in [s]} \widehat b_{j} \widetilde U^{j},
\end{equation}
using a second collection of coefficients $\widehat b_{j},\, j \in s$, but with the \textit{same} vector fields $U^{i},\widetilde U^{i},\, i \in [s]$. Thus each embedded method can be specified by \textit{augmenting} the corresponding Butcher tableau accordingly,
\begin{center}
\begin{tabular}{c|ccccc}
$0$ & \\
$c_{2}$ & $a_{21}$ \\
$c_{3}$ & $a_{31}$ & $a_{32}$ \\
\vdots & \vdots & \vdots & $\ddots$ \\
$c_{s}$ & $a_{s1}$ & $a_{s2}$ & $\dotsb$ & $a_{s(s-1)}$ \\ \hline
& $b_{1}$ & $b_{2}$ & $\dotsb$ & $b_{s-1}$ & $b_{s}$ \\
& $\widehat b_{1}$ & $\widehat b_{2}$ & $\dotsb$ & $\widehat b_{s-1}$ & $\widehat b_{s}$
\end{tabular}
\end{center}
Proper embedded methods combine a pair of RK schemes of \textit{different} order $q, \widehat q$ so that $\|V-\widehat V\|$ indicates if the step size $h$ is small enough, at each step $k$ of the overall iteration. Since the vectors $U^{i},\widetilde U^{i},\, i \in [s]$ are used twice, this comes at little additional costs. We also point out that unlike the linear case, the magnitude $\|V\|$ of tangent vectors has much less influence, because the scheme \eqref{eq:RKMK-algorithm} that is consecutively applied at each step $k$, is based on \eqref{eq:ass-flow-dot-V} with the initial condition $V(0)=0$. As a consequence, the magnitude $\|V\|$ of the update \eqref{eq:RKMK-V-update} will be relatively small at each step $k$.

We list the tableaus of two embedded methods that we evaluate in Section \ref{sec:Experiments}.
\begin{center}
\parbox{0.175\textwidth}{\centering
\begin{tabular}{c|cc}
$0$ & \\ 
$1$ & $1$ \\ \hline
& $1$ & $0$ \\ \hline
& $1/2$ & $1/2$
\end{tabular}
} \hspace{0.1\textwidth}
\parbox{0.175\textwidth}{\centering
\begin{tabular}{c|ccc}
$0$ & \\
$1/3$ & $1/3$ \\
$2/3$ & $0$ & $2/3$ \\ \hline
& $1/4$ & $0$ & $3/4$ \\ \hline
& $1/3$ & $2/3$ & $0$
\end{tabular}
} \\
\parbox{0.175\textwidth}{\centering (\textbf{RK-1/2})}
\hspace{0.1\textwidth}
\parbox{0.175\textwidth}{\centering (\textbf{RK-3/2})}
\end{center}
The first method combines the forward Euler scheme and Heun's method of order $2$. The second method complements Heun's method of order $3$ as specified on \cite[p.~166]{Hairer:2008aa}. We call \textbf{RKMK-1/2} and \textbf{RKMK-3/2} the geometric versions of these schemes when they are applied in connection with \eqref{eq:RKMK-algorithm}.

We conclude this section by specifying the extension of \eqref{eq:RKMK-algorithm} in order to include adaptive step size control. Fix parameters $0<\tau \ll 1$, $n_{\tau} \in \N$ and set an initial sufficiently small step size $h=h_{0}$. At each step $k$, using the distance $d_{I}$ defined by \eqref{eq:def-DI}:
\begin{equation}\label{eq:RK-embedded}
\begin{minipage}{0.8\textwidth}
\begin{enumerate}[1.]
\item
Compute $U^{i},\widetilde U^{i},\, i \in [s]$.
\item 
Compute $V, \widehat V$ by \eqref{eq:RK-embedded-V-hat-V}.
\item
If $d_{I}(V,\widehat V) < \frac{\tau}{n_{\tau}}$, then increase the step size: $h \leftarrow 1.25 h$, compute the update $W(h)$ by \eqref{eq:RKMK-W-update}, set $W_{0} \leftarrow W(h)$ and proceed with the next iteration $k+1$ and step 1.
Otherwise continue with step 4.
\item
If $d_{I}(V,\widehat V) < \tau$, then keep the step size $h$, compute the update $W(h)$ by \eqref{eq:RKMK-W-update}, set $W_{0} \leftarrow W(h)$ and proceed with the next iteration $k+1$ and step 1. \\
Otherwise continue with step 5.
\item
Decrease the step size $h \leftarrow \frac{h}{2}$ and repeat iteration $k$, i.e.~continue with step 1.
\end{enumerate}
\end{minipage}
\end{equation}
Typical parameter values are $\tau=0.01, n_{\tau}=20$. Starting with an small initial step size $h_{0}$, the algorithm adaptively generates a sequence $(h_{k})$ whose values increase whenever the local error estimate is much smaller (by a factor $n_{\tau}$) than the prescribed tolerance $\tau$.

\section{Experiments and Discussion}
\label{sec:Experiments}

This section is organized as follows (see also Figure \ref{fig:paperStructure}).
We specify details of our implementation in Section \ref{sec:Implementation-Details}.
Section \ref{sec:exp-nonlinear} reports the evaluation of the geometric RKMK schemes \eqref{eq:RKMK-algorithm} with embedded step size control \eqref{eq:RK-embedded}, for integrating the \textit{nonlinear} assignment flow.
Section \ref{sec:exp-linear} is devoted to the \textit{linear} assignment flow: Assessment of how closely it approximates the nonlinear assignment flow, evaluation of the
RK schemes \eqref{eq:RK-Euclidean} with adaptive step size selection \eqref{eq:hk-RK-linear}, and evaluation of the exponential integrator introduced in Section \ref{sec:Exponential-Integrator}.

\subsection{Implementation Details}
\label{sec:Implementation-Details}

All algorithms were implemented and evaluated using Mathematica. In particular, we did \textit{not} apply any assignment normalization as suggested by \cite[Section 3.3.1]{Astrom:2017ac}, since Mathematica can work with arbitrary numerical precision. As a consequence, our results refute the claim of the authors of \cite{Bergmann:2017aa} that properties of the assignment flow are largely caused by this normalization step and a particular numerical scheme. Rather, our experiments illustrate \textit{intrinsic} properties of the assignment flow as well as the reliability and efficiency of a \textit{variety} of algorithms for integrating this flow numerically, conforming to the underlying geometry, as derived in the present paper.

Throughout the experiments, we used \textit{uniform} weights in \eqref{eq:def-SW} and \eqref{eq:def-geom-mean}, respectively, since how to choose these `control variables' in a proper way depending on the application at hand, is subject of our current work  and will be reported elsewhere. Yet, we point out that the algorithms of the present paper \textit{do} cover such more general scenarios. For example, the simplest geometric RKMK scheme \eqref{eq:RKMK-algorithm} was recently used for integrating the assignment flow in \textit{unsupervised} scenarios \cite{Zern:2018ac}, where labels evolve and hence distance vectors \eqref{eq:def-Di} no longer are static but vary with time $D = D(t)$, too.

\subsubsection{Ground Truth Flows} In order to obtain a baseline for assessing the performance of linearizations, of approximate numerical integration by various schemes or both, we always solved the assignment flow (nonlinear or linear) with high numerical accuracy using the geometric implicit Euler scheme (nonlinear flow) or the euclidean implicit Euler scheme (linear flow), with a sufficiently  small step size $h$. This requires to solve a fixed point equation as part of every iteration $k$, involving the nonlinear mapping on the right-hand side of \eqref{eq:ass-flow-dot-V} in case of the nonlinear assignment flow, or the linear mapping on the right-hand side of \eqref{eq:dot-V-ass-flow-linear} in case of the linear assignment flow. These fixed point equations were iteratively solved as well, and the corresponding iterations terminated when subsequent elements of the corresponding subsequences $(V^{k_{i}})_{i \geq 0}$ that measure the residual of the fixed point equation, satisfied
\begin{equation}\label{eq:def-DI}
d_{I}(V^{k_{i+1}},V^{k_{i}}) =
\frac{1}{|J|} \max_{i \in I} \|V_{i}^{k_{i+1}}-V_{i}^{k_{i}}\| \leq 10^{-8}.
\end{equation}
Starting these inner iterative loops with $V^{k_{0}}=V^{(k)}$ and terminating with $V^{k_{i,\mrm{end}}}$, we set $V^{(k+1)}=V^{k_{i,\mrm{end}}}$ and continued with the outer iteration $k+1$.

\subsubsection{Termination criterion}
As suggested by \cite{Astrom:2017ac},
all iterative numerical schemes generating sequences $(W^{(k)})$ were terminated when the average entropy of the assignment vectors dropped below the threshold
\begin{equation}\label{eq:termination-criterion}
-\frac{1}{|I||J|}\sum_{i \in I}\sum_{j \in J} W_{ij}^{(k)} \log W_{ij}^{(k)} < 10^{-3}.
\end{equation}
If this happens, then -- possibly up to a tiny subset of pixels $i \in I$ -- all assignment vectors $W_{i}$ are very close to a unit vector and hence almost uniquely indicate a labeling.

\begin{figure}
\centerline{
\includegraphics[width=0.2\textwidth]{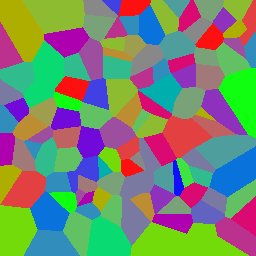}
\includegraphics[width=0.2\textwidth]{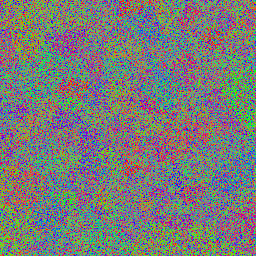}
\hspace{0.05\textwidth}
\includegraphics[width=0.2\textwidth]{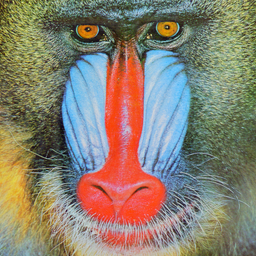}
\includegraphics[width=0.2\textwidth]{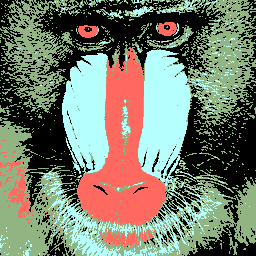}
}
\centerline{
\parbox{0.4\textwidth}{\centering\small (a)}
\hspace{0.05\textwidth}
\parbox{0.4\textwidth}{\centering\small (b)}
}
\caption{
Labeling scenarios for evaluating numerical schemes for integrating the assignment flow. (a) Computer-generated data with $31$ labels (left) and a noisy version used as input data (right). (b) A color image used as input data (left) using 4 color values as labels. These labels are illustrated on the right where each pixel has been replaced by the closest label.
}
\label{fig:labeling-scenarios}
\end{figure}
\subsubsection{Data}\label{sec:data}
Besides using the one-dimensional signal shown by Figure \ref{fig:signal-1D} that enables to visualize the entire evolution of the flow as plot in a single simplex (cf.~Figure \ref{fig:simplex-plots}), we used two further labeling scenarios for evaluating the numerical integration schemes.

Figure \ref{fig:labeling-scenarios}(a) shows a scenario adopted from \cite[Fig.~6]{Astrom:2017ac} using $\rho=0.1$ and $|\mc{N}_{i}|=7 \times 7$. The input data (right panel) comprise 31 labels encoded as vertices (unit vectors) of a corresponding simplex. This results in \textit{uniform} distances \eqref{eq:def-Di} and enables to assess in an unbiased way the effect of regularization by geometric diffusion in terms of the similarity map \eqref{eq:def-SW}.

Figure \ref{fig:labeling-scenarios}(b) shows a color image together with 4 color vectors used as labels, as illustrated by the panel on the right. In contrast to the data of Figure \ref{fig:labeling-scenarios}(a) with a high level of noise and a uniform data term (as motivated and explained above), the input data shown on the left of Figure \ref{fig:labeling-scenarios}(b) are not noisy but comprise spatial structures at quite different scales (fine texture, large homogeneous regions), causing a nonuniform data term and a more complex assignment flow.

Both scenarios together provide a testbed in order to  check and compare schemes for numerically integrating the assignment flow.

\begin{figure}
\centerline{
\includegraphics[width=0.2\textwidth]{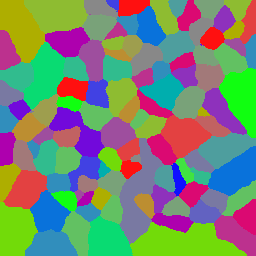}\hfill
\includegraphics[width=0.35\textwidth]{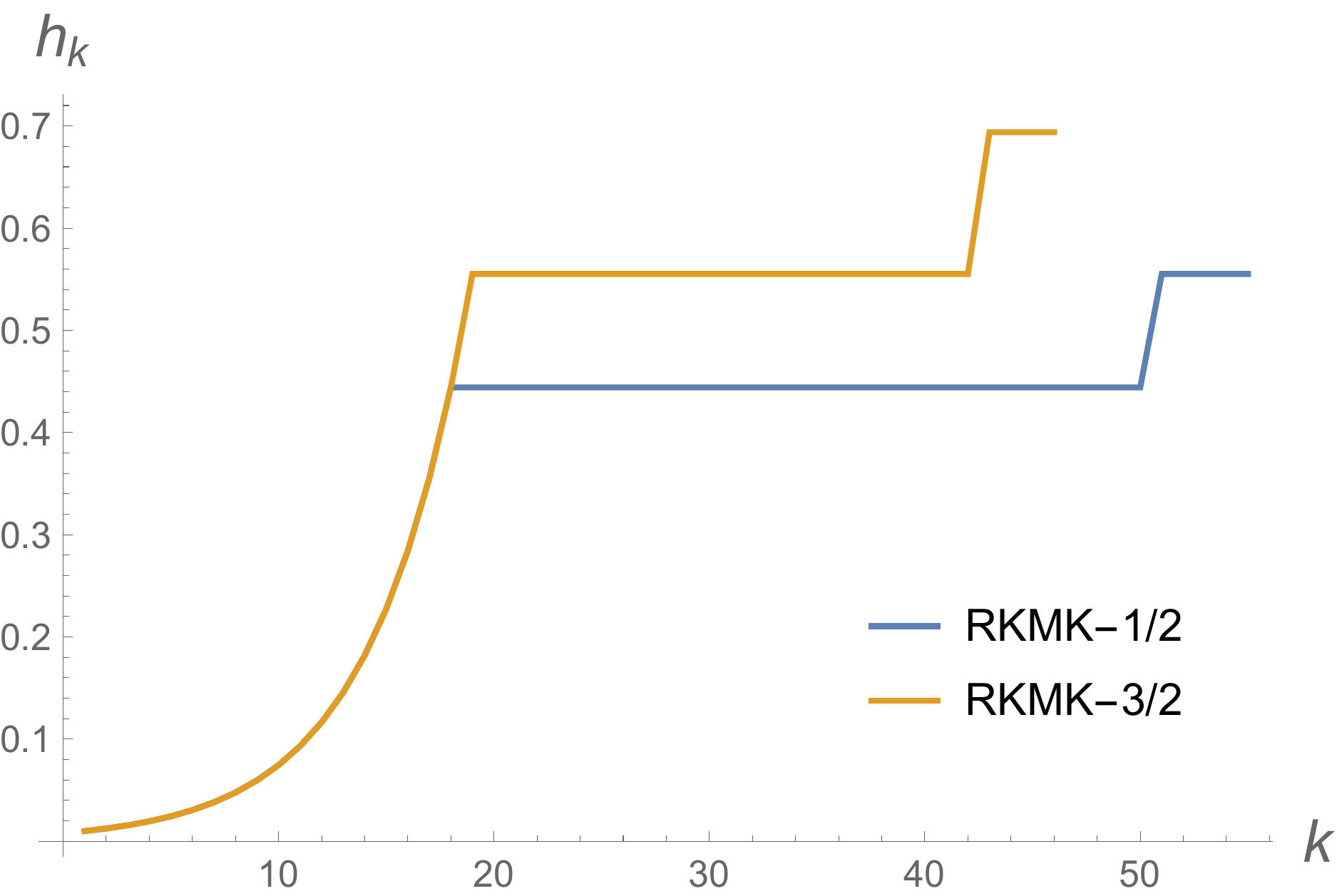}\hfill
\includegraphics[width=0.2\textwidth]{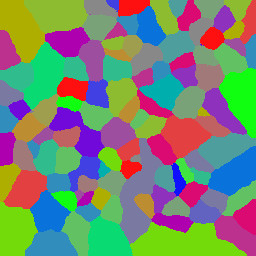}\hfill
\includegraphics[width=0.2\textwidth]{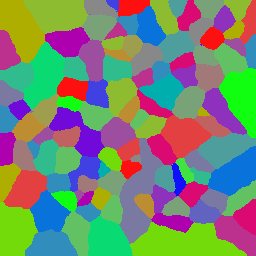}
}
\centerline{
\parbox{0.2\textwidth}{\centering\small (a)} \hfill
\parbox{0.35\textwidth}{\centering\small (b)} \hfill
\parbox{0.2\textwidth}{\centering\small (c)} \hfill
\parbox{0.2\textwidth}{\centering\small (d)}
}
\caption{
\textbf{Nonlinear assignment flow, embedded RKMK schemes.}
Results of processing the data shown by Fig.~\ref{fig:labeling-scenarios}(a), right panel (parameter: $\rho=0.1, |\mc{N}_{i}|=7 \times 7$).
(a) Ground truth labeling resulting from integrating the (full nonlinear) assignment flow using the \textit{implicit} geometric Euler scheme (step size $h=0.5$). (b) Sequences $(h_{k})$ of adaptive step sizes generated by the embedded geometric schemes RKMK-1/2 and RKMK-3/2 of Section \ref{sec:step-size-nonlinear}. The corresponding labeling results of these \textit{explicit} schemes are shown as panels (c) and (d). Because there is almost no difference to the ground truth result, both explicit schemes integrate the assignment flow sufficiently accurate.
}
\label{fig:embedded-RKMK-a}
\end{figure}
\begin{figure}
\begin{minipage}{0.64\textwidth}
\centerline{
\includegraphics[width=0.32\textwidth]{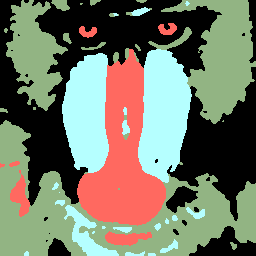}\hfill
\includegraphics[width=0.32\textwidth]{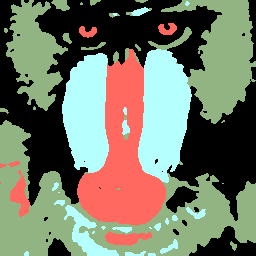}\hfill
\includegraphics[width=0.32\textwidth]{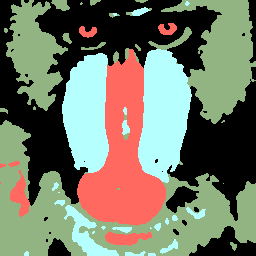}
}
\vspace{0.01\textwidth}
\centerline{
\includegraphics[width=0.32\textwidth]{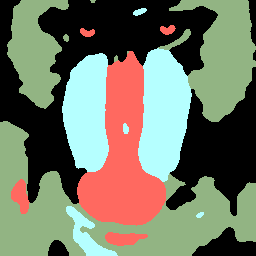}\hfill
\includegraphics[width=0.32\textwidth]{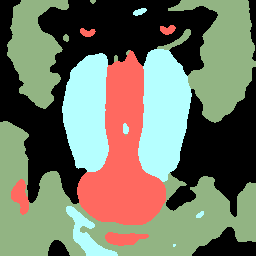}\hfill
\includegraphics[width=0.32\textwidth]{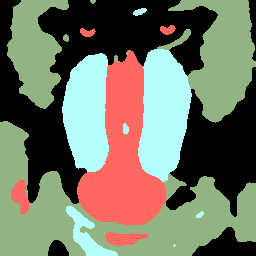}
}
\centerline{
\parbox{0.32\textwidth}{\centering\small ground truth}
\hfill
\parbox{0.32\textwidth}{\centering\small RKMK-1/2}
\hfill
\parbox{0.32\textwidth}{\centering\small RKMK-3/2}
}
\end{minipage}
\begin{minipage}{0.35\textwidth}
\includegraphics[width=\textwidth]{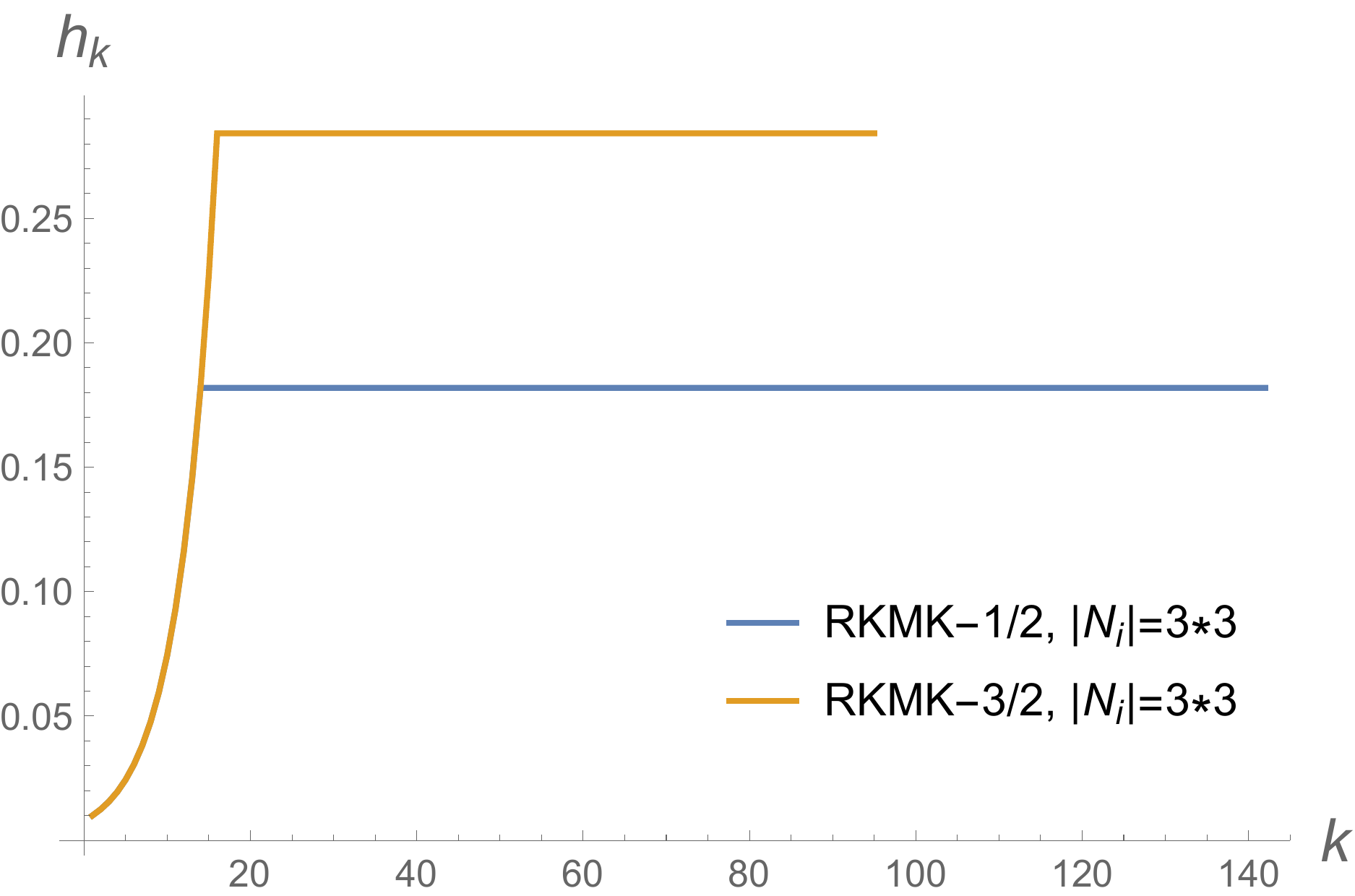}
\includegraphics[width=\textwidth]{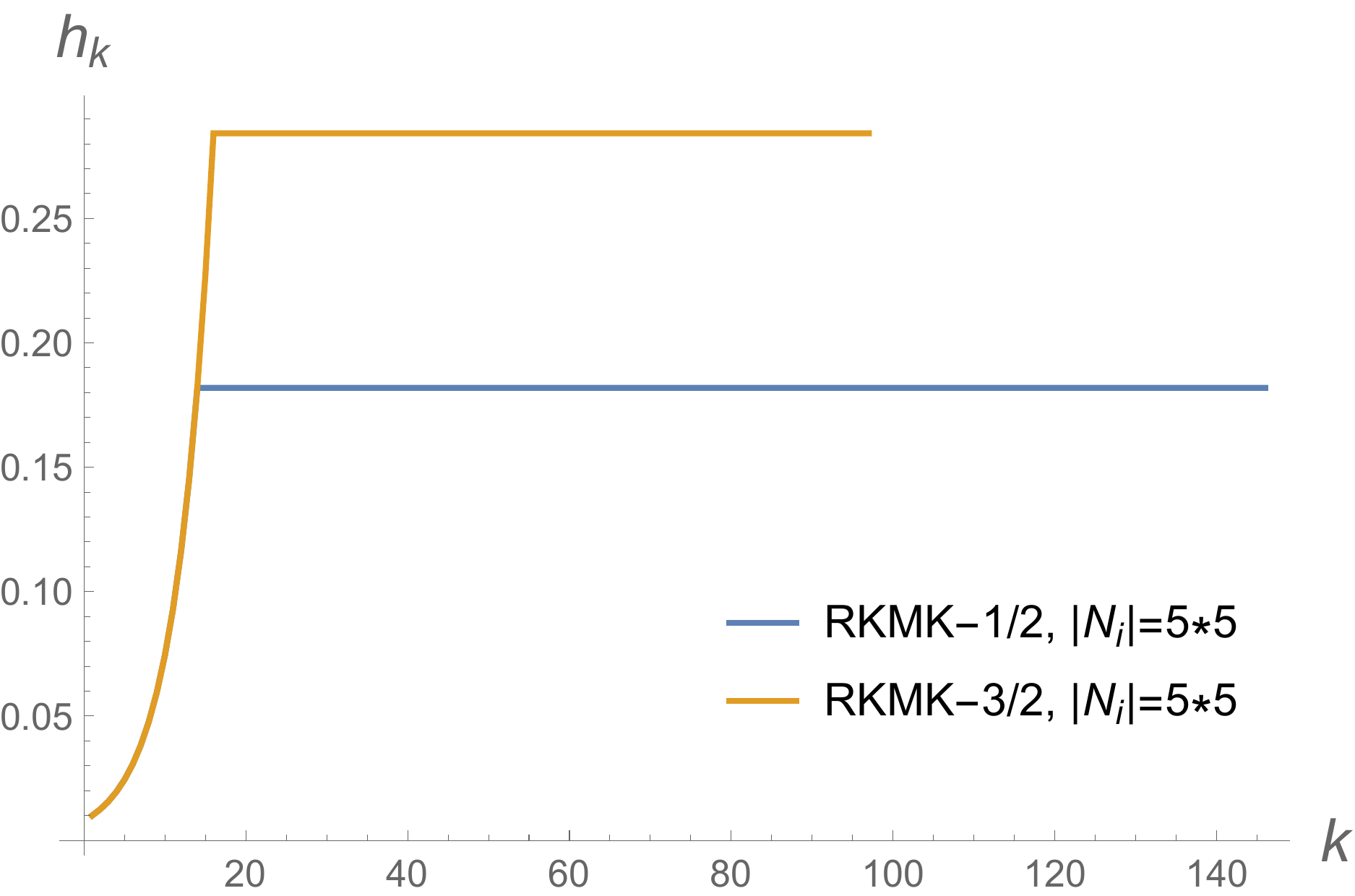}
\end{minipage}
\caption{
\textbf{Nonlinear assignment flow, embedded RKMK schemes.}
Results of processing the data shown by Fig.~\ref{fig:labeling-scenarios}(b), right panel (parameters: $\rho=0.5$, $|\mc{N}_{i}|=3 \times 3$ (less regularization; top row) and $|\mc{N}_{i}|=5 \times 5$ (more regularization; bottom row). The ground truth labeling was computed by integrating the (full nonlinear) assignment flow using the \textit{implicit} geometric Euler scheme (step size $h=0.5$). The embedded RKMK schemes (Section \ref{sec:step-size-nonlinear}) generated the adaptive step sizes shown in the panels on the right and almost identical labeling results.
}
\label{fig:embedded-RKMK-b}
\end{figure}
\subsection{Nonlinear Flow: Embedded RKMK-Schemes}
\label{sec:exp-nonlinear}

Figures \ref{fig:embedded-RKMK-a} and \ref{fig:embedded-RKMK-b} show the results of the two embedded RKMK schemes of Section \ref{sec:step-size-nonlinear} used to integrate the full nonlinear assignment flow \eqref{eq:assignment-flow}, for the data shown by right panels of Figure \ref{fig:labeling-scenarios} (a) and (b).

The two embedded RKMK schemes combine RKMK schemes of different approximation order $q/q'$, 1/2 and 3/2, respectively, which reuse vector field evaluations \eqref{eq:RKMK-algorithm} in order to produce sequences of tangent vectors $(V^{(k)}), (\hat V^{(k)})$ that enable to estimate the local approximation error by monitoring the distances $d_{I}(V^{(k)},\hat V^{(k)})$. As specified by \eqref{eq:RK-embedded}, step sizes $h_{k}$ adaptively increase provided a prescribed error tolerance is not violated.

The parameter values $\tau=0.01$ (tolerance) and $n_{\tau}=20$ (tolerance factor), used to produce the results shown by Figures \ref{fig:embedded-RKMK-a} and \ref{fig:embedded-RKMK-b}, suffice to integrate accurately the full nonlinear assignment flow by the respective \textit{explicit} schemes, as the comparison with the ground truth labeling generated by the implicit geometric Euler scheme shows. Since RKMK-3/2 has higher order $q$ than RKMK-1/2, larger step sizes can be tolerated (panel (b)). On the other hand, each iteration of RKMK-3/2 is about twice expensive as RKMK-1/2.

Both plots of the step size sequences $(h_{k})$ reveal that the initial step size $h_{0}=0.01$ was much too small (conservative), and that a fixed value of $h_{k}$ is adequate for most of the iterations. This value was larger for the experiment corresponding to Figure \ref{fig:embedded-RKMK-a} due to the uniform data term (by construction, as explained in Section \ref{sec:data}) and the more uniform scale of spatial structures. By contrast, the presence of spatial structures at quite different scales in the data corresponding to Figure \ref{fig:embedded-RKMK-b} causes a more involved assignment flow to be integrated, and hence to a smaller step size after adaption. Comparing the two rightmost panels of Figure \ref{fig:embedded-RKMK-b} shows that the strength of regularization (neighborhood size $|\mc{N}_{i}|$) had only little influence on the sequences of step sizes.

We never observed decreasing step sizes in these \textit{supervised} scenarios, that is step 5.~of \eqref{eq:RK-embedded} never was active. This may change in more involved scenarios, however (cf.~Remark \ref{rem:non-basic-cases}).

Overall, a few dozens of explicit iterations suffice for accurate geometric numerical integration of the assignment flow. Each iteration may be implemented in a fine-grained parallel way and has computational costs roughly equivalent to a convolution, besides mapping to the tangent space $\mc{T}_{0}$ and back to the assignment manifold $\mc{W}$, at each iteration.

\subsection{Linear Assignment Flow}
\label{sec:exp-linear}

The approach of Section \ref{sec:ExponentialIntegrators} involves two different approximations:
\begin{enumerate}[(i)]
\item
the \textit{linear} assignment flow \eqref{eq:ass-flow-linear} approximating the \textit{full} assignment flow \eqref{eq:assignment-flow}, and
\item
the numerical integration of the linear assignment flow using two alternative numerical schemes:
\begin{enumerate}[(a)]
\item
adaptive RK schemes (Section \ref{sec:step-size-linear}) based on the parametrization of Prop.~\ref{prop:linear-ass-parametrization} and
\item
the exponential integrator (Section \ref{sec:Exponential-Integrator}).
\end{enumerate}
\end{enumerate}
Due to the remarkable approximation properties of the linear assignment flow when a \textit{single} linearization at the barycenter is only used (Section \ref{sec:Approximation-Property}), we entirely focused on this flow when evaluating the numerical schemes (a) and (b) in Sections \ref{sec:exp-RK-adaptive} and \ref{sec:exp-exponential-integrator}.

\begin{figure}
\centerline{
\includegraphics[width=0.325\textwidth]{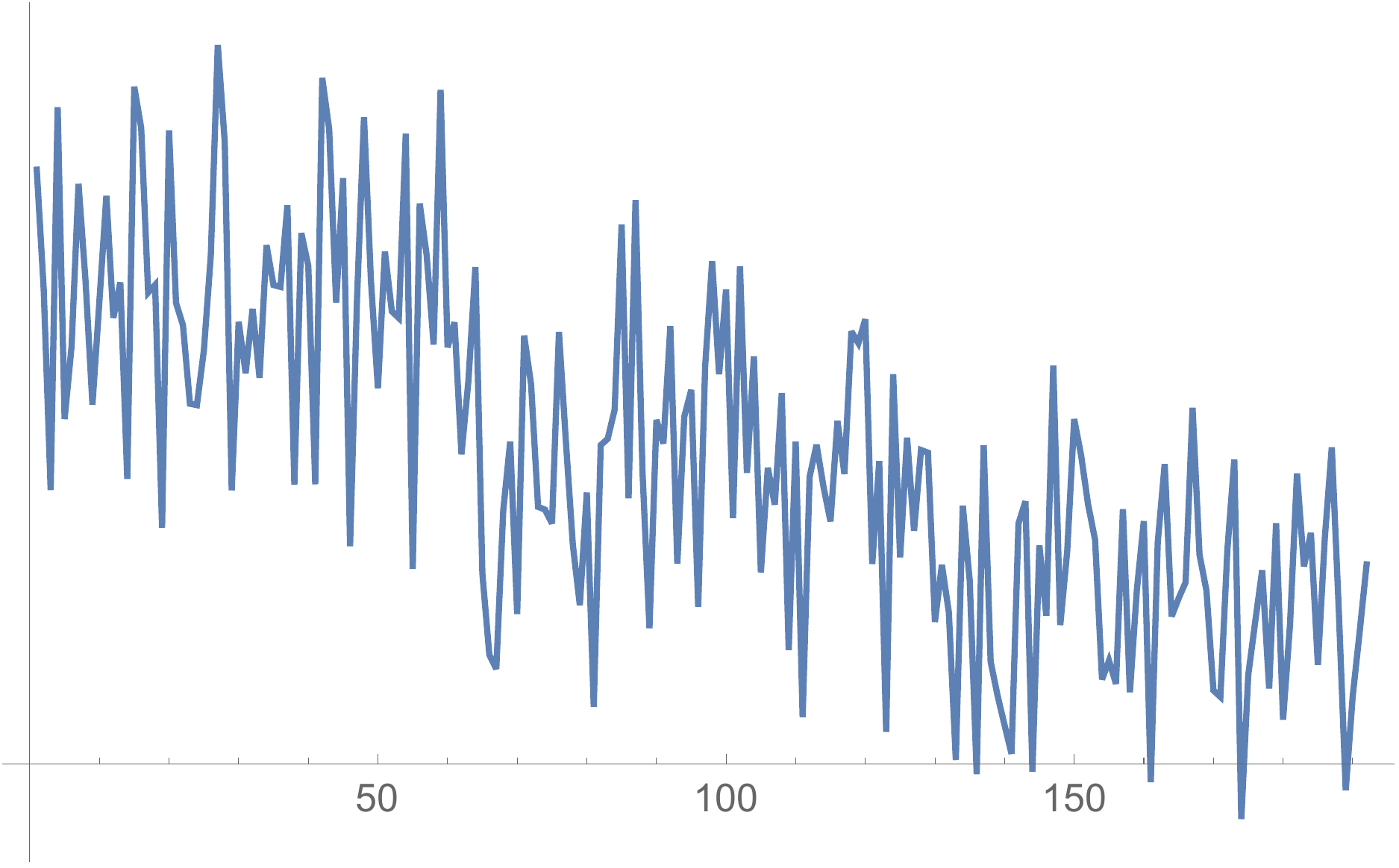}
\hfill
\includegraphics[width=0.325\textwidth]
{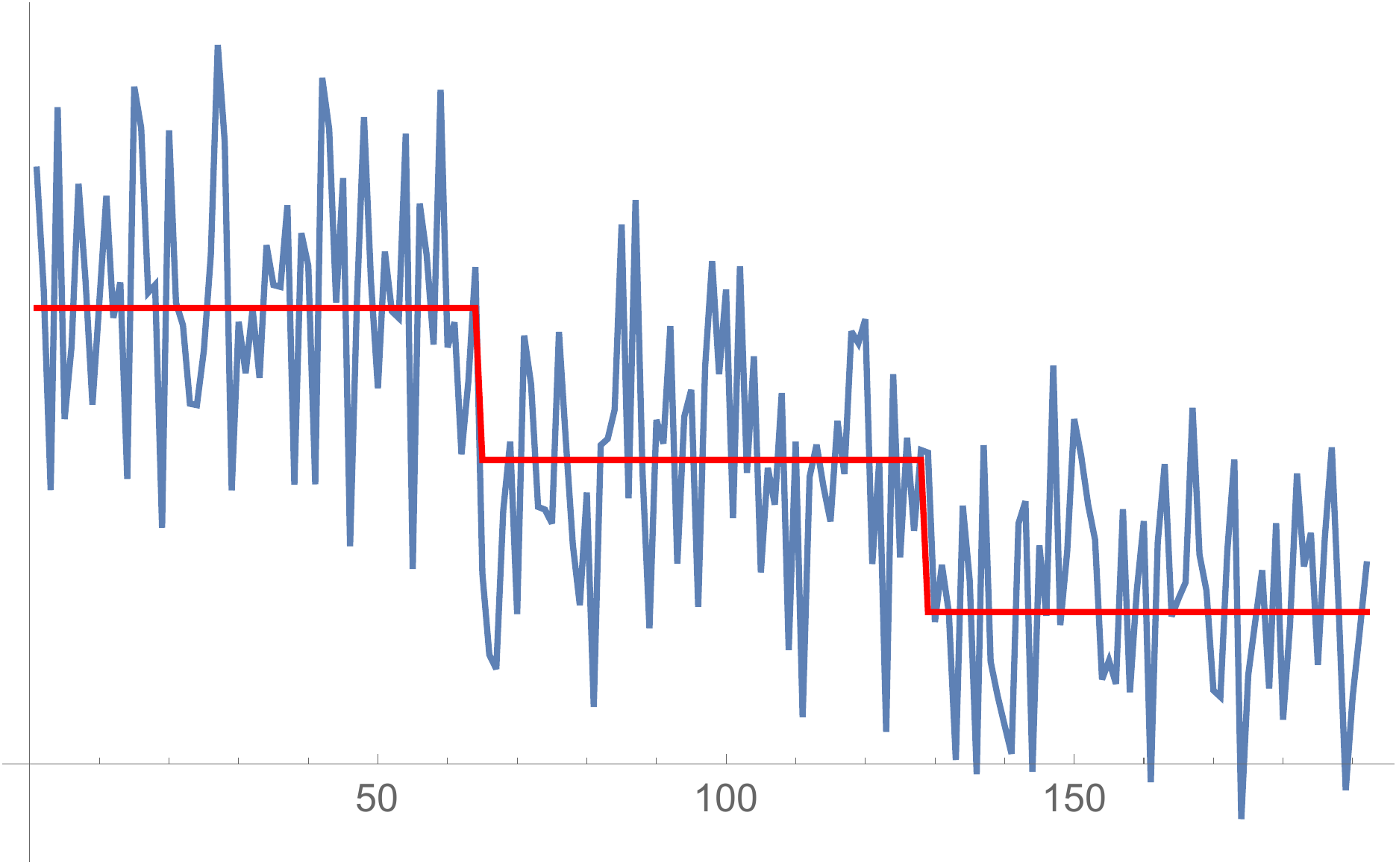}
\hfill
\includegraphics[width=0.325\textwidth]
{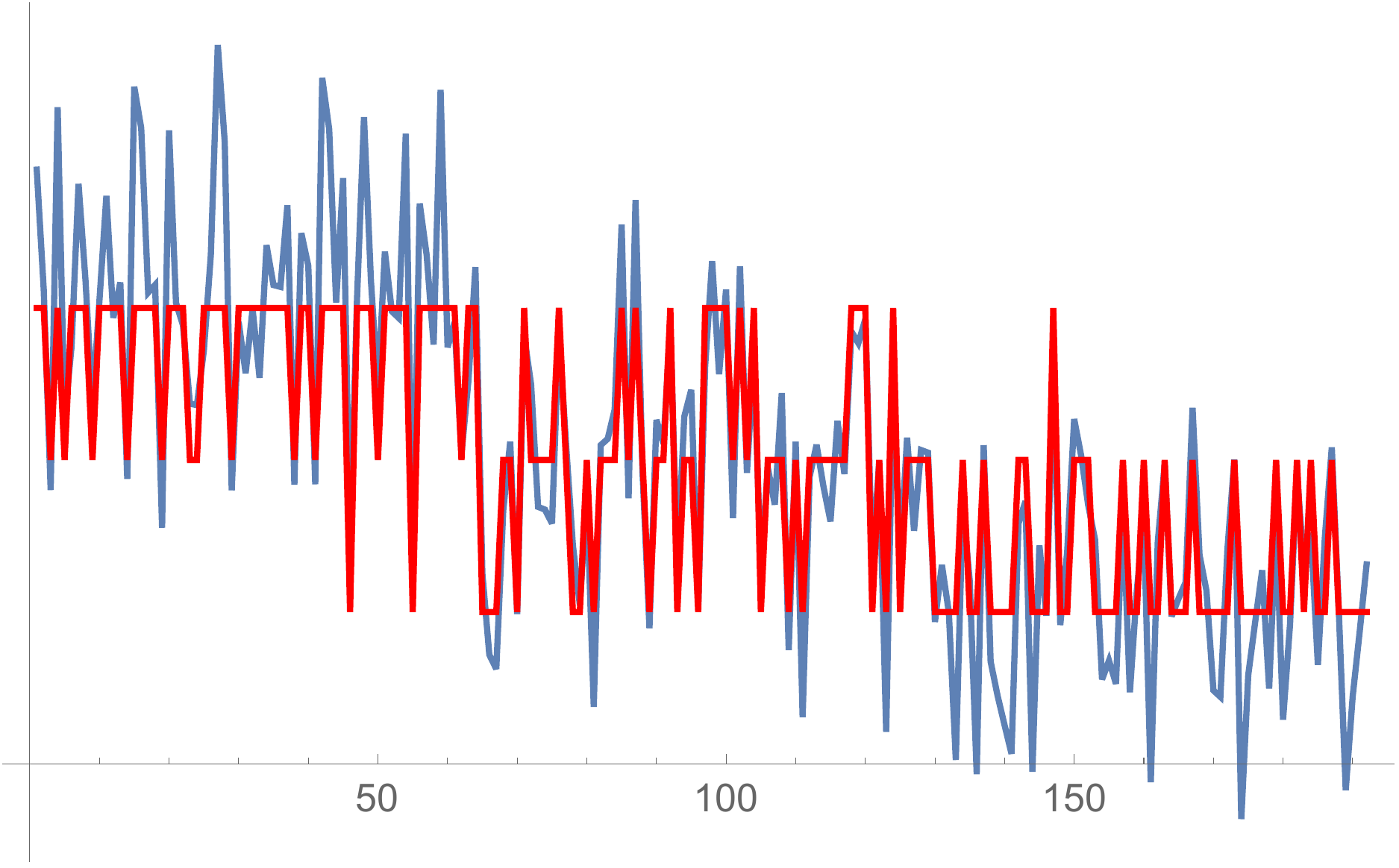}
}
\parbox{0.325\textwidth}{\centering\small (a)}\hfill
\parbox{0.325\textwidth}{\centering\small (b)}\hfill
\parbox{0.325\textwidth}{\centering\small (c)}
\caption{
(a) A noisy 1D signal used for the experiments of Section \ref{sec:Approximation-Property}. (b) The piecewise constant signal (red) used to generate (a) by superimposing noise. (c) \textit{Local} rounding to the next label and comparing to (b) indicates the noise level. Local rounding is equivalent to omitting regularization in the assignment flow by replacing the interacting similarity vectors $S(W)$ in \eqref{eq:assignment-flow} by the non-interacting likelihood vectors $L(W)$ of \eqref{eq:def-LW}.
}
\label{fig:signal-1D}
\end{figure}
\subsubsection{Approximation Property}
\label{sec:Approximation-Property}
We report a series of experiments for the 1D signal depicted by  Figure \ref{fig:signal-1D} using both the full and the linear assignment flow in order to check how closely the latter approximates the former. Then we discuss the linear assignment flow for the two 2D scenarios shown by Figure \ref{fig:labeling-scenarios}.

The parameter value $\rho=0.1$ for scaling in \eqref{eq:def-LW} the data, for all 1D experiments discussed below. This gave a larger weight to the `data term' so that -- in view of the noisy data (Fig.~\ref{fig:signal-1D}) -- the regularization property of the assignment flow \eqref{eq:assignment-flow}, in terms of the similarity vectors $S_{i}(W)$ interacting through \eqref{eq:def-SW}, was essential for labeling.

\begin{figure}
\centerline{
\includegraphics[width=0.325\textwidth]{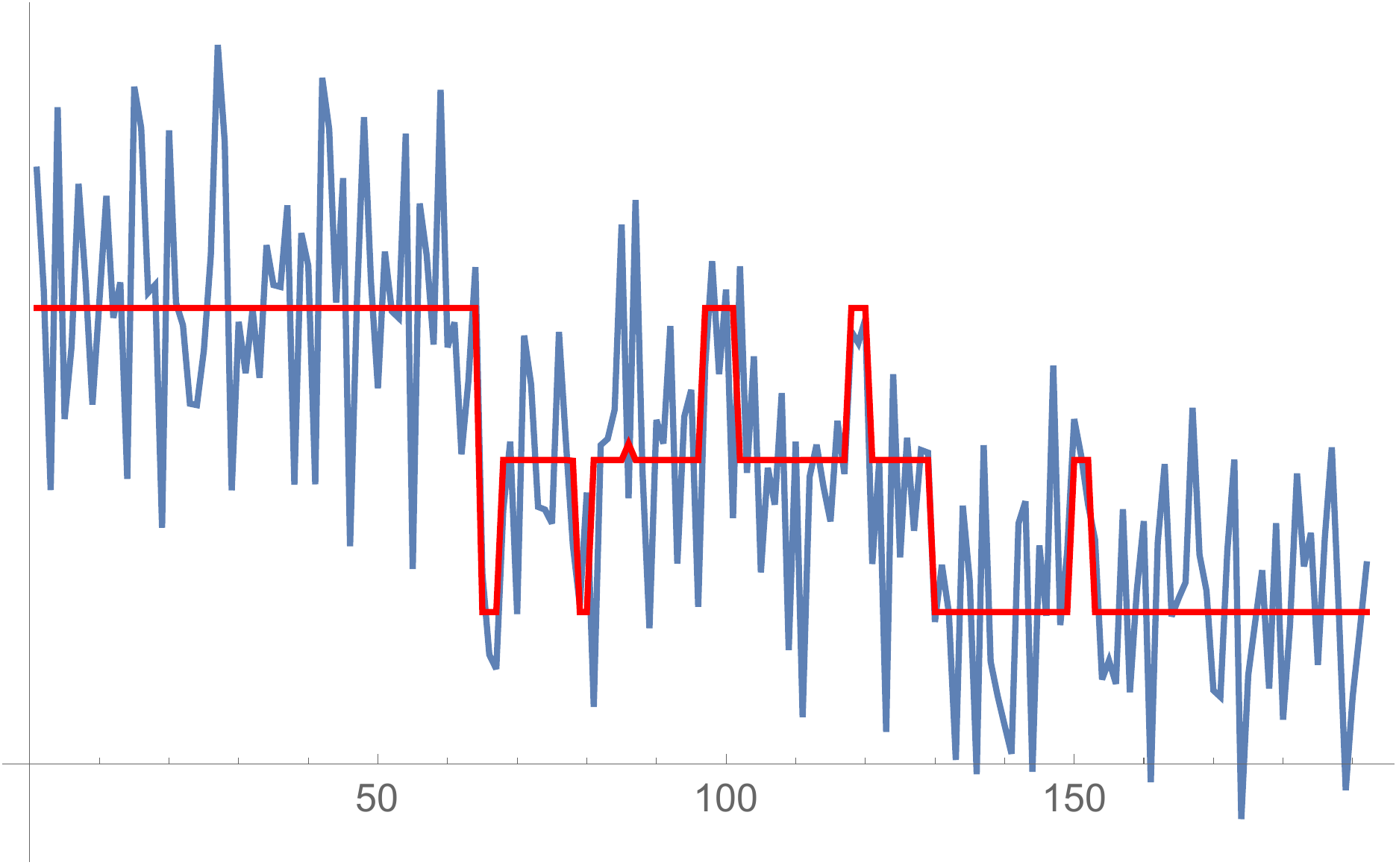}\hfill
\includegraphics[width=0.325\textwidth]{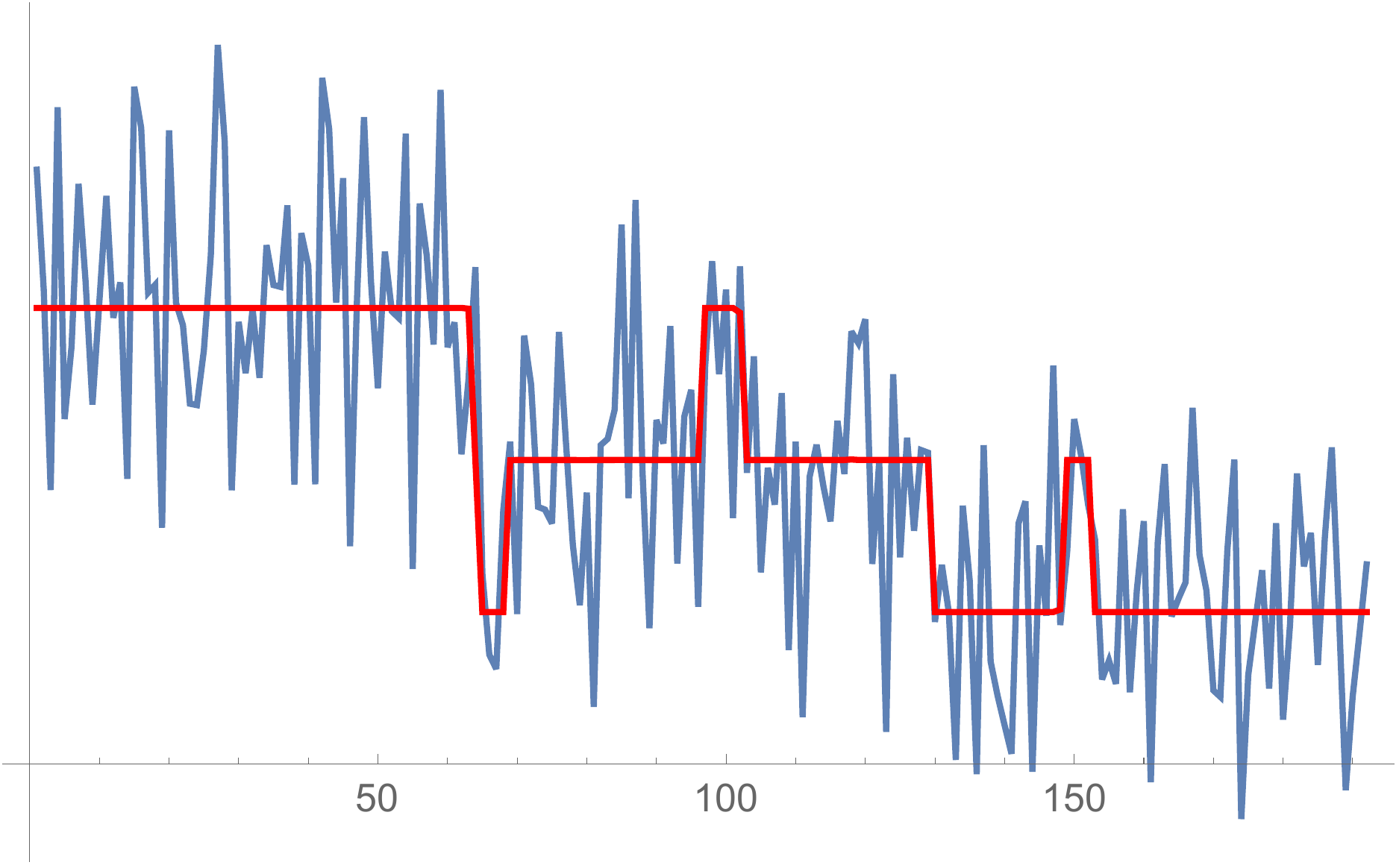}\hfill
\includegraphics[width=0.325\textwidth]{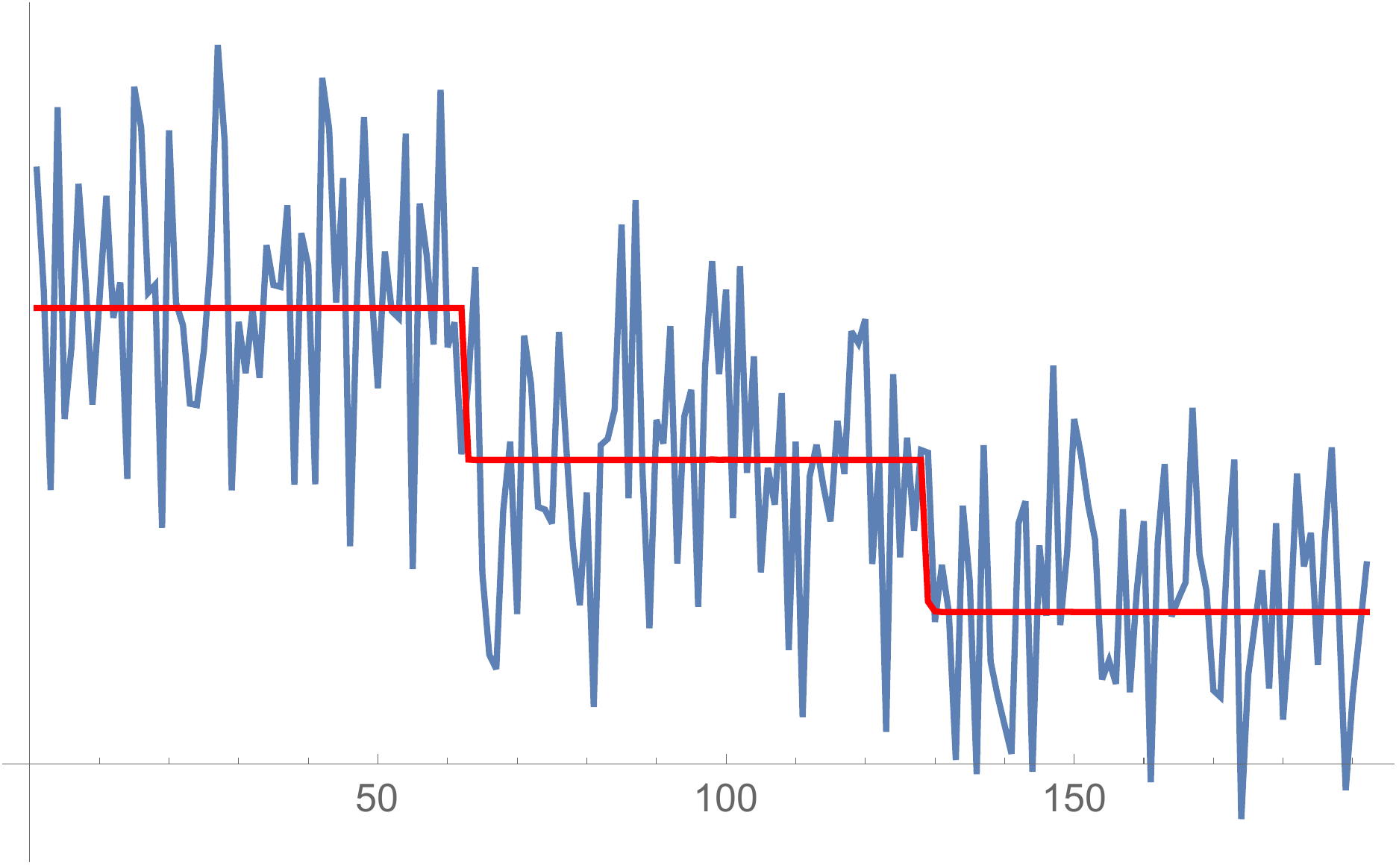}
}
\parbox{0.325\textwidth}{\centering\small $|\mc{N}_{i}|=3\;\text{(nonlinear flow)}$}\hfill
\parbox{0.325\textwidth}{\centering\small $|\mc{N}_{i}|=5\;\text{(nonlinear flow)}$}\hfill
\parbox{0.325\textwidth}{\centering\small $|\mc{N}_{i}|=9\;\text{(nonlinear flow)}$}
\centerline{
\includegraphics[width=0.325\textwidth]{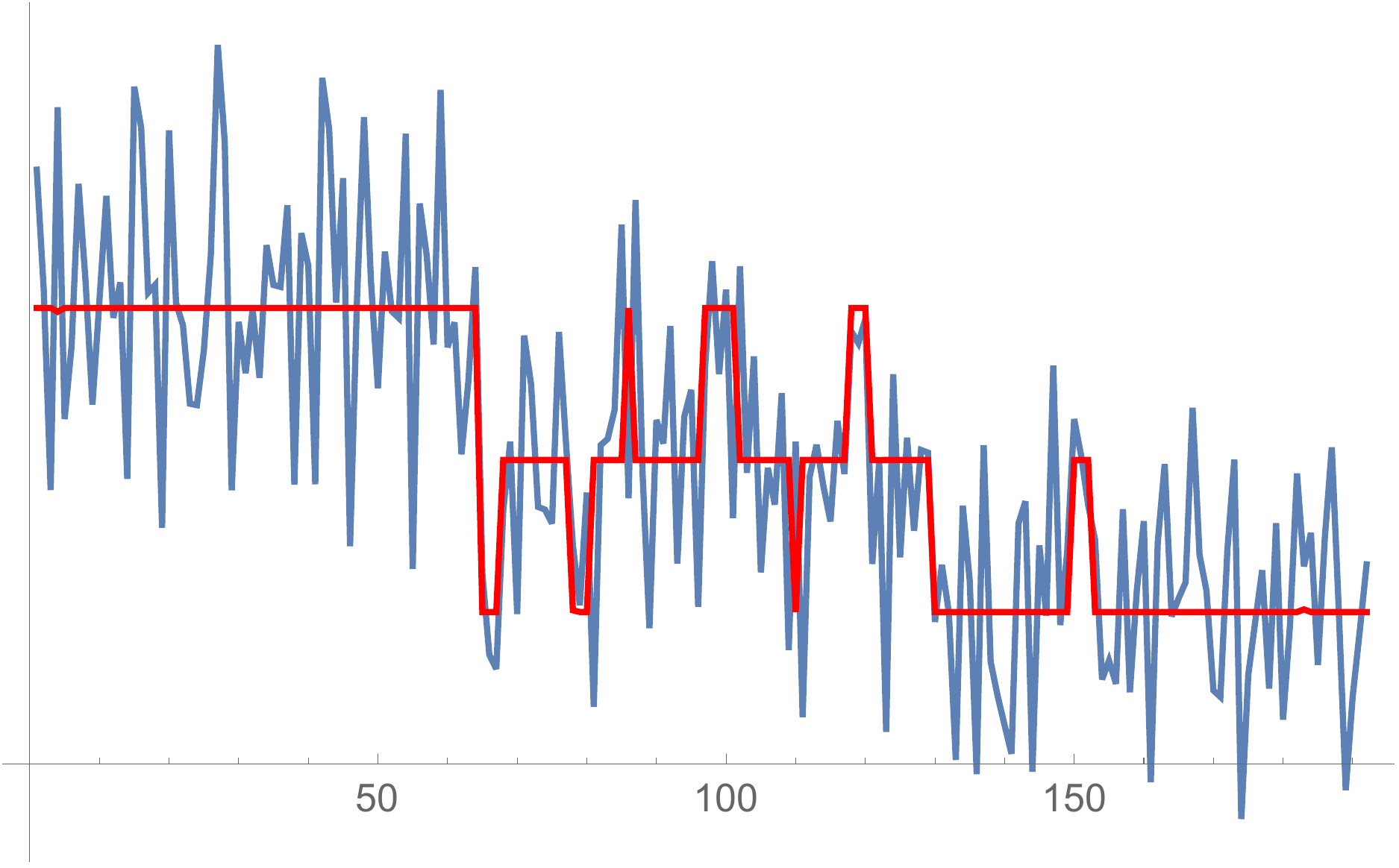}\hfill
\includegraphics[width=0.325\textwidth]{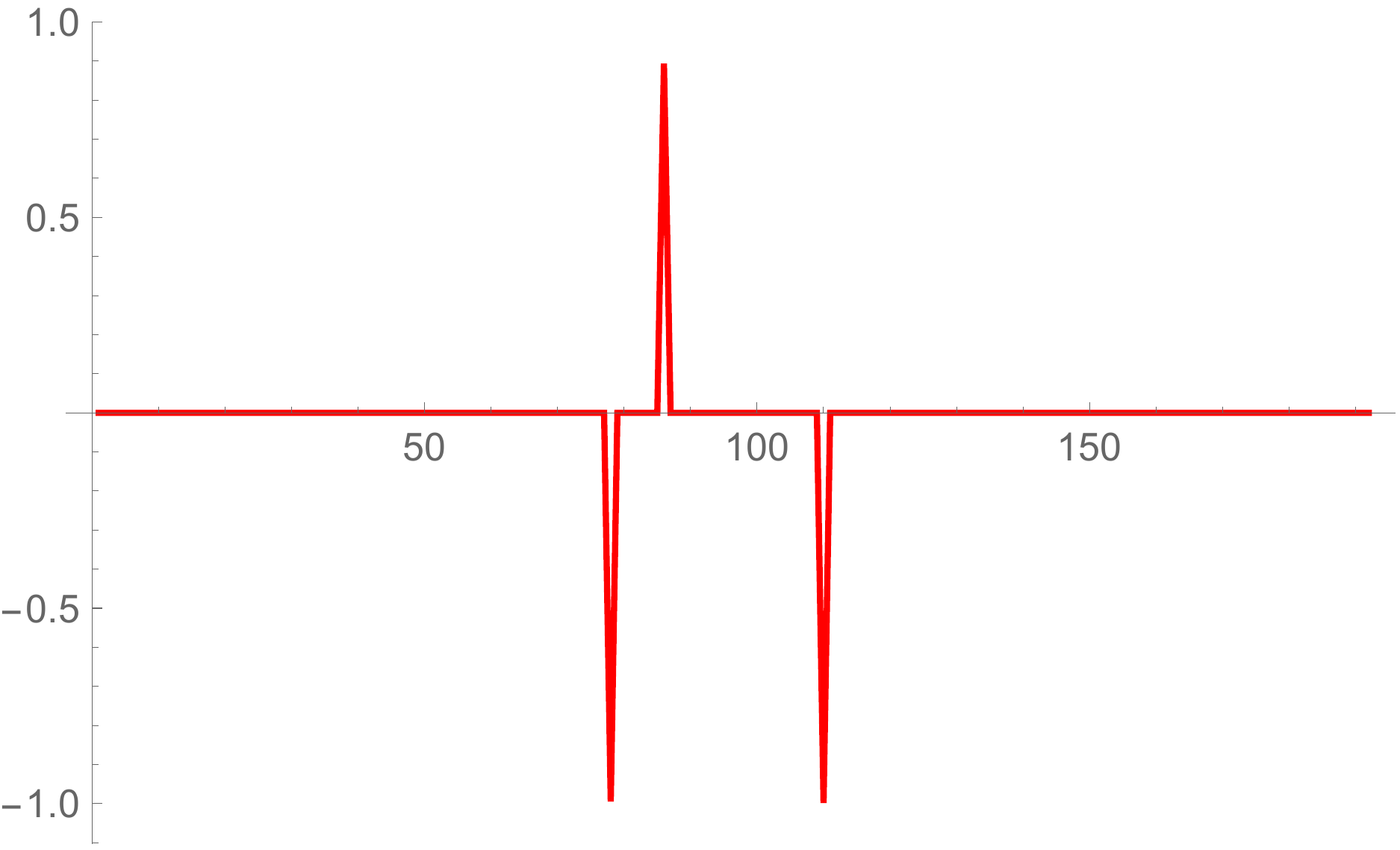}\hfill
\includegraphics[width=0.325\textwidth]{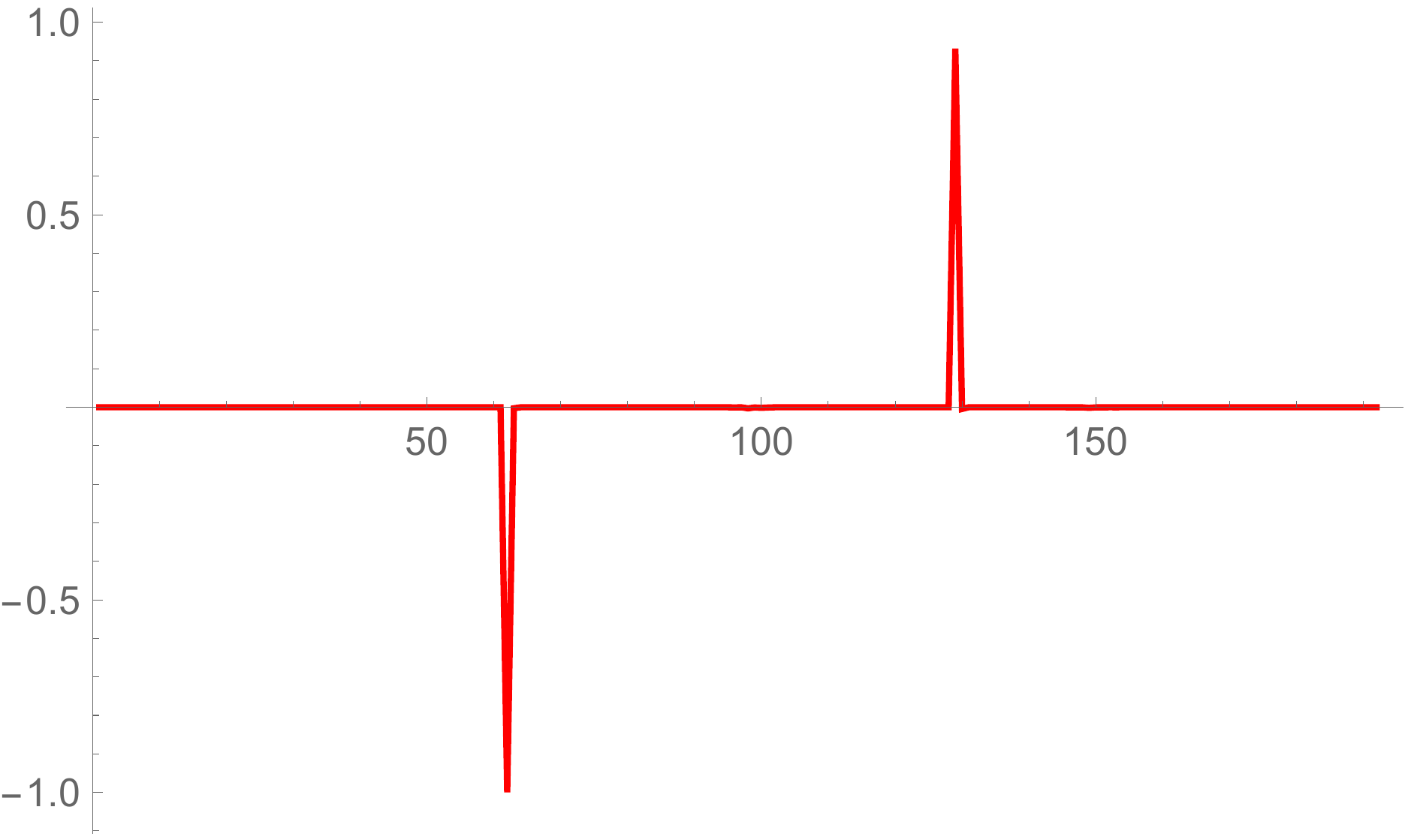}
}
\parbox{0.325\textwidth}{\centering\small $|\mc{N}_{i}|=3$ (linear flow)}\hfill
\parbox{0.325\textwidth}{\centering\small labeling error ($|\mc{N}_{i}|=3$)}\hfill
\parbox{0.325\textwidth}{\centering\small labeling error ($|\mc{N}_{i}|=9$)}
\caption{
\textbf{Nonlinear vs. linear assignment flow.}
\textsc{top row:} Labelings determined by the assignment flow as a reference for the linear assignment flow, using different neighborhood sizes $|\mc{N}_{i}|$. \textsc{bottom row:} Labeling determined by the linear assignment flow that differs in $3$ pixels from the corresponding above result of the full assignment flow. These errors and the two errors in the case $|\mc{N}_{i}|=9$ are shown in the center and right panel. These and further results listed as Table \ref{tab:linear-assignment-flow} show that the linear assignment flow achieves high-quality labelings.
}
\label{fig:1D-labeling-error}
\end{figure}
\begin{figure}
\centerline{
\includegraphics[width=0.275\textwidth]{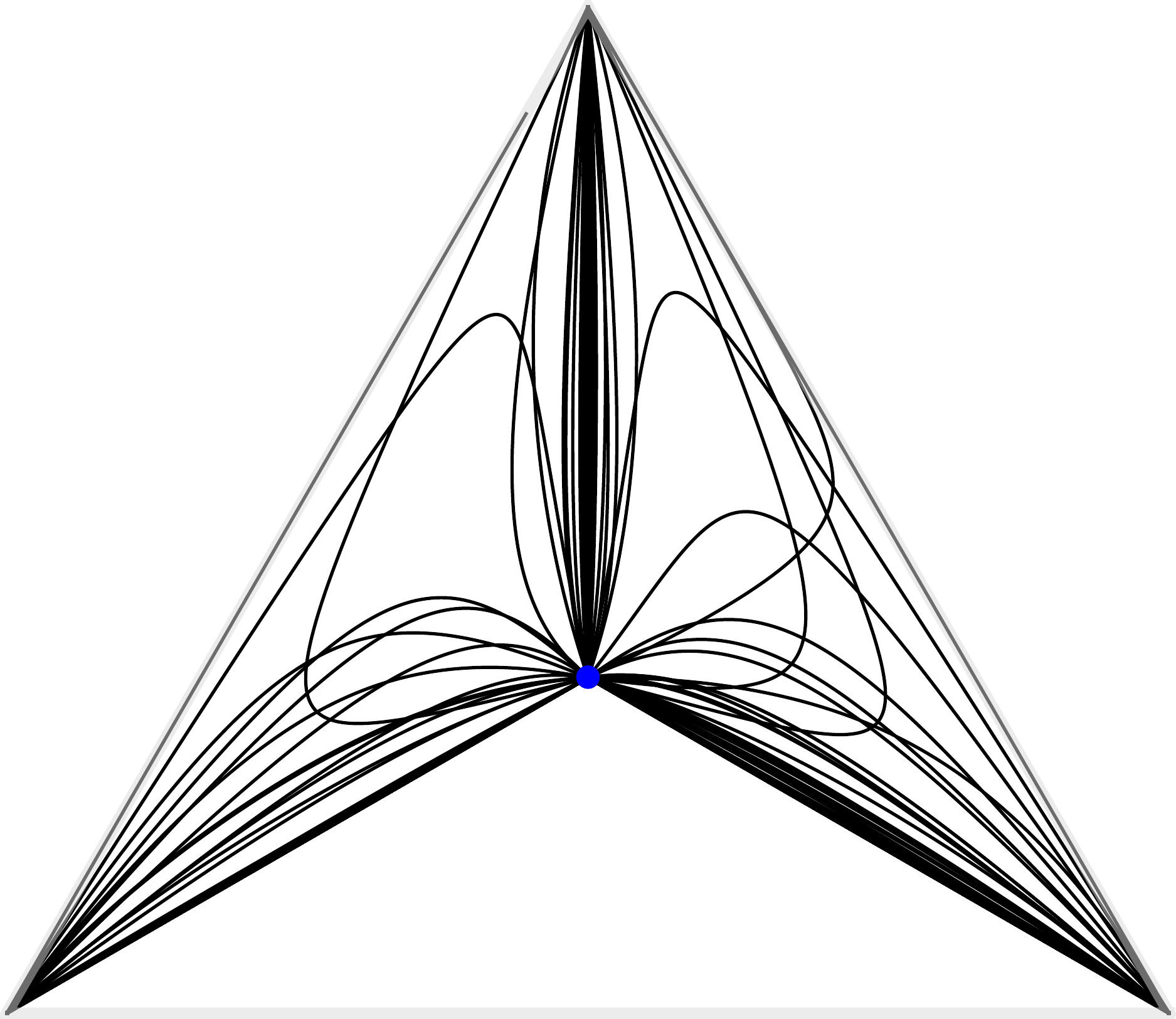}
\hspace{0.02\textwidth}
\includegraphics[width=0.275\textwidth]{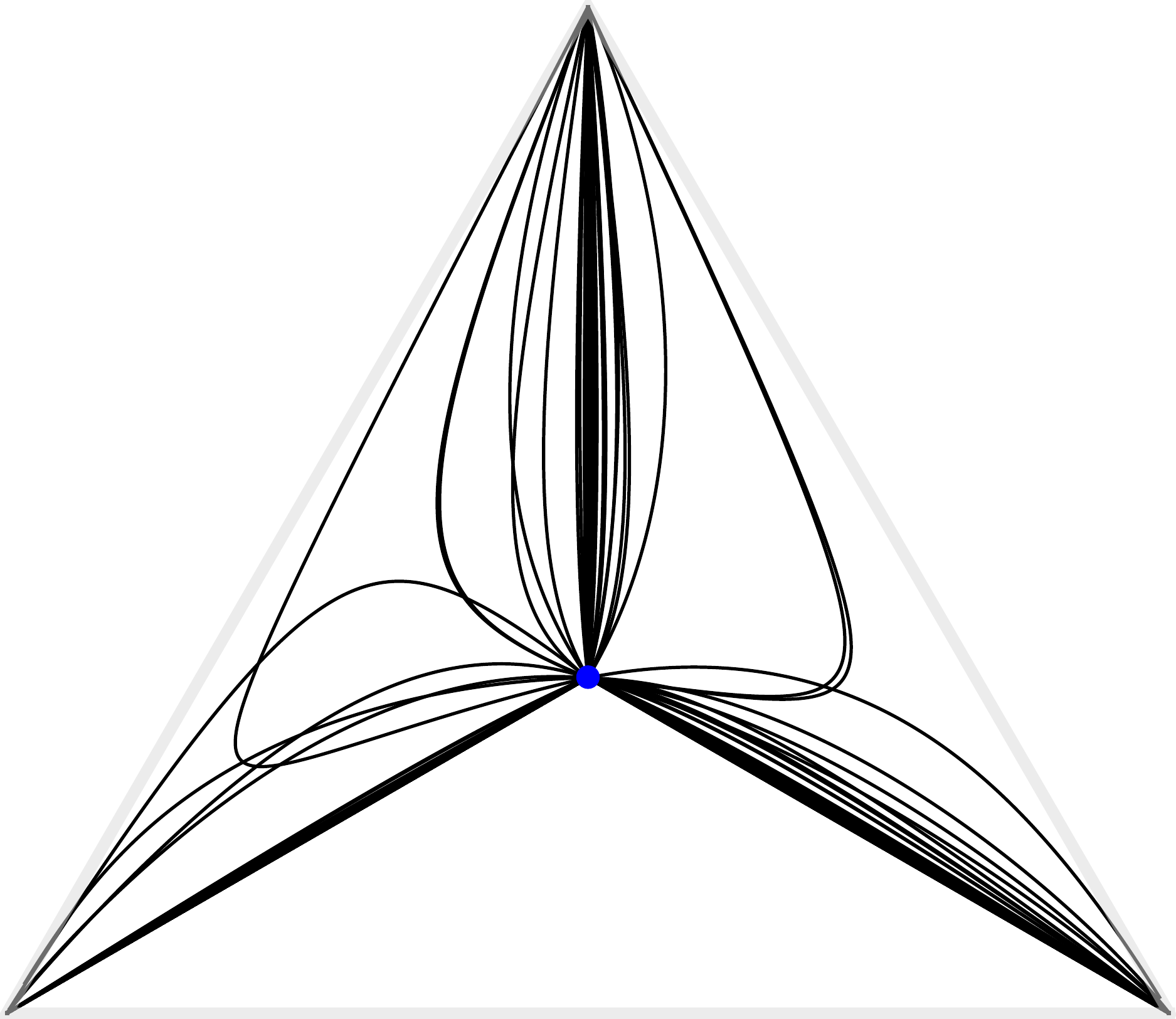}
\hspace{0.02\textwidth}
\includegraphics[width=0.275\textwidth]{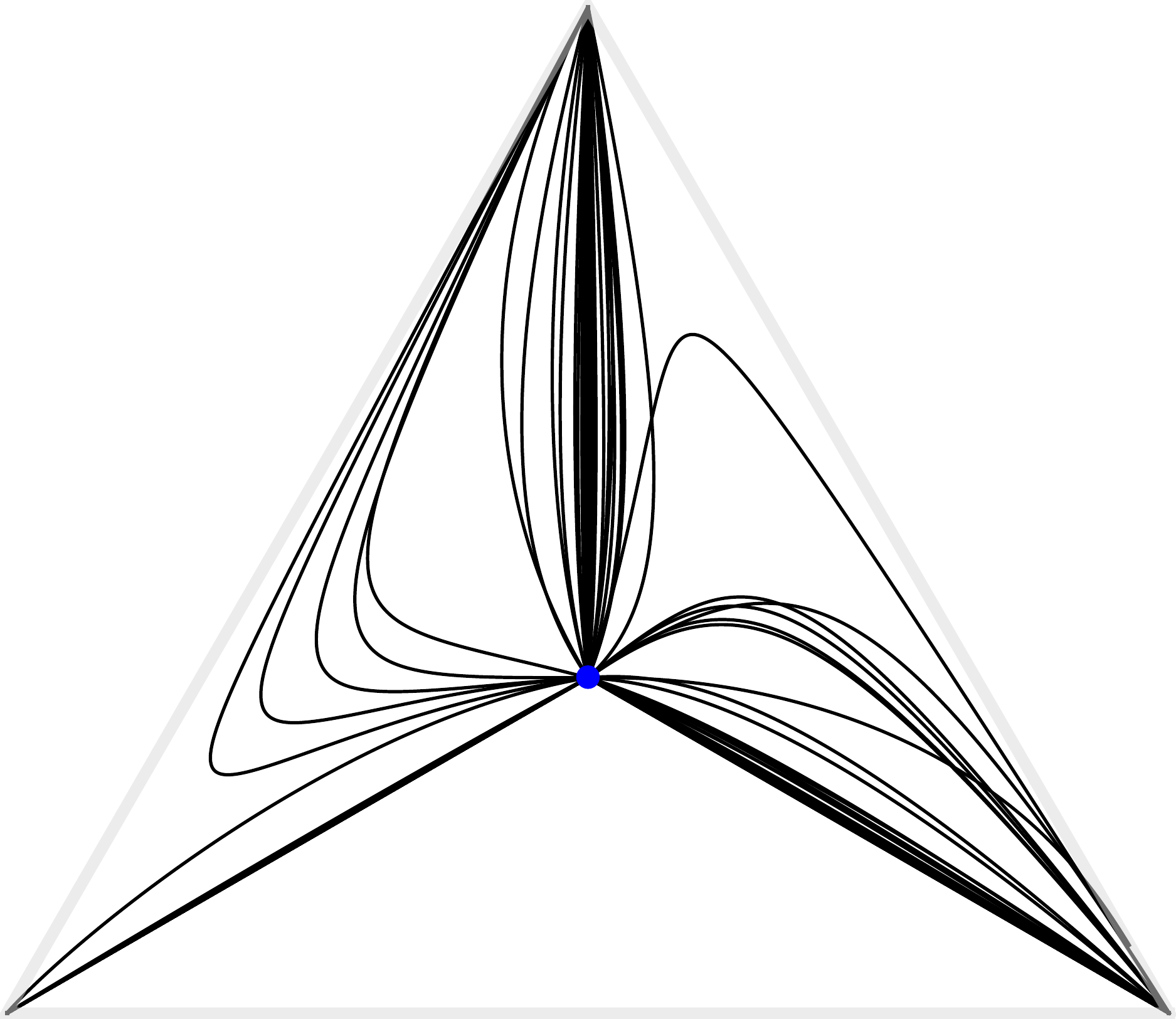}
}
\parbox{0.275\textwidth}{\centering\small $|\mc{N}_{i}|=3$}\hspace{0.02\textwidth}
\parbox{0.275\textwidth}{\centering\small $|\mc{N}_{i}|=5$}\hspace{0.02\textwidth}
\parbox{0.275\textwidth}{\centering\small $|\mc{N}_{i}|=9$}
\parbox{\textwidth}{\centering\small nonlinear assignment flow}
\centerline{
\includegraphics[width=0.275\textwidth]{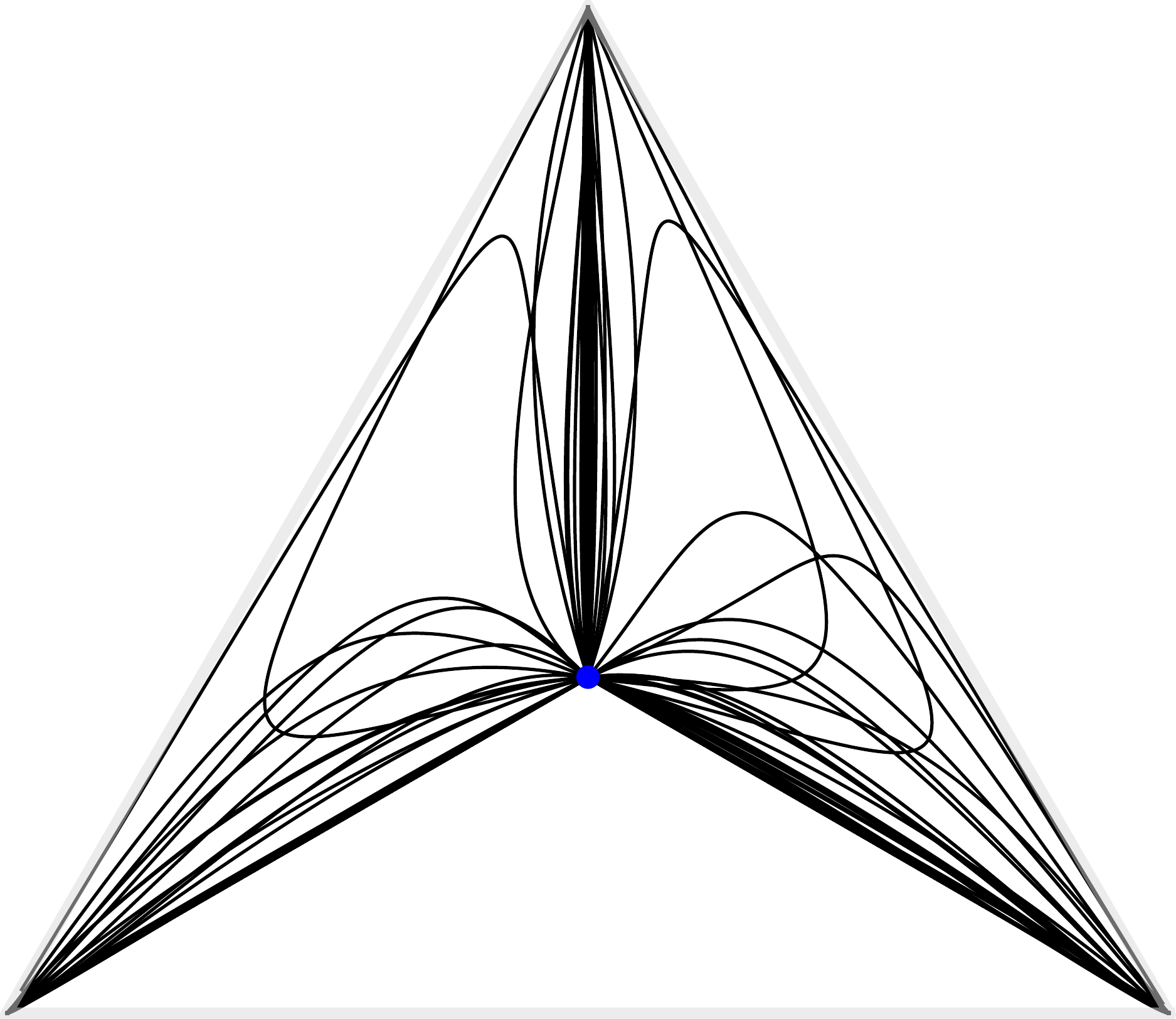}
\hspace{0.02\textwidth}
\includegraphics[width=0.275\textwidth]{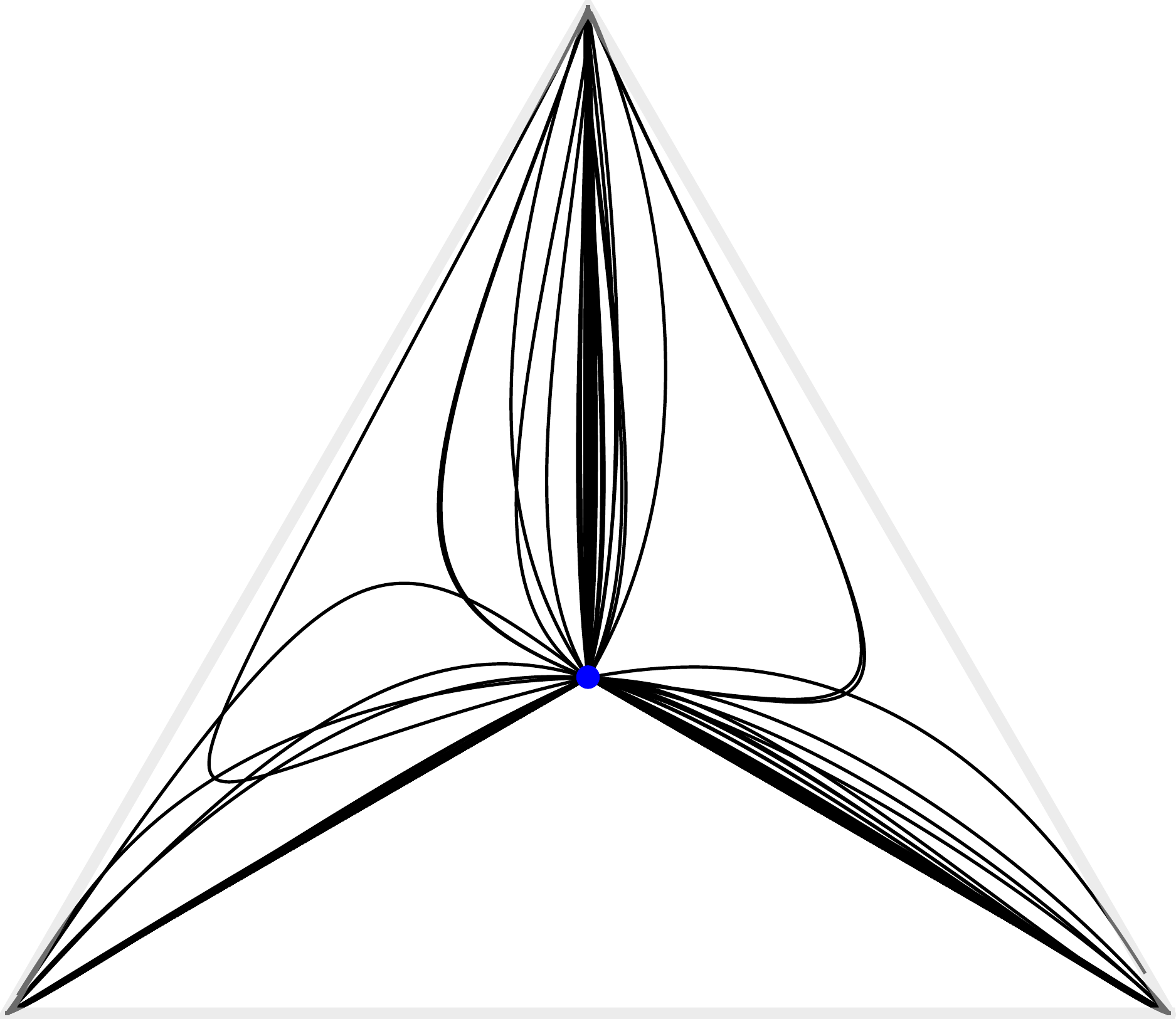}
\hspace{0.02\textwidth}
\includegraphics[width=0.275\textwidth]{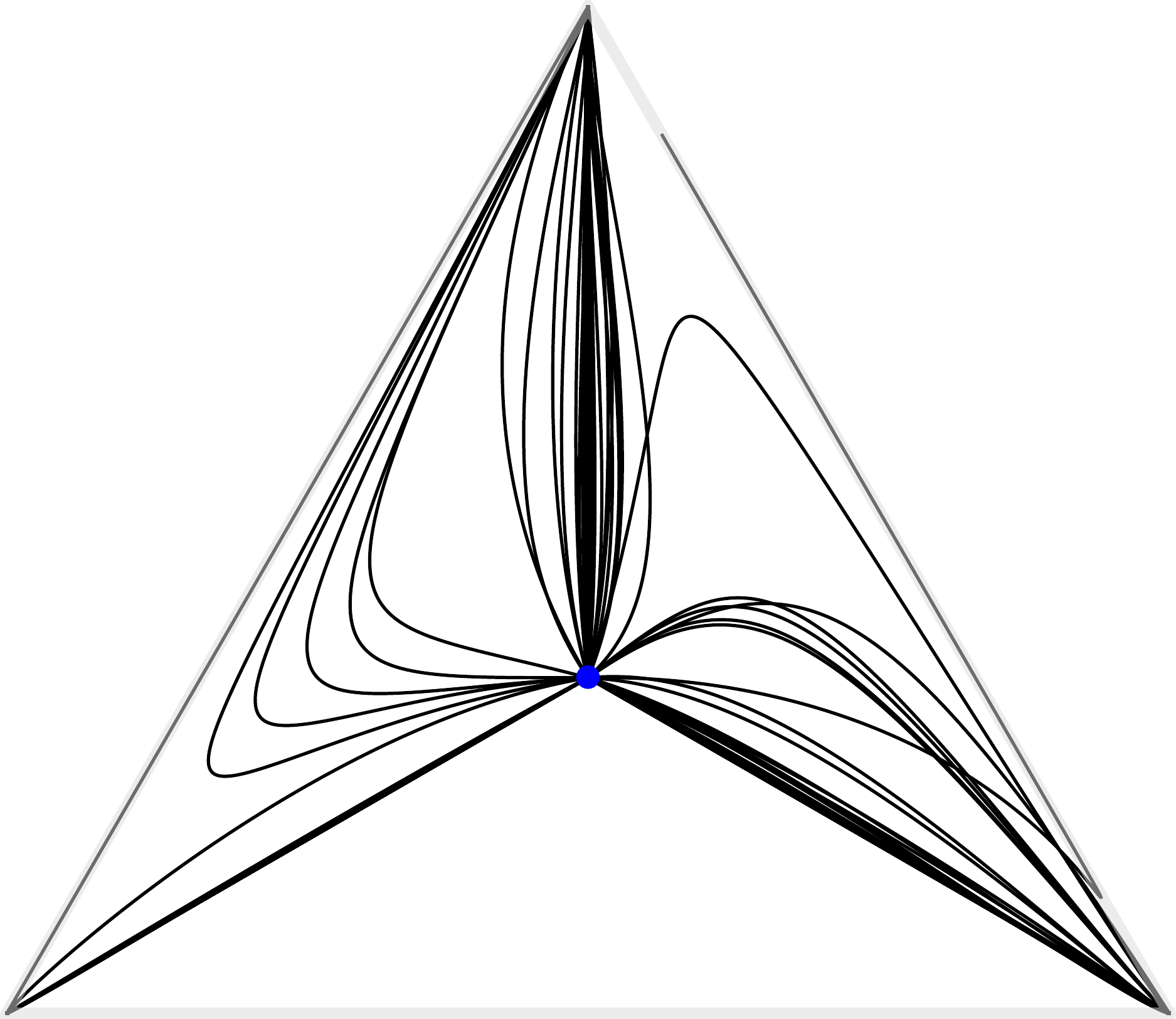}
}
\parbox{0.275\textwidth}{\centering\small $|\mc{N}_{i}|=3$}\hspace{0.02\textwidth}
\parbox{0.275\textwidth}{\centering\small $|\mc{N}_{i}|=5$}\hspace{0.02\textwidth}
\parbox{0.275\textwidth}{\centering\small $|\mc{N}_{i}|=9$}
\parbox{\textwidth}{\centering\small linear assignment flow}
\caption{
\textbf{Nonlinear vs. linear assignment flow.}
Comparison of the assignment flow \eqref{eq:assignment-flow} (top row) and the linear assignment flow \eqref{eq:ass-flow-linear} (bottom row) in terms of all $|I|$ solution curves $W_{i}(t),\, i \in I$, plotted in the $3$-label simplex. A major part of the trajectories approaches more or less directly a vertex, whereas another part changes the original direction due to regularization by geometric smoothing. Except for the cases with the minimal neighborhood $|\mc{N}_{i}|=3$, the similarity of both flows is apparent. This illustrates the reason for very small observed numbers of labeling errors, as listed and depicted by Table \ref{tab:linear-assignment-flow} and Figure \ref{fig:1D-labeling-error}.
}
\label{fig:simplex-plots}
\end{figure}

We first explain how the linearizations of the assignment flow were controlled.
According to Proposition \ref{prop:linear-ass-parametrization}, using the parametrization \eqref{eq:W-ass-flow-linear} and the linear ODE \eqref{eq:dot-V-ass-flow-linear}
is equivalent to the linear assignment flow \eqref{eq:ass-flow-linear}. Using again the parametrization \eqref{eq:W-ass-flow-linear} and repeating the proof of Prop.~\ref{prop:linear-ass-parametrization} shows that the full assignment flow \eqref{eq:assignment-flow} is locally governed by the \textit{nonlinear} ODE
\begin{equation}
\dot V = \Pi_{W_{0}}\big(S(\exp_{W_{0}}(V))\big),\qquad V(0)=0.
\end{equation}
Taking into account \eqref{eq:s0-S0-natural} and subtracting the right-hand side of the approximation \eqref{eq:dot-V-ass-flow-linear} from the above right-hand side gives
\begin{equation}\label{eq:S-exp-V-approximation}
\Pi_{W_{0}}\big(S(\exp_{W_{0}}(V))\big)
- \Pi_{W_{0}}(s_{0}+S_{0} V)
= \Pi_{W_{0}}\big(S(\exp_{W_{0}}(V)) - S(W_{0}) - dS_{W_{0}}(V)\big),
\end{equation}
which shows that this approximation deteriorates with increasingly large tangent vectors $V$. As a consequence, we first solved the linear flow \eqref{eq:ass-flow-linear} using \eqref{eq:linear-ass-parametrization} \textit{without} updating the point of linearization $W_{0}=\eins_{\mc{W}}$ and fixed after termination at $k_{\mrm{end}}$ (= number of required outer iterations) the constant
\begin{equation}
\|V\|_{\max} = \max_{i \in I}\|V_{i}^{(k_{\mrm{end}})}\|.
\end{equation}
Then we solved the linear assignment flow again and updated the linearization point $W_{0}$ in view of \eqref{eq:S-exp-V-approximation} whenever
\begin{equation}\label{eq:def-V-s-threshold}
\max_{i \in I}\|V_{i}^{(k)}\| > \frac{\|V\|_{\max}}{c},\qquad c \geq 1,
\end{equation}
using the parameter $c$ to control the number of linearizations: a single linearization and no linearization update if $c=1$ and an increasing number of updates for larger values of $c$. We updated components of the linearization point $W_{0,i}$ by $W_{i}(h),\, i \in I$ only when $\min_{j \in J} W_{ij}(h) > 0.01$, in order to keep linearization points inside the simplex, in view of the entries \eqref{eq:dSik-entries} of $dS_{W_{0}}$ normalized by components of $W_{0k}$.

After termination, the induced labelings were compared to those of the full assignment flow, and the number of wrongly assigned labels was taken as a quantitative measure for the approximation property of the linear assignment flow:
\begin{table}[h]
\centering
\begin{tabular}{c|ccccc}
$|\mc{N}_{i}|$ &
$c=1$ & $c=2$ & $c=3$ & $c=4$ & $c=5$ \\ \hline
3 &
1/3 & 4/3 & 7/3 & 10/1 & 13/0 \\
5 &
1/0 & \\
9 &
1/2 & 5/0
\end{tabular}
\caption{Approximation property of the linear assignment flow.
The entries $x/y$ specify the number $x$ of linearizations and the number $y$ of wrongly assigned labels (out of 192 assigned labels), depending on the neighborhood size $|\mc{N}_{i}|$ (strength of regularization) and the parameter $c$ specifying the tangent space threshold \eqref{eq:def-V-s-threshold}.
}
\label{tab:linear-assignment-flow}
\end{table}
Except for the minimal neighborhood size $|\mc{N}_{i}|=3$, a single linearization almost suffices to obtain a correct labeling. Overall, the maximal number of $3$ labeling errors (out of $192$) is very small, and these errors merely correspond to shifts by a single pixel position of the signal transition in the case $|\mc{N}_{i}|=9$ (see Figure \ref{fig:1D-labeling-error}). We conclude that for supervised labeling, the linear assignment flow \eqref{eq:ass-flow-linear} (which is nonlinear(!) -- cf.~Remark \ref{rem:ass-flow-linear}) indeed captures a major part of nonlinearity of the full assignment flow \eqref{eq:assignment-flow}. Figure \ref{fig:simplex-plots} illustrates the similarity of the two flows (and the dissimilarity in the case $|\mc{N}_{i}|=3$) in terms of all $|I|=192$ sequences $(W_{i}^{(k)}),\,i \in |I|$, plotted as piecewise linear trajectories.

We cannot assure, however, that this approximation property persists in more general cases (cf.~Remark \ref{rem:non-basic-cases}) whose study is beyond the scope of the present paper.

\begin{figure}
\centerline{
\includegraphics[width=0.2\textwidth]{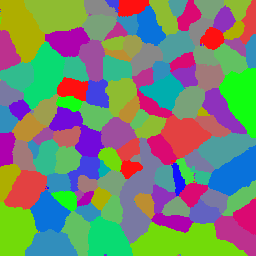}\hspace{0.05\textwidth}
\includegraphics[width=0.2\textwidth]{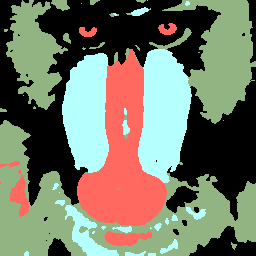}
\includegraphics[width=0.2\textwidth]{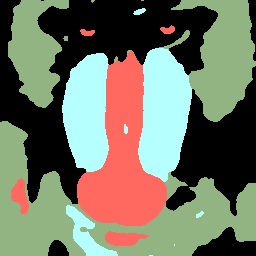}
}
\caption{
\textbf{Linear assignment flow.}
Labeling results for the two scenarios of Figure \ref{fig:linear-BE}, using the \textit{linear} assignment flow with a single linearization at the barycenter and the \textit{implicit} Euler scheme for numerical integration. Comparison with the labeling results of the \textit{nonlinear} assignment flow (Fig.~\ref{fig:embedded-RKMK-a}(a), Fig.~\ref{fig:embedded-RKMK-b} `ground truth') demonstrates a remarkable approximation property of the linear assignment flow.
}
\label{fig:linear-BE}
\end{figure}
\vspace{0.5cm}
We now turn to the scenarios shown by \ref{fig:labeling-scenarios}. Figure \ref{fig:linear-BE} shows the results obtained using the implicit Euler scheme and the \textit{same} parameter settings that were used to integrate the nonlinear flow, to obtain the ground truth flows and results depicted by  Figure \ref{fig:embedded-RKMK-a}(a) and Figure \ref{fig:embedded-RKMK-b}, left-most panel, respectively. Comparing the labelings returned by the linear and nonlinear assignment flow, respectively, confirms the discussion of the 1D experiments detailed above: The results agree except for a very small subset of pixels close to signal transitions which are immaterial for subsequent image interpretation.

The results shown by Figure \ref{fig:linear-BE} served as ground truth for studying the \textit{explicit} numerical schemes of Sections \ref{sec:exp-RK-adaptive} and \ref{sec:exp-exponential-integrator} for integrating the linear assignment flow.

\begin{figure}
\centerline{
\includegraphics[width=0.2\textwidth]{Figures/31-colors-LinearBE-Labels}\hfill
\includegraphics[width=0.35\textwidth]{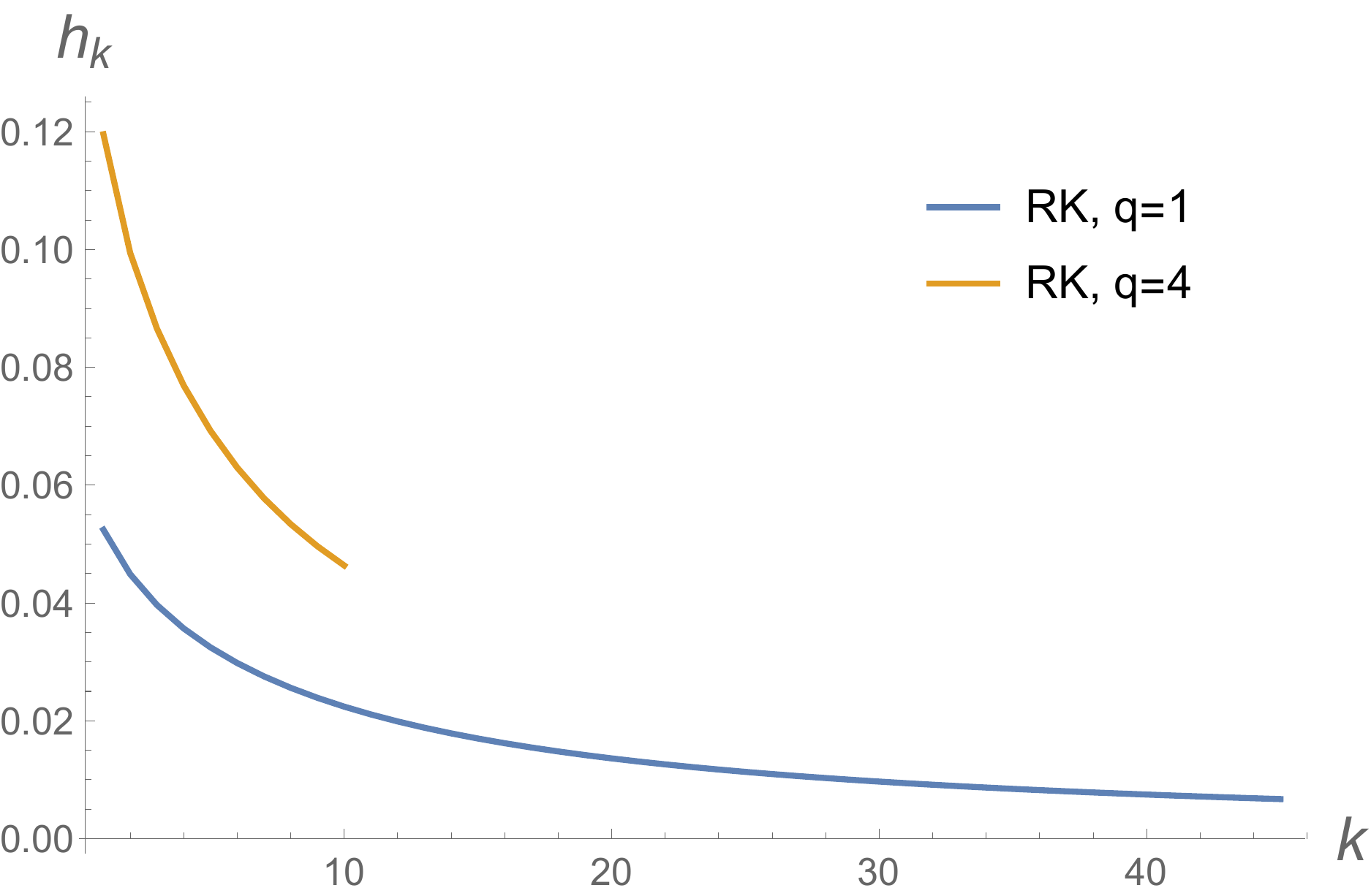}\hfill
\includegraphics[width=0.2\textwidth]{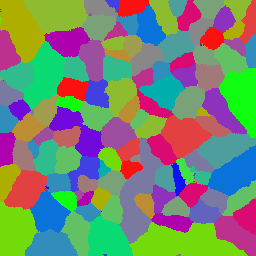}\hfill
\includegraphics[width=0.2\textwidth]{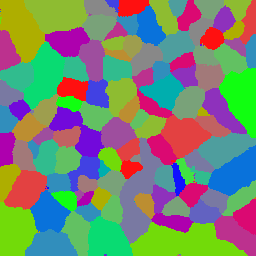}
}
\centerline{
\parbox{0.2\textwidth}{\centering\small (a)} \hfill
\parbox{0.35\textwidth}{\centering\small (b)} \hfill
\parbox{0.2\textwidth}{\centering\small (c)} \hfill
\parbox{0.2\textwidth}{\centering\small (d)}
}
\caption{
\textbf{Linear assignment flow, adaptive RK-schemes.} Results of the linear assignment flow \eqref{eq:ass-flow-linear} based on the parametrization \eqref{eq:linear-ass-parametrization}, the RK schemes \eqref{eq:RK-Euclidean} of order $q=1$ (FE) and $q=4$ (RK4), and adaptive step size selection based on the local error estimate \eqref{eq:local-error-bound}. The labeling results (c), (d) for $q=1$, $q=4$ are almost identical to ground truth (a) from Figure \ref{fig:linear-BE}.
}
\label{fig:31-colors-adaptive-RK}
\end{figure}
\begin{figure}
\begin{minipage}{0.64\textwidth}
\centerline{
\includegraphics[width=0.32\textwidth]{Figures/Mandrill-N1-LinearBE-Labels}\hfill
\includegraphics[width=0.32\textwidth]{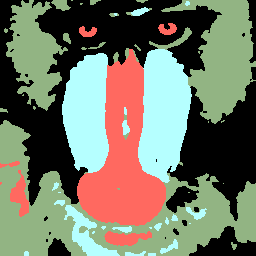}\hfill
\includegraphics[width=0.32\textwidth]{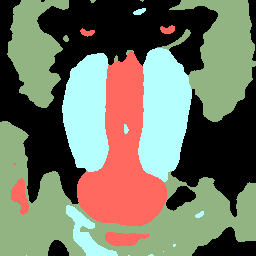}
}
\vspace{0.01\textwidth}
\centerline{
\includegraphics[width=0.32\textwidth]{Figures/Mandrill-N2-LinearBE-Labels}\hfill
\includegraphics[width=0.32\textwidth]{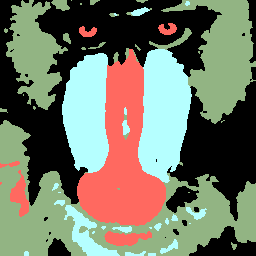}\hfill
\includegraphics[width=0.32\textwidth]{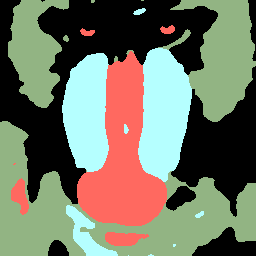}
}
\centerline{
\parbox{0.32\textwidth}{\centering\small ground truth}
\hfill
\parbox{0.32\textwidth}{\centering\small RK: $|\mc{N}_{i}|=3 \times 3$}
\hfill
\parbox{0.32\textwidth}{\centering\small RK: $|\mc{N}_{i}|=5 \times 5$}
}
\end{minipage}
\begin{minipage}{0.35\textwidth}
\includegraphics[width=\textwidth]{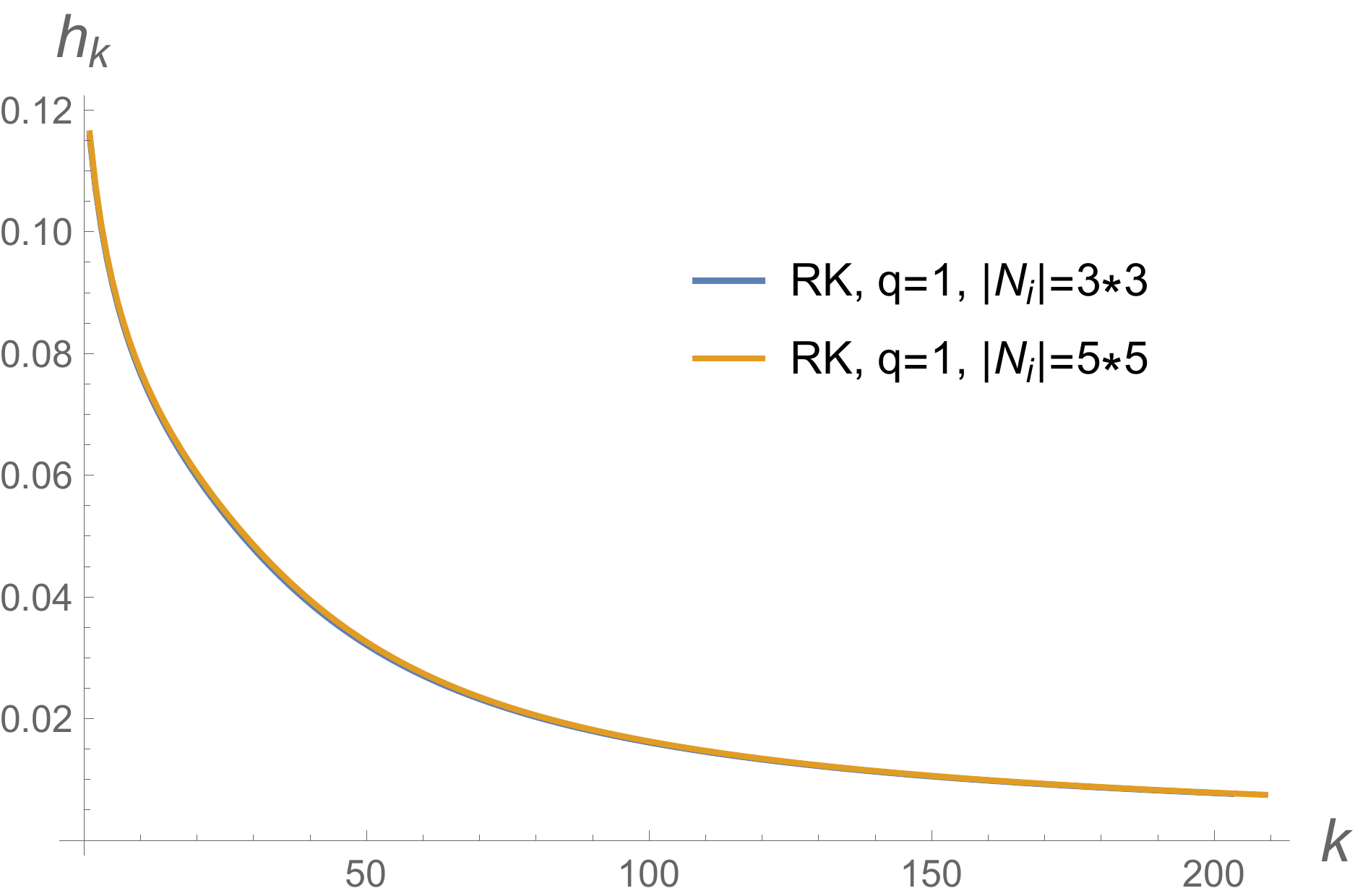}
\includegraphics[width=\textwidth]{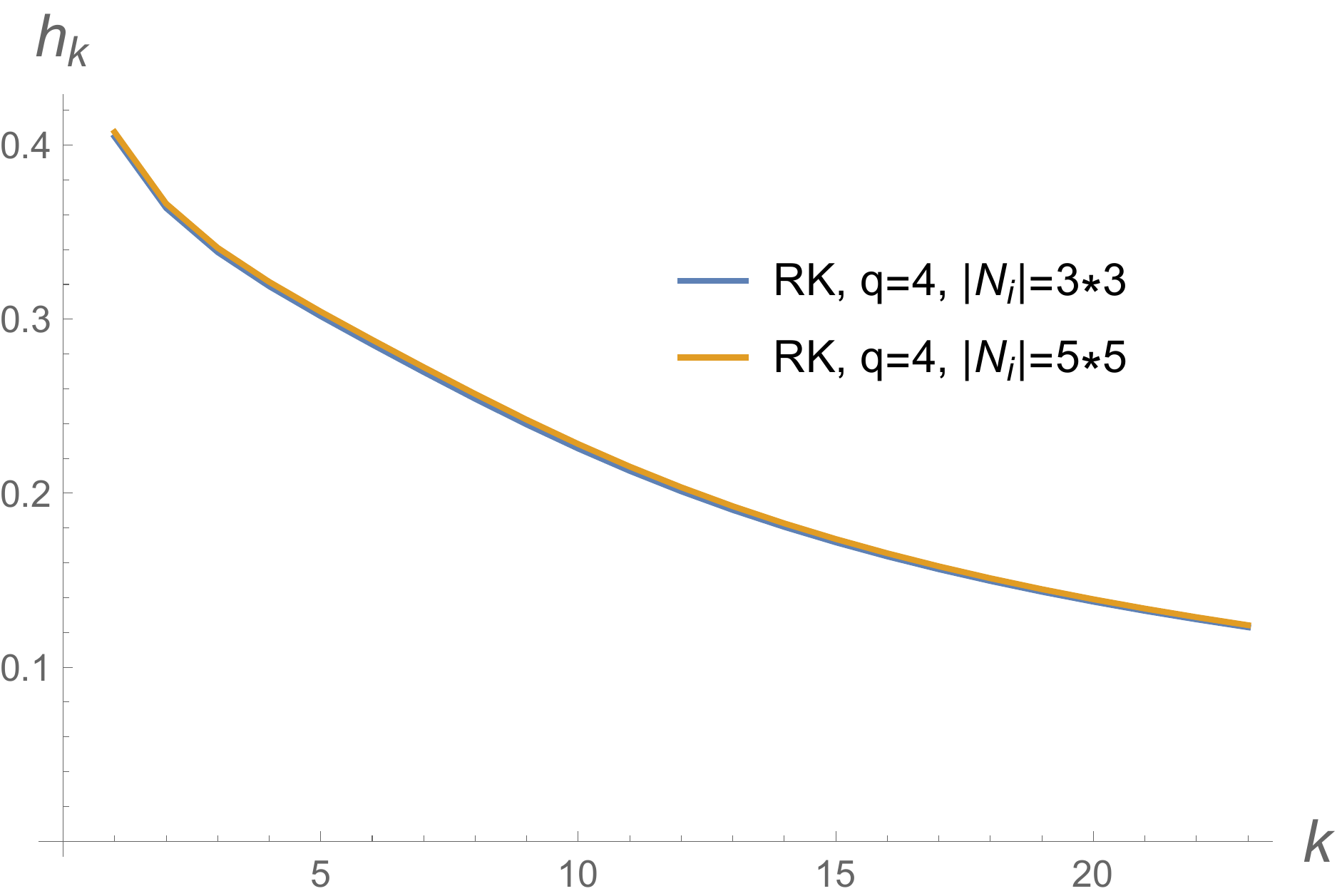}
\end{minipage}
\caption{
\textbf{Linear assignment flow, adaptive RK-schemes.}
Results of the linear assignment flow \eqref{eq:ass-flow-linear} based on the parametrization \eqref{eq:linear-ass-parametrization}, the RK schemes \eqref{eq:RK-Euclidean} of order $q=1$ (FE) and $q=4$ (RK4), and adaptive step size selection based on the local error estimate \eqref{eq:local-error-bound}. For fixed $q$, increasing the neighborhood size $|\mc{N}_{i}|$ has almost no effect: the step size sequences agree up to the second digit. The labeling results (c), (d) for $q=1$, $q=4$ are almost identical to ground truth from Figure \ref{fig:linear-BE}.
}
\label{fig:mandrill-adaptive-RK}
\end{figure}

\subsubsection{Adaptive RK Schemes}
\label{sec:exp-RK-adaptive}

We evaluated the adaptive RK schemes (FE) of order $q=1$ and (RK4) of order $q=4$, due to \eqref{eq:RK-Euclidean}, supposed to integrate the linear ODE \eqref{eq:dot-V-ass-flow-linear}, after rearranging the polynomials of \eqref{eq:RK-Euclidean} in Horner form.

Figures \ref{fig:31-colors-adaptive-RK} and \ref{fig:mandrill-adaptive-RK} show the results for the \textit{linear} assignment flow based on a single linearization at the barycenter, using the results shown by Figure \ref{fig:linear-BE} as ground truth.  The step sizes $h_{k}$ were computed at each iteration $k$ using the local error estimate \eqref{eq:local-error-bound} such that $\frac{1}{|I|^{1/2}} \|V(t_{k+1})-V^{(k+1)}\| \leq \tau = 0.01$, that is on average $\|V_{i}(t_{k+1})-V_{i}^{(k+1)}\| \leq \tau$ for all pixels $i \in I$. The spectral norm $\|A\|$ was computed beforehand using the basic power iteration.

As explained above when the criterion \eqref{eq:hk-RK-linear} was introduced, step sizes must decrease due to the increasing norms $\|V^{(k)}\|$, in order to keep bounded the local integration error. Furthermore, in agreement with Figure \ref{fig:RK-linear-boundFactor}, raising the order $q$ of the integration scheme leads to significantly larger step sizes and hence to smaller total numbers of iterations, at the cost of more expensive iterative steps. Yet, roughly taking these additional costs into account by multiplying the total iteration numbers for $q=4$ by $4$, indicates that raising the order $q$ reduces the overall costs. In this respect our findings for the linear assignment flow differ from the corresponding findings for the nonlinear assignment flow and the embedded RKMK schemes, discussed in Section \ref{sec:exp-nonlinear}.

\subsubsection{Exponential Integrator}
\label{sec:exp-exponential-integrator}

For integrating the linearized assignment flow with exponential integrators,
we consider equation~\eqref{eq:V-by-phi1} and the Krylov space approximation~\eqref{eq:vphi-1-approximation}
\begin{equation}\label{eq:experiments-exp-int}
    V(T) = T \varphi_1 \big(T A\big) a \approx T \|a\| V_m \varphi_1(T H_m) e_1.
\end{equation}
As the evaluation of $\varphi_1(T H_m) e_1$ is explained in Section \ref{sec:Exponential-Integrator},
we only discuss here the choice of the parameters $m$ and $T$.

The dimension $m$ of the Krylov subspace controls the quality of the approximation, where larger values theoretically
lead to a better approximation.
In our experiments, rather small numbers, like $m=5$, turned out to suffice to produce labelings very close to the ground truth labelings, that were generated by
the implicit Euler method -- see Figures~\ref{fig:expintChoiceOfm} and \ref{fig:expintMandrill}.
As the runtime of the algorithm increases with growing $m$, this parameter should
not be chosen too large.

\begin{figure}
    \centering
    \begin{subfigure}[b]{0.48\textwidth}
        \includegraphics[width=\textwidth]{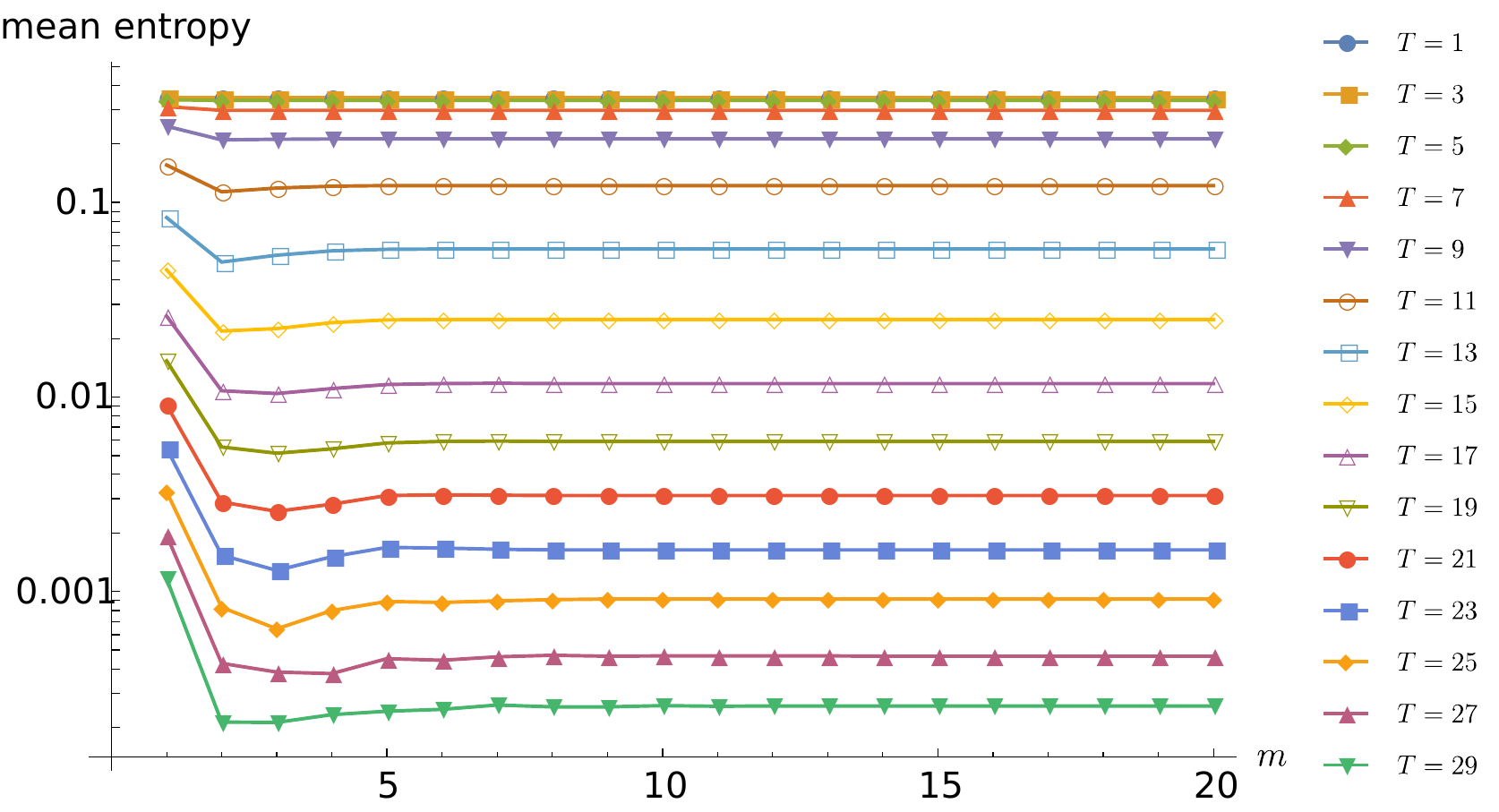}
        \caption{}
    \end{subfigure}
    \begin{subfigure}[b]{0.48\textwidth}
        \includegraphics[width=\textwidth]{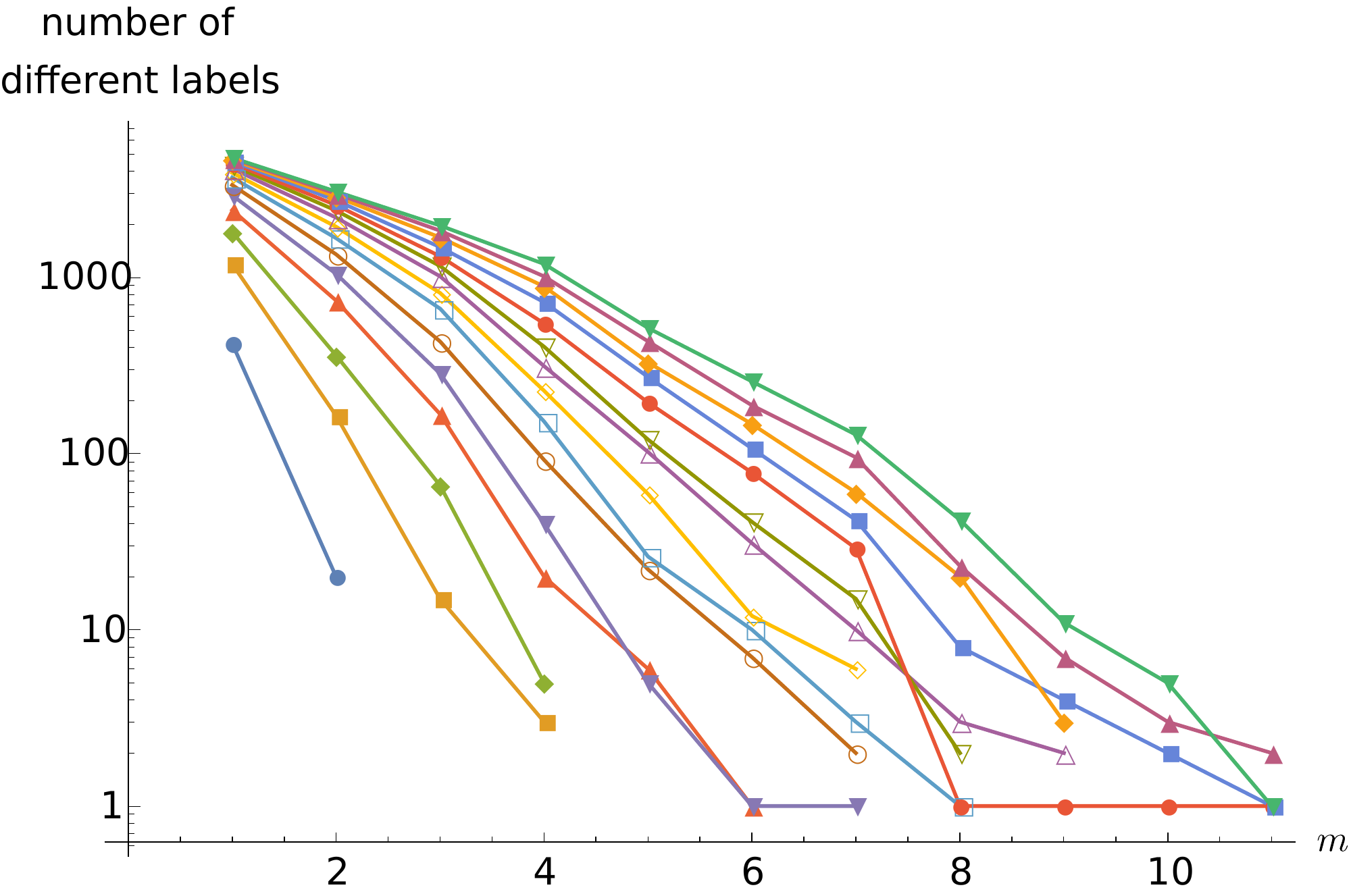}
        \caption{}
    \end{subfigure}
    \caption{\textbf{Linear assignment flow, exponential integrators, influence of the Krylov subspace dimension $m$ and the point of evaluation $T$.} The results correspond to the labeling problem shown by Figure \ref{fig:expintMandrill}. They demonstrate that $m$ can be chosen quite small.
        Figure (a) shows that the mean entropy decreases with increasing $T$ but does not decrease for \textit{fixed} $T$ and $m > 3$.
         Figure (b) shows for $T=1, \dotsc, 29$ (same color code as (a)) and for each $m$ the number of labels that change when the dimension of the Krylov subspace is increased to $m+1$. All curves decrease with increasing $m$, and curves that reach $0$ label changes just discontinue due to the logarithmic scale.
        The plot shows that independent of $T$ (i.e.~for any $T$) there is an $m$ such that increasing the dimension of the Krylov space
        does not change the labeling at all.
    }\label{fig:expintChoiceOfm}
\end{figure}

Scaling $A$ and $a$ in \eqref{eq:V-t-Duhamel} affects the vector field defining the linear ODE \eqref{eq:dot-V-ass-flow-linear}.
Hence, fixing any time point $T$ depends on this scaling factor, too.
As a consequence, since $A$ and $a$ depend on the problem data \eqref{eq:V-t-Duhamel}, the choice of $T$ is problem dependent.
On the other hand, the discussion following the proof of Theorem \ref{thm:local-error} showed that $\|V(t)\|$ increases with $t$, and $T$ merely has to chosen large enough such that $W(T)$ defined by \eqref{eq:W-ass-flow-linear} satisfies the termination criterion \eqref{eq:termination-criterion} -- see \eqref{eq:termination-rough} for a rough estimate.
Choosing $T$ overly large will cause  numerical underflow and overflow issues, however.

Almost all runtime is consumed by the Arnoldi iteration producing the subspace basis $V_{1},\dotsc,V_{m}$. Due to the small dimension $m$, the total runtime is very short, and time required for the subsequent evaluation of the right-hand side of \eqref{eq:experiments-exp-int} is neclectable.

\begin{figure}
    \centering
    \begin{subfigure}[b]{0.3\textwidth}
        \includegraphics[width=\textwidth]{Figures/mandrill-rho0_5-Ni-1}
        \caption{}
    \end{subfigure}
    \quad
    \begin{subfigure}[b]{0.3\textwidth}
        \includegraphics[width=\textwidth]{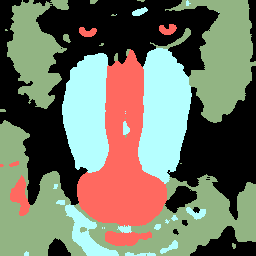}
        \caption{}
    \end{subfigure}
    \quad
    \begin{subfigure}[b]{0.3\textwidth}
        \includegraphics[width=\textwidth]{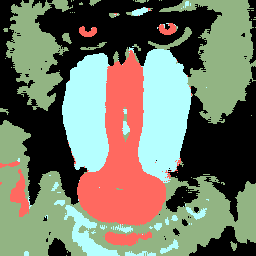}
        \caption{}
    \end{subfigure}
    \caption{\textbf{Linear assignment flow, exponential integrators, sample labelings.}
        Figure (a) shows the ground truth labeling as generated by the implicit Euler method.
        Figure (b) displays the labeling generated by the exponential integrator using the Krylov subspace dimension $m=5$,
        which is very close to the ground truth labeling. As demonstrated by Figure \ref{fig:expintChoiceOfm}, $m$ cannot be chosen too small, however, since labelings then start to deteriorate rapidly. This is illustrated by Figure (c) which shows the labeling for $m=3$. Comparison with (b) shows that dimensions $4$ and $5$ `contain' small-scale details of the correct labeling induced by the linear assignment flow.
    }\label{fig:expintMandrill}
\end{figure}

\section{Conclusion}
\label{sec:Conclusion}

We investigated numerical methods for image labeling by integrating the large system of nonlinear ODEs defining the assignment flow \eqref{eq:assignment-flow}, which evolves on the assignment manifold. All methods exactly respect the underlying geometry. 
Specifically, we adapted RKMK methods and showed that \textit{embedded} RKMK-methods work very well for automatically adjusting the step size, at neglectable additional costs. Raising the order enables leads to larger step sizes, which is compensated by the higher computational costs per iteration, however. In either case, each iteration only involves convolution like operations over local neighborhoods together with pixelwise nonlinear functions evaluations.

We derived and introduced the \textit{linear} assignment flow, a nonlinear approximation of the (full) assignment flow that is governed by a large linear system of ODEs on the tangent space. Experiments showed that the approximation is remarkably close, as measured by the number of different label assignments.

We investigated two further families of numerical schemes for integrating the linear assignment flow: established RK schemes with adaptive step size selection based on a local integration error estimate, and exponential integrators for approximately evaluating Duhamel's integral using a Krylov subspace. In the former case, higher-order schemes really pay, unlike for the RKMK schemes and the full assignment flow, as mentioned above. Choosing the classical RK scheme with $q=4$, for example, few dozens of iterations suffice to reach the termination criterion, with high potential for parallel implementation. The exponential integrators, on the other hand, directly approximate the integral determining $V(T)$ and in this sense are non-iterative. Here, a Krylov subspace basis of low dimension merely has to be computed, using a standard iterative method. Even though this method is differs mathematically from the RK schemes, it has potential for real-time implementation as well.

All methods provide a sound basis for more advanced image analysis tasks that involve labeling by evaluating the assignment flow as a subroutine. Accordingly, our future work concerns an extension of the unsupervised labeling approach \cite{Zern:2018ac}, where label dictionaries are directly learned from data through label assignment. Furthermore, methods under investigation for learning to adapt regularization parameters of the assignment flow to specific image classes, require advanced discretization and numerical methods based on the results reported in the present paper.

\vspace{0.5cm}

\noindent
\textbf{Acknowledgements.} This work was supported by the German Research Foundation (DFG), grant GRK 1653.
\bibliographystyle{amsalpha}
\bibliography{ExponentialIntegrators}
\end{document}